\documentclass[preprint]{article}
\usepackage{graphicx}
\usepackage{algorithmic}
\usepackage{algorithm}
\usepackage{amsmath}
\usepackage{amssymb}
\usepackage{caption}
\usepackage{subcaption}
\usepackage{amsfonts}
\usepackage{amsthm}
\usepackage{fullpage}
\usepackage{url}
\usepackage{color}
\usepackage{authblk} 

\def\N{{\mathbb{N}}}
\def\R{{\mathbb{R}}}

\def\dt{{\textrm{d}t}}
\def\dx{{\textrm{d}x}}
\def\dy{{\textrm{d}y}}
\providecommand{\keywords}[1]{{\textit{Key words: }} #1}

\newtheorem{theorem}{Theorem}

\title{Adaptive Local Iterative Filtering for Signal Decomposition and Instantaneous Frequency analysis}

\author[1]{Antonio Cicone\thanks{antonio.cicone@univaq.it}}
\author[2]{Jingfang Liu\thanks{jingfang.liu1@gmail.com}}
\author[2]{Haomin Zhou\thanks{hmzhou@math.gatech.edu}}
\affil[1]{Marie Curie fellow of the Istituto Nazionale di Alta Matematica, DISIM, Universit\`a degli Studi dell'Aquila, via Vetoio n.1, 67100 L'Aquila, Italy}
\affil[2]{School of Mathematics, Georgia Institute of Technology, 686 Cherry St, Atlanta, GA 30332, USA}

\begin{document}
\maketitle

\begin{abstract}{Time--frequency analysis for non--linear and non--stationary signals is extraordinarily challenging.
To capture features in these signals, it is necessary for the analysis methods  to be local, adaptive
and stable. In recent years, decomposition based analysis methods, such as the empirical mode
decomposition (EMD) technique pioneered by Huang et al.,  were developed by different research
groups. These methods decompose a signal into a finite number of components
on which the time--frequency analysis can be applied more effectively.

In this paper we consider the Iterative Filtering (IF) approach as an alternative to EMD. We provide
sufficient conditions on the filters that ensure the convergence of IF applied to any $L^2$ signal. Then
we propose a new technique, the Adaptive Local Iterative Filtering (ALIF) method, which uses
the IF strategy together with an adaptive and data driven filter length selection to achieve the decomposition.
Furthermore we design smooth filters with compact support from solutions of
Fokker--Planck equations (FP filters) that can be used within both IF and ALIF methods.
These filters fulfill the derived sufficient conditions for the convergence of the IF algorithm.
Numerical examples are given to demonstrate the performance and stability of IF
and ALIF techniques with FP filters. In addition, in order to have a complete and truly local
analysis toolbox for non--linear and non--stationary signals, we propose new definitions for
the instantaneous frequency and phase which depend exclusively on local properties of a signal.}\end{abstract}


\section{Introduction}

Data and signal analysis has become increasingly important these days. Decomposing signals and finding features of  data is quite challenging especially when the data is non--stationary and it is generated by a non--linear system.  Time--frequency analysis have been substantially studied in the past, we refer to  \cite{cohen1995time} and \cite{grochenig2000foundations} for more information on this rich subject. Traditionally, Fourier spectral analysis as well as wavelet transforms have been commonly used. Both approaches are effective and easy to implement, however there are some limitations. Fourier transform works well when the data is periodic or stationary and the associated systems are linear, it can not deal with non--stationary signals or data from non--linear systems. The wavelet transform is also a linear analysis tool. Both approaches use predetermined bases and they are not designed to be data--adaptive. Hence these techniques often can not achieve desirable results for non--linear and non--stationary signals.

In the last decade, several decomposition techniques have been proposed to analyze non--linear and non--stationary signals. All these methods share the same approach: first they decompose a signal into simpler components and then they apply a time--frequency analysis to each component separately. The signal decomposition can be achieved in two ways: by iteration or by optimization.

The first iterative algorithm of this kind, the empirical mode decomposition (EMD), was introduced by Huang et al. \cite{huang1998empirical} in 1998.  This method aims to iteratively decompose a signal into a finite sequence of intrinsic mode functions (IMFs) whose instantaneous frequencies are well behaved. We will come back to instantaneous frequency later in this paper, let us instead describe the iterative structure of EMD which is called the Sifting Process.

Let $\mathcal{L}$ be an operator getting the moving average of a signal $f(x)$ and $\mathcal{S}$ be an operator capturing the fluctuation part $\mathcal{S}(f)(x)=f(x)-\mathcal{L}(f)(x)$. Then the first IMF produced by the sifting process is
\begin{equation}\label{equ4}
I_1(x) = \lim_{n\to\infty} \mathcal{S}_{1,n}\left(f_{n}\right)(x)
\end{equation}
where $f_{n}(x)=\mathcal{S}_{1,n-1}(f_{n-1})(x)$ and $f_{1}(x)=f(x)$.
Here the limit is reached so that applying $\mathcal{S}$ one more time does not change the signal.

The subsequent IMFs are obtained one after another by
\begin{equation}\label{equ5}
I_k(x) = \lim_{n\to\infty} \mathcal{S}_{k,n}(r_{n})(x)
\end{equation}
where $r_{n}(x)=\mathcal{S}_{k,n-1}(r_{n-1})(x)$ and $r_{1}(x)=r(x)$ which is the remainder $f(x)-I_1(x)-\ldots-I_{k-1}(x)$. The sifting process stops when $r(x)=f(x)-I_1(x)-I_2(x)-\ldots-I_m(x)$ becomes a trend signal, which means it has at most one local maximum or minimum.  So the decomposition of $f(x)$ is
\begin{equation}\label{equ6}
f(x)=   \sum_{j=1}^m I_j(x) +  r(x)
\end{equation}

In this iterative process the moving average $\mathcal{L}(f)(x)$ is given by the mean function of the upper envelope and the lower envelope, which are given by cubic splines connecting local maxima and local minima of $f(x)$ respectively.
However, this method is not stable under perturbations
since cubic splines are used repeatedly in the iteration.
To overcome this issue, Huang et al. developed  the Ensemble Empirical Mode Decomposition (EEMD) \cite{wu2009ensemble} where the IMFs are taken as the mean of many different trials. In each trial
a random perturbation is artificially added to the original signal. More details on EMD method and its analysis can be found, for instance, in \cite{huang2005hilbert} and \cite{huang2003confidence, rilling2003empirical, el2010analysis, sharpley2006analysis, rilling2008one, rilling2009sampling, feldman2006time}.

Another iterative decomposition technique is the Iterative Filtering (IF) method which is inspired by EMD \cite{lin2009iterative}. It uses the same algorithm framework as the original EMD, but the moving average of a signal $f(x), x\in \mathbb{R}$, is derived by the convolution of $f(x)$ with low pass filters, for example the double average filter $a(t)$ given by
 \begin{equation}\label{equ9}
 a(t)=\frac{l+1-|t|}{(l+1)^2}, \quad t\in[-l,\, l]
 \end{equation}
IF is stable under perturbation and the convergence is guaranteed for periodic and $l^\infty$ signals using uniform filters \cite{lin2009iterative,wang2013convergence}. However, the convergence for general signals with uniform and non--uniform filters hasn't been explored yet.  Recently, Wang et al. in \cite{wang2012iterative} and \cite{wang2012mode} developed the mode decomposition evolution equations which can achieve similar decompositions to the IF algorithm through some high order partial differential equations.

A different way to decompose a signal is by optimization. Using the multicomponent amplitude modulation and frequency modulation (AM--FM) representation, which has been studied for example in \cite{loughlin1997comments} and \cite{wei1998instantaneous}, Hou et al. developed an adaptive data analysis via sparse time--frequency representation in \cite{hou2009variant} and \cite{thomas2011adaptive};   Daubechies et al.  proposed synchrosqueezed wavelet transforms in \cite{daubechies2011synchrosqueezed} and Gilles the Empirical wavelet transform \cite{gilles2013empirical} as an EMD--like tool. These methods assume that each IMF can be written as an AM--FM
or wavelet function, a signal $f(x)$ is decomposed into a group of IMFs by seeking  the minimizer of some functional of $f(x)$. We refer  interested readers to \cite{hou2009variant, thomas2011adaptive, daubechies2011synchrosqueezed, meignen2007new, pustelnik2012multicomponent, selesnick2011resonance, dragomiretskiy2014variational} and \cite{wu2011one}
for more details on this kind of techniques.

In this paper we review the IF technique and present a new one,
called Adaptive Local Iterative Filtering (ALIF) algorithm,
which generalizes the IF method to non--uniform filters applied to general oscillatory signals.

There are two main aspects in ALIF that are different from the existing IF algorithm. One is that
we use a Fokker--Planck equation, a second order partial differential equation (PDE), to
construct smooth low pass filters which have compact support that we call FP filters. The other is that we
adapt the filter length point by point according to the signal itself.

To effectively handle non--linear and non--stationary signals, it is highly desirable to
use filters with compact  support. In fact filters with long support
may mix features that are far apart in a signal which could be troublesome, especially
for signals with transient information. However, the compact support low pass filters,
such as the double average filters, used in the existing IF algorithms are not
smooth enough. They may create artificial oscillations in subsequent IMFs, due to their
non--smoothness. This motivates us to design filters from the solution of
Fokker--Planck equations. These newly designed FP filters are compactly supported, infinitely differentiable and vanishing to zero smoothly at both ends. Such features ensure that no artificial oscillations are introduced during the iterative filtering process.

More importantly, to capture the non--stationary changes in the frequency and
amplitude of a signal, the length of the filters must be adapted accordingly.
However not all compact support low pass filters and adaptive strategies can lead to convergent decompositions.
The adaptive strategy must be carefully designed.
In this paper we propose a strategy that is completely data driven and we show a few numerical results produced using this method.

Moreover, we propose alternative definitions for the instantaneous frequency and phase. In the existing
instantaneous frequency and phase analysis algorithms, Hilbert transform is used to build analytical signals. However Hilbert transform is a global operator, which is not ideal to handle signals
with transient features. To localize the analysis, we define the instantaneous frequency and phase
of an IMF, obtained by the IF and ALIF algorithms, as the rotation speed and the rotation angle calculated by the normalized IMF and its derivative.
We show that such definition for the instantaneous frequency can better
capture the frequency changes in non--linear signals.

The rest of the paper is organized as follows: In Section \ref{sec:IF} we review the IF method and we study its convergence when used to decompose a general non--periodic and non--stationary signal. In Section \ref{sec:ALIF}, we present the ALIF algorithm and we discuss about its convergence.
Section \ref{sec:filters} is devoted to develop the so called FP filters that can be used in both IF and ALIF techniques.
In Section \ref{sec:InstFreq}  we give new definitions of instantaneous frequency and phase as well as a method to compute them. Finally in Section \ref{sec:Experiments}  we show numerical results obtained applying IF and ALIF methods on different kinds of signals.

\section{Iterative Filtering Algorithm}\label{sec:IF}

The Iterative Filtering (IF) \cite{lin2009iterative} is, as the name suggests, an iterative technique which allows to decompose a given signal into a finite number of simple components called Intrinsic Mode Functions (IMFs).

As defined in \cite{huang1998empirical}, an IMF is a function fulfilling two properties:
the number of extrema and the number of zero crossings must either equal or differ at most by one;
considering an upper envelope connecting all the local maxima and a lower envelope connecting all the local minima of the function, their mean has to be zero at any point.

Given a signal $f(x), x\in\R$, let $\mathcal{L}$ be an operator such that $\mathcal{L}(f)$ is a moving average of the signal $f(x)$. Considering a low pass filter like, for instance, the double average filter
$a(t)$,
the moving average of the signal $f(x)$ is given by the convolution
$$\mathcal{L}(f)(x)=\int_{-l}^l f(x+t)a(t) \dt.$$
In the IF method the operator $\mathcal{L}(f)$ is always given by the convolution of the signal $f$ and some filter function $w$.

If we define $f_1=f$ and the operator $\mathcal{S}_{1,n}(f_n)= f_n-\mathcal{L}_{1,n}(f_n)=f_{n+1}$
which captures the fluctuation part of $f_n$, then the first IMF is given by $I_1=\lim_{n\rightarrow\infty} \mathcal{S}_{1,n}(f_n)$, where $\mathcal{L}_{1,n}$ depends on the mask length $l_n$, which is the length of the filter at step $n$. Similarly, applying the operators $\mathcal{S}$ to the remainder signal $f-I_1$, we obtain $I_2$, the second IMF. Iterating the process we get the $k$-th IMF as $I_k=\lim_{n\rightarrow\infty} \mathcal{S}_{k,n}\left(r_n\right)=r_{n+1}$, where $r_1=f-I_1-\ldots -I_{k-1}$.
The IF method stops when $r=f-I_1-\ldots -I_{m},\, m\in \N$, becomes a trend signal, which means the remainder $r$ has at most one local maximum or minimum, and hence the signal is decomposed into \eqref{equ6}. 

The IF algorithm contains two nested loops: an Inner Loop, to compute each single IMF, and an Outer Loop, to derive all the IMFs.

\begin{algorithm}
\caption*{\textbf{IF Algorithm} IMF = IF$(f)$}
\begin{algorithmic}
\STATE IMF = $\left\{\right\}$
\WHILE{the number of extrema of $f$ $\geq 2$}
      \STATE $f_1 = f$
      \WHILE{the stopping criterion is not satisfied}
                  \STATE  compute the filter length $l_n$ for $f_{n}$
                  \STATE  $f_{n+1}(x) = f_{n}(x) -\int_{-l_n}^{l_n} f_n(x+t)w_n(t)\dt$
                  \STATE  $n = n+1$
      \ENDWHILE
      \STATE IMF = IMF$\,\cup\,  \{ f_{n}\}$
      \STATE $f=f-f_{n}$
\ENDWHILE
\STATE IMF = IMF$\,\cup\,  \{ f\}$
\end{algorithmic}
\end{algorithm}

For the computation of the mask length $l_n$, following \cite{lin2009iterative}, we can compute it as
\begin{equation}\label{eq:Unif_Mask_length}
l_n:=2\left\lfloor\chi \frac{N}{k}\right\rfloor
\end{equation}
where $N$ is the total number of sample points of a signal $f_n(x)$, $k$ is the number of its extreme points, $\chi$ is a parameter usually fixed around 1.6, and $\left\lfloor \cdot \right\rfloor$ rounds a positive number to the nearest integer closer to zero.

We observe that even though the IF algorithm allows for the updating of the mask length $l_n$ at each step of the inner loop, in the implemented algorithm for every inner loop we do compute the mask length only in the first step and then use the same value throughout all the other steps. The reason for doing so is to ensure that each IMF produced by the method does contain a well defined set of instantaneous frequencies.

For this reason the operators $\mathcal{S}$ and $\mathcal{L}$ are going to be independent on the step number $n$. Hence, given the signal $f$, the first IMF is simply given by $I_1=\lim_{n\rightarrow\infty} \mathcal{S}^n(f)$, where $\mathcal{S}(f)= f-\mathcal{L}(f)$ and $\mathcal{L}(f)(x)=\int_{-l}^l f(x+t)w(t) \dt$, with $l$ the mask length computed in the first step of the inner loop and $w(t)$ some suitable filter function.
In the implemented algorithm, instead of letting $n$ to go to infinity, we stop the inner loop based on some stopping criterion. We discuss about this in more details in Section \ref{sec:ALIF}.

The convergence for the inner loop is guaranteed for periodic signals \cite{lin2009iterative} and it has been studied for $l^\infty$ functions in \cite{wang2013convergence}.
However, the convergence of IF applied to general signals using uniform or non--uniform filter hasn't been explored yet. With this in mind, in the following we provide an explicit formula for each IMF and we show sufficient conditions on the filter $w(t)$ that ensure the convergence of the IF inner loop. In Section \ref{sec:filters} we present a class of filters that fulfill these sufficient conditions.

\subsection{The Convergence of the IF inner loop}

Let $f(x),x\in \mathbb{R}$ be a continuous signal. When the uniform filter $w(t),$ $t\in[-l,l]$, is used, the moving average of $f(x)$ computed in the inner loop of the IF algorithm can be written as the operator
\begin{equation}\label{movingave}
\mathcal{L}(f)(x) := \int_{-l}^l f(x+t)w(t)\dt
\end{equation}
If we define the operator $\mathcal{S}$ as
\begin{equation*}
  \mathcal{S}(f):= f - \mathcal{L}(f) = (I-\mathcal{L})(f)
\end{equation*}
then one step of the inner loop of IF corresponds to applying $\mathcal{S}$ to the current signal.
Furthermore if we fix the mask length $l$ throughout all the steps of an inner loop, the intermediate function sequence  generated is $\{\mathcal{S}^n(f)\}$.
Hence the convergence of the inner loop is equivalent to the convergence of this sequence. Assuming $\{\mathcal{S}^n(f)\}$ is convergent,
the first IMF of $f(x)$ is given by $I_1 = \lim\limits_{n\rightarrow\infty}\mathcal{S}^n(f)$.

It is proved in \cite{lin2009iterative} that $\{\mathcal{S}^n(f)\}$ is convergent when $f$ is a periodic signal.
In this section, we discuss the convergence of the sequence $\{ \mathcal{S}^n(f)\}$ for $L^2$ signals. Before doing that, we need some preliminary analysis, which takes the symmetry of the filter and the Fourier transform of ${\mathcal{S}^n(f)}$ into account.

Let the filter $w(t)$ be continuous, compactly supported and symmetric, i.e. $w(t)=w(-t), t\in[-l,l]$. So $w(t) \in L^2(\mathbb{R})$. Then the moving average of $f(x)$ computed by
(\ref{movingave}) is the convolution of $f$ and $w$:
\begin{equation}
\begin{split}
\mathcal{L}(f)(x)&= \int_{-l}^l f(x+t)w(t)\dt = \int_{-l}^l f(x-t)w(t)\dt \\
&=\int_{-\infty}^{\infty} f(x-t)w(t)\dt = (f * w)(x).
\end{split}
\end{equation}

The Fourier transform of $w$ is
$\mathcal{F}(w)(\xi)=\int_{-\infty}^{\infty}w(t)e^{-2\pi i t \xi}\dt, \xi \in \mathbb{R}$.
If the signal $f$ is in $L^2(R)$, then the Fourier transform of $f$ is
$\mathcal{F}(f)(\xi)=\int_{-\infty}^{\infty}f(x)e^{-2\pi i x \xi}\dx,  \xi \in \mathbb{R}$.
By the convolution theorem of the Fourier transform, we have $\mathcal{F}(\mathcal{L}(f))(\xi) = \mathcal{F}(f)(\xi)\mathcal{F}(w)(\xi)$,~$\xi\in \mathbb{R}$.
Thus by linearity of the Fourier transform
\begin{equation*}
\begin{split}
\mathcal{F}(\mathcal{S}^n(f))(\xi)&=\mathcal{F}((I-\mathcal{L})^n f)(\xi) = [1-\mathcal{F}(w)(\xi)]^n \mathcal{F}(f)(\xi),  \quad \xi\in \mathbb{R}
\end{split}
\end{equation*}
Based on this preliminary analysis, we present the convergence theorem of the sequence $\{\mathcal{S}^n(f)\}$.
\begin{theorem}\label{theo_1}
Let $w(t), t\in[-l,l]$ be $L^2$, symmetric, nonnegative, $\int_{-l}^l w(t)\dt = 1$ and let $f(x)\in L^2(\mathbb{R})$. \newline
If $|1- \mathcal{F}(w)(\xi)| < 1 $ or $\mathcal{F}(w)(\xi)=0$,
Then
 $\{\mathcal{S}^n(f)\}$ converges and
 $$\lim\limits_{n\rightarrow \infty}{\mathcal{S}^n(f)(x)}= \int_{-\infty}^{\infty} \mathcal{F}(f)(\xi) \chi_{\{\mathcal{F}(w)(\xi)=0 \}}
 e^{2\pi i \xi x} \textrm{d}\xi$$
\end{theorem}
\begin{proof}
$f(x)\in L^2(\mathbb{R})$, thus by the Parseval's Relation
$\int_{-\infty}^{\infty} |\mathcal{F}(f)(\xi)|^2 \textrm{d}\xi = \int_{-\infty}^{\infty}|f(x)|^2\dx <\infty$. There are two possibilities, either
$|1-\mathcal{F}(w)(\xi)|<1$  or  $\mathcal{F}(w)(\xi)=0$.

Case ``$|1-\mathcal{F}(w)(\xi)|<1$''
\begin{equation*}
|\mathcal{F}(\mathcal{S}^n(f))(\xi)| = |[1-\mathcal{F}(w)(\xi)]^n \mathcal{F}(f)(\xi)| = |1-\mathcal{F}(w)(\xi)|^n |\mathcal{F}(f)(\xi)| <|\mathcal{F}(f)(\xi)|
\end{equation*}
in particular $\lim\limits_{n\rightarrow\infty}|\mathcal{F}(\mathcal{S}^n(f))| = 0$, which implies that also $\mathcal{F}(\mathcal{S}^n(f))$ converges to 0 as $n\rightarrow\infty$.

Case ``$\mathcal{F}(w)(\xi)=0$''
\begin{equation*}
\mathcal{F}(\mathcal{S}^n(f))(\xi) = [1-\mathcal{F}(w)(\xi)]^n \mathcal{F}(f)(\xi) = \mathcal{F}(f)(\xi)
\end{equation*}

In summary

\begin{equation*}
\lim\limits_{n\rightarrow\infty}\mathcal{F}(\mathcal{S}^n(f))(\xi) =\left\{
  \begin{array}{l l}
   0 & \quad \text{if} \quad |1-\mathcal{F}(w)(\xi)|<1  \\
    \mathcal{F}(f)(\xi) & \quad \text{if} \quad  \mathcal{F}(w)(\xi)=0
  \end{array} \right.
\end{equation*}

Since the Fourier transform is an invertible operator, $\{\mathcal{S}^n(f)\}$ is also convergent and its limit is
\begin{equation}\label{eq:LimitUnFilt}
\lim\limits_{n\rightarrow \infty}{\mathcal{S}^n(f)(x)}= \int_{-\infty}^{\infty} \mathcal{F}(f)(\xi) \chi_{\{\xi:\,\mathcal{F}(w)(\xi)\, =\, 0 \}}
 e^{2\pi i \xi x} \textrm{d}\xi
\end{equation}
\end{proof}

This theorem provides sufficient conditions on the filter that guarantee the convergence of the inner loop. Furthermore (\ref{eq:LimitUnFilt}) is an explicit formula for the IMF of a signal $f$ obtained using IF equipped with the filter~$w$.

The previous sufficient conditions on the filter are not unrealistic. For example, the double average filter $a(t)$ given in (\ref{equ9}) satisfies these requirements and $\mathcal{F}(a)(\xi)=0 $ when $\xi=\frac{k}{l+1}, 1\leq k \leq l+1$.
But it is not the only choice. In fact, filters which have the property $|1- \mathcal{F}(w)(\xi)| < 1 $ or $\mathcal{F}(w)(\xi)=0$
are easy to obtain.

It is well known that for symmetric and nonnegative filters $w$, $\mathcal{F}(w)(\xi)$ is real and
$$\mathcal{F}(w)(\xi) = \int_{-\infty}^{\infty} w(t)\cos(-2\pi i t\xi) \dt. $$
In addition, since $\int_{-l}^l w(t)\dt = 1$, we have

\begin{equation*}
\begin{split}
\left| \mathcal{F}(w)(\xi)\right|= \left|\int_{-\infty}^{\infty} w(t)\cos(-2\pi i t\xi)\dt\right|
& \leq \int_{-\infty}^{\infty}\left| w(t)\cos(-2\pi i t\xi)\right|\dt < \int_{-\infty}^{\infty} \left| w(t) \right|\dt =  \int_{-l}^l w(t)\dt = 1
\end{split}
\end{equation*}

Hence for symmetric and nonnegative filters $w(t),\; t\in [-l,l]$, $-1 < \mathcal{F}(w)(\xi)<1$, for every $\xi\in \mathbb{R}$.
To have that $0\leq \mathcal{F}(w)(\xi)<1$, we can simply consider the filter $u(t)$, $t\in[-2l, 2l]$, given by the convolution of the filter $w(t)$, $t\in[-l, l]$, with itself, i.e.
$$u(t)=w(t)\star w(t).$$
Therefore the Fourier transform of $u(t)$ is simply given by $\mathcal{F}(u)(\xi)=\mathcal{F}(w)(\xi)\cdot\mathcal{F}(w)(\xi)$, which satisfies $0\leq \mathcal{F}(u)(\xi)<1$, for every $ \xi\in \mathbb{R}$.
So every filter given by the convolution of a symmetric, nonnegative and finitely supported $L^2$ filter with itself satisfies the sufficient conditions of Theorem \ref{theo_1}. In Section \ref{sec:filters} we present a class of this kind of filters.

\section{Adaptive Local Iterative Filtering Techniques}\label{sec:ALIF}

In this Section we present the Adaptive Local Iterative Filtering  (ALIF) algorithm and discuss about its convergence. This method is based on IF, the main differences are in the way we compute locally and adaptively the filter length and that to generate the moving average of signals we make use of the so called FP filters produced as solutions of Fokker--Planck equations.

\begin{algorithm}
\caption*{\textbf{ALIF Algorithm} IMF = ALIF$(f)$}
\begin{algorithmic}
\STATE IMF = $\left\{\right\}$
\WHILE{the number of extrema of $f$ $\geq 2$}
      \STATE $f_1 = f$
      \WHILE{the stopping criterion is not satisfied}
                  \STATE  compute the filter length $l_n(x)$ for $f_{n}(x)$
                  \STATE  $f_{n+1}(x) = f_{n}(x) -\int_{-l_n(x)}^{l_n(x)} f_n(x+t)w_n(x,t)\dt$
                  \STATE  $n = n+1$
      \ENDWHILE
      \STATE IMF = IMF$\,\cup\,  \{ f_{n}\}$
      \STATE $f=f-f_{n}$
\ENDWHILE
\STATE IMF = IMF$\,\cup\,  \{ f\}$
\end{algorithmic}
\end{algorithm}

As for IF, ALIF algorithm contains two loops: the inner and the outer loop. The former captures a single IMF, while the latter produces all the IMFs embedded in a signal.

Given a signal $f(x), x\in\R$, we define the operator that represents the moving average of the signal $f(x)$ as $\mathcal{L}_n(f)(x)=\int_{-l_n(x)}^{l_n(x)} f(x+t)w_n(x,t)\dt$, where $w_n(x,t)$, $t\in [-l_n(x), l_n(x)]$, is a filter with mask length $2l_n(x)$.

Assuming $f_1=f$, if we define the operator which captures the fluctuation part of $f_n$ as $\mathcal{S}_{1,n}(f_n)= f_n-\mathcal{L}_{1,n}(f_n)=f_{n+1}$, then the first IMF is given by $I_1=\lim_{n\rightarrow\infty} \mathcal{S}_{1,n}(f_n)$, where
$\mathcal{L}_{1,n}$ depends on the mask length $l_n(x)$ at step $n$.

In practice we do not let $n$ to go to infinity, instead we use a stopping criterion. We first define $I_{1,n}=\mathcal{S}_{1,n}(f)$ where $\mathcal{S}_{1,n}$ denotes the operator used in the $n$--th step of the $1$--st inner loop. We define then
\begin{equation}\label{eq:SD}
SD:=\frac{\|I_{1,n}-I_{1,n-1}\|_2}{\|I_{1,n-1}\|_2}
\end{equation}
We can either stop the process when the value $SD$ reaches a certain threshold as suggested in \cite{huang1998empirical} and \cite{lin2009iterative} or we can simply introduce a limit on the maximal number of iterations for all the inner loops.
It is also possible to adopt different stopping criteria for different inner loops.
Considering for instance a noisy signal, we may use a loose stopping criterion for the first few IMFs, this will reduce the number of noise components. For the remaining IMFs we could use instead a more restrictive stopping criterion to detect finer differences in the patterns of the components.
These stopping criteria can be used in both the ALIF and IF methods.

To produce the $k$-th IMF, as we do in the IF method, we apply the previous process
to the remainder signal $r=f-I_1-\ldots -I_{k-1},\, k\in \N$, and the algorithm stops when $r$ becomes a trend signal, meaning it has at most one local extremum.

The difference here is that the operator $\mathcal{L}$ is no more a plain convolution of $f$ and $w$ since the mask length $l_n(x)$ does depend on $x$.

The computation of $l_n(x)$ is a crucial step in the ALIF technique. It has to be a strictly positive function and it can be either a constant for every $x$, in which case $l_n(x)$ is a \emph{uniform mask length} and the ALIF algorithm simply reduces to the IF method, or $l_n(x)$ can change point by point producing a \emph{non--uniform mask length}.
In both cases the choice of $l_n(x)$ is not unique. We discussed the uniform case in the previous section where we provided the formula (\ref{eq:Unif_Mask_length}) for the computation of the mask length. In this section we focus on the non--uniform case.

Given a signal $f(x)$, like the one depicted in Figure \ref{fig:doubleHat}, we make the assumption that the distance of its subsequent local extrema gives a measure of the local average instantaneous period of the highest frequency IMF contained in $f(x)$. Based on this assumption the main idea is to compute the non--uniform mask length $l_n(x)$ as a multiple of the distance of subsequent local minima and maxima of $f(x)$. In order to have a continuously varying and smooth $l_n(x)$ we can interpolate the values of the distance of the subsequent extrema of $f(x)$, ref. red dashed curve plotted in Figure \ref{fig32b}. However the derived mask length contains, usually, high frequency oscillations and the ALIF method, equipped with such non--uniform mask length, produces artifacts and in general diverges. To avoid these problems we remove the high frequency oscillations contained in  $l_n(x)$ to produce a slowly varying mask length. This can be achieved, for instance, applying the IF method to the computed mask length to identify the first few IMFs corresponding to high frequency oscillations. The new slowly varying mask length is produced subtracting the high frequency IMFs from $l_n(x)$, ref. solid blue curve in Figure \ref{fig32b}.

\begin{figure}[H]
        \begin{subfigure}[b]{0.5\textwidth}
                \centering
                \includegraphics[width=\textwidth]{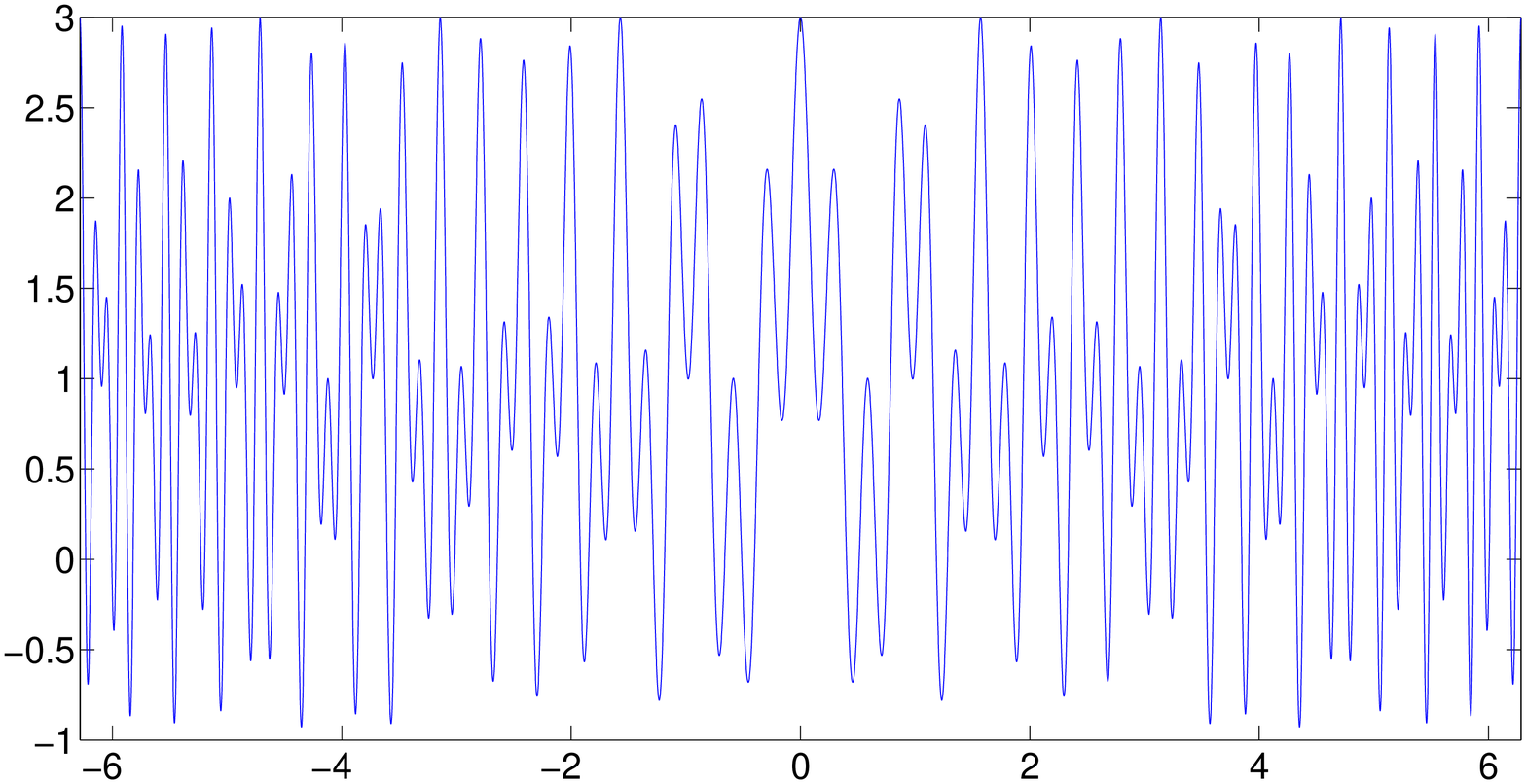}
                \caption{}
                \label{fig:doubleHat}
        \end{subfigure}%
        ~ 
        \begin{subfigure}[b]{0.5\textwidth}
                \centering
                \includegraphics[width=\textwidth]{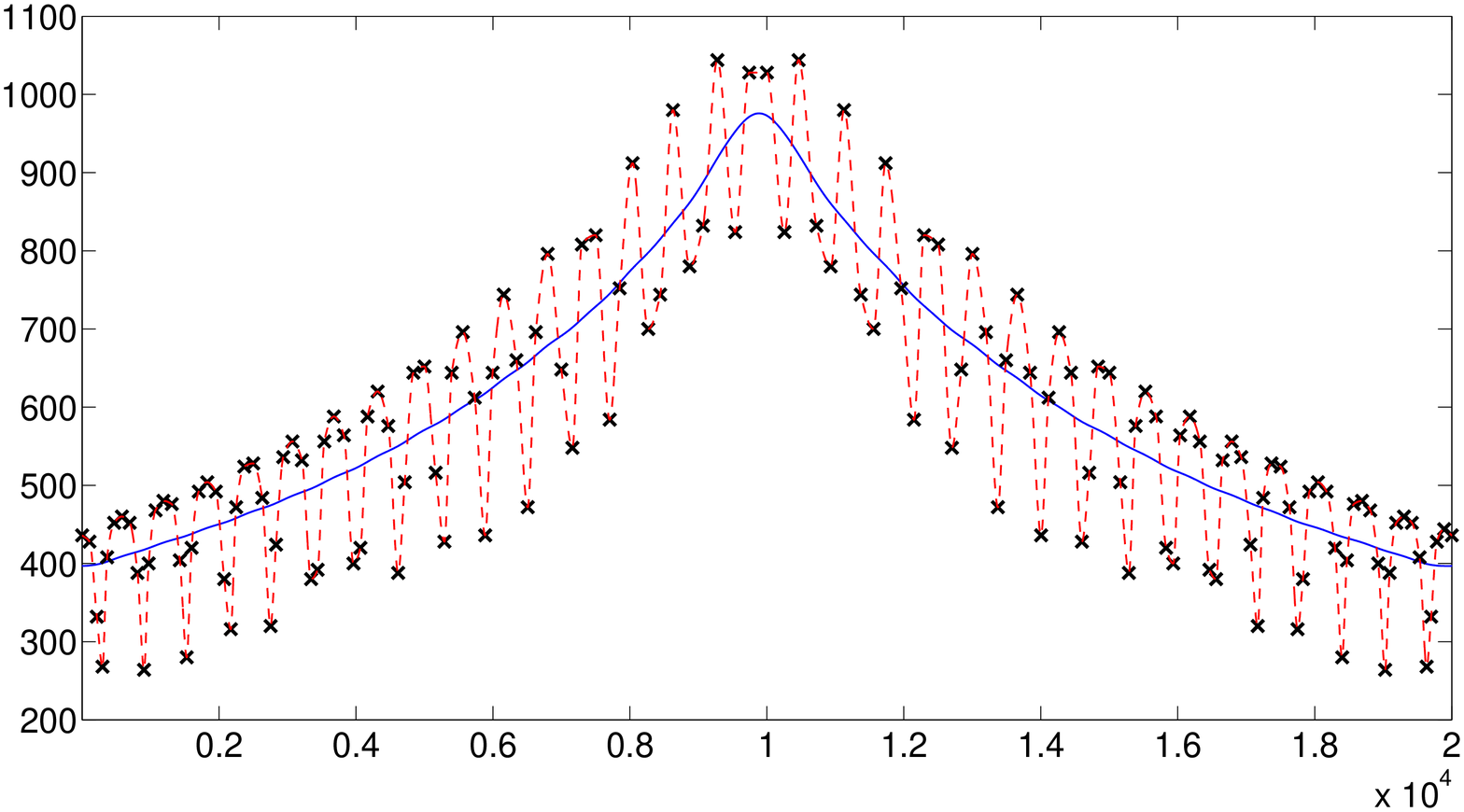}
                \caption{}
                \label{fig32b}
        \end{subfigure}
        \caption{ (\subref{fig:doubleHat}) Signal $f(x)$. (\subref{fig32b}) The black crosses mark the distance between consecutive extrema. The dashed red curve is the interpolant of the black crosses. The solid blue line is the mask length $l_n(x)$ produced after removing high frequencies oscillations from the red dashed curve. }\label{fig:MaskLength}
\end{figure}

The ALIF method, equipped with the newly generated mask length, converges and it does decompose the given signal into the expected IMFs, as shown in Section \ref{sec:Experiments}, Example 3, page \pageref{ex3}.

However a fundamental question is still open: under which conditions is the ALIF method convergent?

What we observed so far is that, to ensure convergence of ALIF method equipped with a non--uniform mask length, we must have a slowly varying mask length. In the limit case, when the mask length is not changing at all, the ALIF method reduces to the IF technique and, thanks to Theorem \ref{theo_1}, we have sufficient conditions on the filter which ensure the convergence for any possible uniform mask length.

To establish a rigorous convergence theorem for the inner loop of the ALIF algorithm when equipped with a non--uniform mask length
we define the operator
$$\mathcal{L}_{w,l}(f) := \int_{-l(x)}^{l(x)} f(x+t)w(x,t)\dt$$
and we consider the
equivalent formulation for the operator $\mathcal{S}_n(f_n)$ given by
\begin{equation}\label{adaptive_2}
f_{n+1}(x)= \mathcal{S}_n(f_n)(x) = f_{n}(x) - \int_{-L}^L f_n(x + g_n(x,y)) W(y)\dy
\end{equation}
where  $W(y), y\in[-L,L]$ is a filter with fixed length $2L$ and $g_n(x,y):\R \times [-L,L]\to \R \times [-l_n(x),l_n(x)]$ is a scaling function which can be set for example to be the linear function $g_n(x,y)=l_n(x)y/L$ or the cubic one $g_n(x,y)=l_n(x)y^3/L^3$. By means of $g_n$ we ensure that $f_n(x+g_n(x,y))$ catches up with the fixed filter $W$.

\begin{theorem}\label{theo_2}
Let $f(x),x\in \mathbb{R}$, be continuous and in $L^{\infty}(\mathbb{R})$.  Let
\begin{equation}\label{requi}
\varepsilon_n = \frac{\|\mathcal{L}_{w_{n+1},l_{n+1}}(f_{n+1})\|_{L^{\infty}} }{\|\mathcal{L}_{w_{n},l_{n}}(f_{n})\|_{L^{\infty}}}, \quad\quad
\delta_n = \frac{\|\mathcal{L}_{w_{n+1},l_{n+1}}(|f_{n+1}|)\|_{L^{\infty}} }{\|\mathcal{L}_{w_{n},l_{n}}(|f_{n}|)\|_{L^{\infty}}}
\end{equation}
If \quad
\begin{equation}\label{con}
\prod_{i=1}^n \varepsilon_i \rightarrow 0  ,\quad \prod_{i=1}^n \delta_i \rightarrow c>0 , \quad \text{as} \quad n\rightarrow \infty
\end{equation}
Then $\{f_n(x)\}$ converges a. e. to an IMF.
\end{theorem}

\begin{proof}
By (\ref{requi}) it follows
\begin{equation}\label{epsilon_n2}
\prod_{i=1}^n \varepsilon_i = \frac{\left\|\int_{-l_{n+1}(x)}^{\,l_{n+1}(x)}f_{n+1}(x+t)w_{n+1}(x,t)\dt \right\|_{L^{\infty}}}{\left\|\int_{-l_1(x)}^{\,l_1(x)}f_{1}(x+t)w_{1}(x,t)\dt \right \|_{L^{\infty}}}
\end{equation}
Then
\begin{equation}
\left\|\int_{-l_{n+1}(x)}^{\,l_{n+1}(x)}f_{n+1}(x+t)w_{n+1}(x,t)\dt \right\|_{L^{\infty}} \!\!\!\!\!\!=\prod_{i=1}^n \varepsilon_i \left\|\int_{-l_{1}(x)}^{\,l_{1}(x)}f_{1}(x+t)w_{1}(x,t)\dt \right\|_{L^{\infty}}
\end{equation}
If we take the limit as $n$ goes to infinity of both sides the right end side tends to zeros by hypothesis, hence
\newline
$\left\{\left\|\int_{-l_{n+1}(x)}^{\,l_{n+1}(x)}f_{n+1}(x+t)w_{n+1}(x,t)\dt \right\|_{L^{\infty}} \right \} \rightarrow 0$ as $n\rightarrow \infty$. As a result the moving average
\begin{equation}\label{mov_ave}
\int_{-l_{n+1}(x)}^{\,l_{n+1}(x)}f_{n+1}(x+t)w_{n+1}(x,t)\dt \rightarrow 0 \quad \text{almost uniformly as} \quad n\rightarrow\infty
\end{equation}
and hence $\{ f_n(x)\}$ is almost uniformly convergent.

Let $F(x)$ denote the $\lim\limits_{n\to \infty}f_n(x)$. Since
\begin{equation}
\left\|\int_{-l_{n+1}(x)}^{\,l_{n+1}(x)}|f_{n+1}(x+t)|w_{n+1}(x,t)\dt\right \|_{L^{\infty}} \!\!\!\! =\; \prod_{i=1}^n \delta_i \left\|\int_{-l_1(x)}^{\,l_1(x)}|f_{1}(x+t)|w_{1}(x,t)\dt\right \|_{L^{\infty}}
\end{equation}
if we rewrite the left hand side using the equivalent formulation introduced in (\ref{adaptive_2}) and take the limit we get
\begin{equation}\label{delta_1}
\left\|\lim\limits_{n\rightarrow \infty} \int_{-L}^{\,L}\left|f_{n+1}(x + g_{n+1}(x,y))\right| W(y)\dy \right\|_{L^{\infty}}\!\!\!\!\!\! =  \lim\limits_{n\rightarrow \infty} \left(\, \prod_{i=1}^n \delta_i  \right) \left\|\int_{-l_1(x)}^{\,l_1(x)}|f_{1}(x+t)|w_1(x,t)\dt \right\|_{L^{\infty}}
\end{equation}
Since the sequence of scaling functions $\left\{g_{n}(x,y)\right\}$ is by construction almost uniformly convergent whenever $\{ f_n(x)\}$ almost uniformly converges, if we assume $\lim\limits_{n\to \infty}g_n(x,y)=G(x,y)$, we want to show that (\ref{delta_1}) becomes
\begin{equation}\label{delta_2}
\left\|\int_{-L}^{\,L}|F(x+G(x,y))|W(y)\dy \right\|_{L^{\infty}}\!\!\!\!\!\! = \; c \left\|\int_{-l_1(x)}^{\,l_1(x)}|f_{1}(x+t)|w_1(x,t)\dt \right\|_{L^{\infty}} \!\!\!\!\!\!>0
\end{equation}

The right hand side follows directly from the hypotheses, while for the left hand side it suffices to prove that $\left\| f_{n}(x + g_{n}(x,y))- F(x+G(x,y))\right\|_{L^{\infty}}\rightarrow 0$ as $n\rightarrow \infty$.  We start observing that

\begin{equation}
\begin{split}
 \left\| f_{n}(x + g_{n}(x,y))\right. & \left.-F(x+g_{n}(x,y))+F(x+g_{n}(x,y))-F(x+G(x,y))\right\|_{L^{\infty}}\\
\leq & \left\| f_{n}(x + g_{n}(x,y))-F(x+g_{n}(x,y))\right\|_{L^{\infty}}\! +\left\|F(x+g_{n}(x,y))-F(x+G(x,y))\right\|_{L^{\infty}}
\end{split}
\end{equation}

Since $\{ f_n(x)\}$ is almost uniformly convergent to $F(x)$, we have $\forall \,\varepsilon > 0$  $\exists\, N$  such that \newline
$\left\| f_{n}(x + g_{n}(x,y))-F(x+g_{n}(x,y))\right\|_{L^{\infty}}\leq \varepsilon/2$ $\forall\, n \geq N$. From the continuity of $f(x)=f_1(x)$ it follows by construction that $f_n(x)$ is continuous for every $n\in\N$ and that also $F(x)$ has to be continuous. Therefore $\forall\, \varepsilon > 0$ $\exists\, \delta$ such that $\left\|F(x+g_{n}(x,y))-F(x+G(x,y))\right\|_{L^{\infty}}\leq \varepsilon/2$ for every $\left\| g_{n}(x,y)-G(x,y)\right\|_{L^{\infty}}\leq \delta$. Finally, since  $\{ g_n(x,y)\}$ is almost uniformly convergent to $G(x,y)$, we have $\forall\, \delta > 0$  $\exists\, \widetilde{N}$  such that  $\left\| g_{n}(x,y)-G(x,y)\right\|_{L^{\infty}}\leq \delta$ $\forall\, n \geq \widetilde{N}$.

In conclusion, if we fix $\varepsilon > 0 $, there are $ N,\, \widetilde{N}$  and $M = \max\{N,\, \widetilde{N}\}$ such that

\begin{equation}
\begin{split}
\left\| f_{n}(x + g_{n}(x,y))- F(x\right.& + \left.G(x,y))\right\|_{L^{\infty}} \leq \left\| f_{n}(x + g_{n}(x,y))- F(x+g_{n}(x,y))\right\|_{L^{\infty}}\! +\\
&+\left\|F(x+g_{n}(x,y))-F(x+G(x,y))\right\|_{L^{\infty}}\leq \frac{\varepsilon}{2} + \frac{\varepsilon}{2} \leq \varepsilon\quad\quad \forall n\geq M
\end{split}
\end{equation}

\noindent Whence (\ref{delta_2}) holds true.

\noindent By (\ref{mov_ave}) the moving average of $F(x)$ is zero, and by (\ref{delta_2}) $F(x)$ itself is non zero.  Hence $F(x)$ is an IMF.
\end{proof}

We point out that this  theorem is simply an a posteriori criterion for the convergence of the ALIF technique. No sufficient conditions on the filter and mask length $l_n(x)$ are given which ensure a priori the convergence of this algorithm.

Furthermore we note that in order to satisfy the convergence conditions in Theorem \ref{theo_2}, it is not necessary to have $\varepsilon_n<1$  for each $n\in \mathbb{N}$. The $\varepsilon_n$'s are allowed to be greater than $1$ for some $n\in\mathbb{N}$.
This is consistent with what we observe in the implementation of ALIF.
There are signals whose $L^{\infty}$ norm grows at the beginning of the inner iteration but eventually converge.
We should also remark that $\delta_n$ does not have to be greater than $1$ to have $\prod \delta_i$ to converge to a positive value, for example $\prod\limits_{n=1}^{\infty} (1- \frac{1}{2^n}) >0$.

We conclude this section observing that, even though the function $l_n(x)$ can be computed at each step $n$ of an inner loop, in the implemented code we compute the mask length only in the first step of an inner loop  and then use that value for the subsequent steps. This is to ensure, as we do in the IF technique, that each IMF contains only a limited set of instantaneous frequencies.
Therefore the operators $\mathcal{S}$ and $\mathcal{L}$ are going to be independent on the step number $n$ so that, given the signal $f$, the $k$--th IMF is $I_k=\lim_{n\rightarrow\infty} \mathcal{S}^n(h)$, where $h=f-I_1-\ldots-I_{k-1}$, $\mathcal{S}(h)=h-\mathcal{L}(h)$, and $\mathcal{L}(h)(x)=\int_{-l(x)}^{l(x)} h(x+t)w(t,x)\dt$, where $l(x)$ is the mask length computed in the beginning of the inner loop and $w(t,x)$ a suitable filter function with support in the interval $[-l(x),\,l(x)]$.

\section{Local Filters developed from a PDE model}\label{sec:filters}

An important aspect in ALIF and IF methods is the choice of the low pass filters used in the inner loops.
To handle non--linear and non--stationary signals, we want to design smooth filters
with compact support. We produce such filters by means of a diffusion process.
The idea is very natural: when applying the diffusion to the data, the oscillations will eventually be eliminated and a curve will be generated as the average.

It is well known that diffusion processes are associated with PDEs.  For a given diffusion equation, we can get its fundamental solution and treat it as the filter in the iterative filtering algorithm. Here we note that the well known heat equation may not be a good choice,
because its solution leads to a filter with infinite support which is not desirable.
To have compact support and smoothness for the filters, we select Fokker--Planck equations
to construct what we call \emph{FP filters}.
\noindent Let us consider the Fokker--Planck equation
\begin{equation}\label{equ15}
p_t=-\alpha(h(x)p)_x+\beta(g^2(x)p)_{xx}, \quad \alpha,\; \beta >0
\end{equation}
Assume $h(x)$ and $g(x)$ are smooth enough functions such that there exist $a<0<b$ satisfying:
\begin{itemize}
\item $g(a)=g(b)=0,\quad g(x)>0 $ for $x\in(a,b)$
\item $h(a)<0<h(b)$
\end{itemize}
The $(g^2(x)p)_{xx}$ term generates diffusion effect and pulls out the solution $p$ from the center of $(a,b)$ towards $a$ and $b$ while the $-(h(x)p)_x$ term transports it from $a$ and $b$ towards the center of the interval $(a,b)$. When the two forces are balanced, the steady state is achieved. There exists a smooth non trivial solution $p(x)$ of the stationary problem:
\begin{equation}\label{equ16}
-\alpha(h(x)p)_x+\beta(g^2(x)p)_{xx}=0
\end{equation}
satisfying $p(x)\geq 0$ for  $x\in(a,b)$, and $p(x)=0$ for $x\not\in(a,b)$. That means the solution is concentrated in the interval $[a,b]$ and there is no leakage outside. So $p(x)$ is a local filter satisfying our requirements.

Based on this analysis and given the  Dirac delta function $\delta(x)$,  we shall use the solution to the initial value problem
\begin{equation}\label{equation16}
\begin{aligned}
& p_t  =-\alpha(h(x)p)_x+\beta(g^2(x)p)_{xx}\\
& p(x,0) =\delta\left(x-\frac{a+b}{2}\right)
\end{aligned}
\end{equation}
as the filter in our decomposition algorithm.  By adjusting the functions $h(x), g(x)$ as well as the coefficients $\alpha,\beta$, we can get different shapes for the FP filter. In Figure \ref{fig4}, we plot two steady states, obtained by  Crank--Nicolson scheme, for fixed $h(x)$ and $g(x)$ functions and two different couples $\alpha$, $\beta$.

\begin{figure}
        \begin{subfigure}[b]{0.3\textwidth}
                \centering                
                \includegraphics[width=\textwidth]{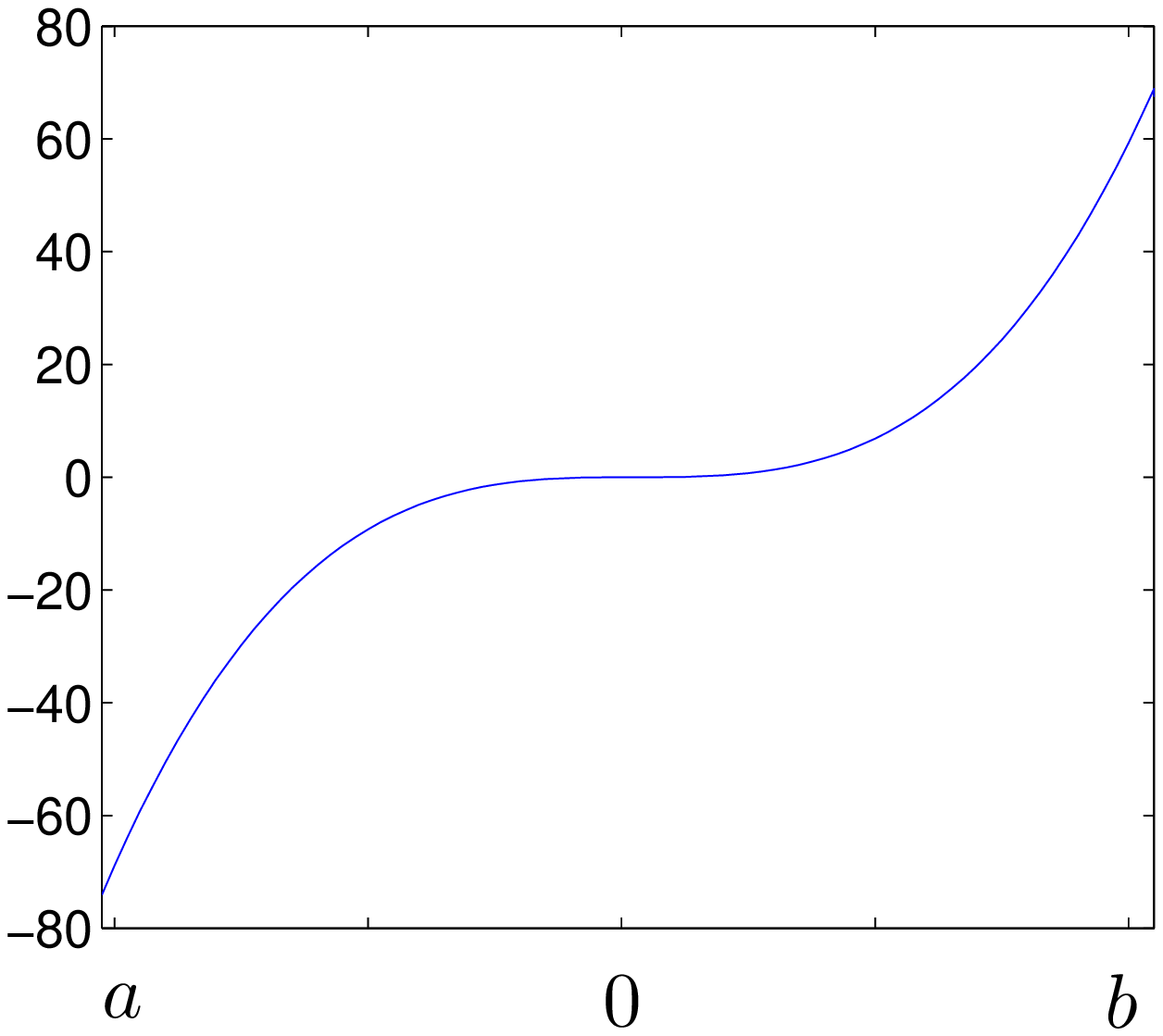}
                \caption{$h(x)$}
                \label{fig4-1}
        \end{subfigure}%
        ~ 
        \begin{subfigure}[b]{0.3\textwidth}
                \centering
                \includegraphics[width=\textwidth]{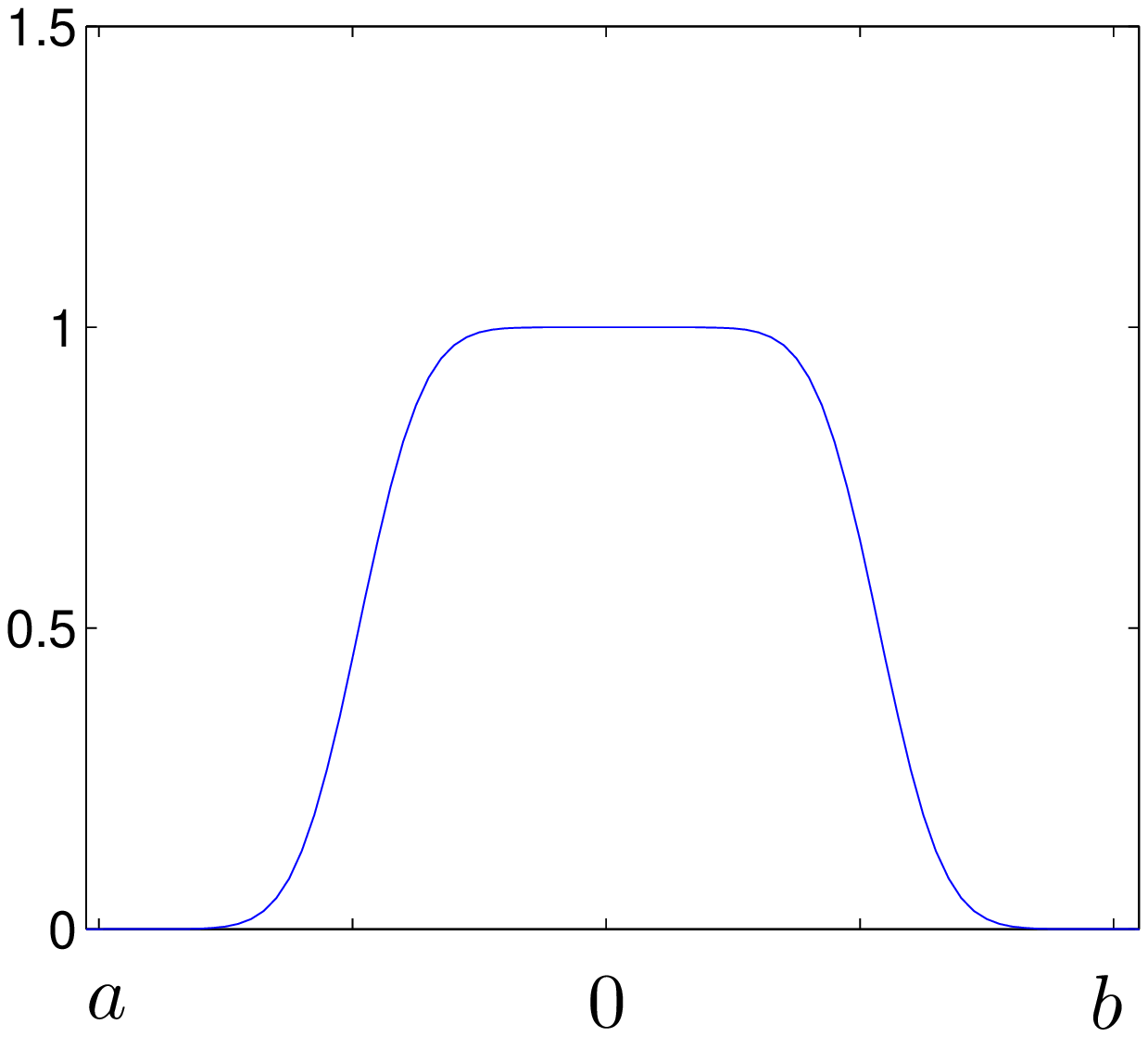}
                \caption{$g^2(x)$}
                \label{fig4-2}
        \end{subfigure}
        ~ 
        \begin{subfigure}[b]{0.3\textwidth}
                \centering
                \includegraphics[width=\textwidth]{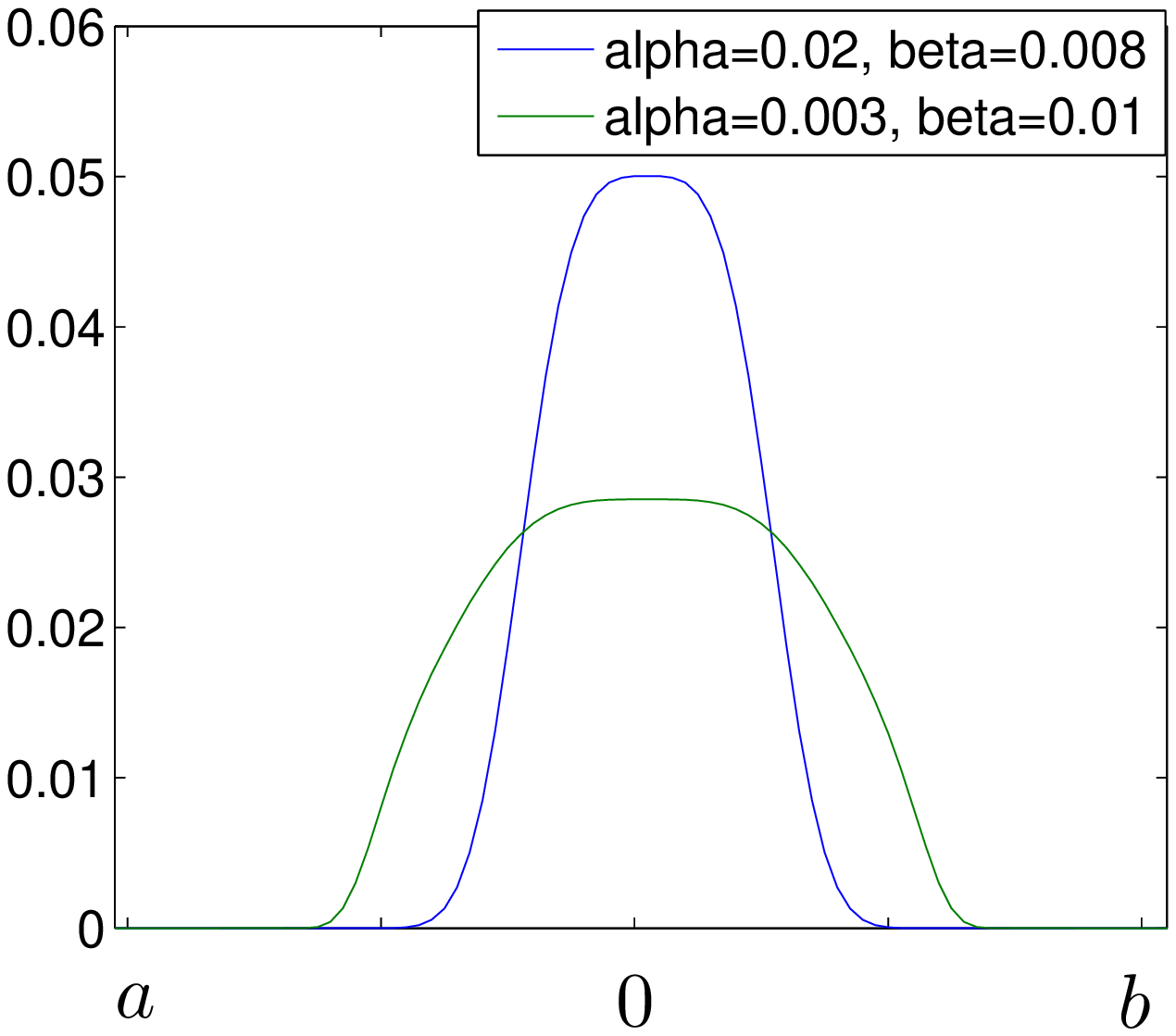}
                \caption{$p(x)$}
                \label{fig4-3}
        \end{subfigure}
        \caption{coefficient functions and steady states of (\ref{equation16}). (\subref{fig4-1}) $h(x)$  is an odd function like $x^3$. (\subref{fig4-2}) $g^2(x)$ is an even function, for instance a smooth approximation of the step function. (\subref{fig4-3}) Two steady states for coefficients $\alpha=0.02$, $\beta=0.008$ and $\alpha=0.003$, $\beta=0.01$ respectively. }\label{fig4}
\end{figure}

We can see that the larger $\alpha$ is, the more the weight is concentrated in the center; on the other hand, the larger $\beta$ is, the more the weight is diffused and less concentrated in the center. When designing the local FP filters based on the Fokker--Planck  equation we first fix the functions $h(x)$ and $g(x)$ and then adjust the coefficients $\alpha$ and $\beta$ to get the filter shape we want.

Another issue is on the design of FP filters with same shape with different lengths. There are at least two approaches available.  One way is to solve the Fokker--Planck equation again with $h(x)$ and $g(x)$ scaled in $x$ for every different $a$ or $b$. Assume we get the steady state of (\ref{equation16}), in order to get the filter with length from $\hat{a}$ to $\hat{b}$, we solve (\ref{equation17}) and the steady state is the filter we want.
\begin{equation}\label{equation17}
\begin{aligned}
& p_t  = -\alpha\left(h\left(\hat{a}\frac{x-(a+b)/2}{a}+\frac{\hat{a}+\hat{b}}{2}\right)p\right)_x\\
&\quad\quad\quad +\beta\left(g^2\left(\hat{a}\frac{x-(a+b)/2}{a}+\frac{\hat{a}+\hat{b}}{2}\right)p\right)_{xx}\\
& p(x,0) =\delta\left(x-\frac{\hat{a}+\hat{b}}{2}\right)
\end{aligned}
\end{equation}

The other way is to solve the Fokker--Planck equation for a fixed $a$ and $b$ only once and take a spatial interpolation of the steady state to get the filter with the desired length.

\begin{figure}
 \begin{center}
                \includegraphics[width=0.8\textwidth]{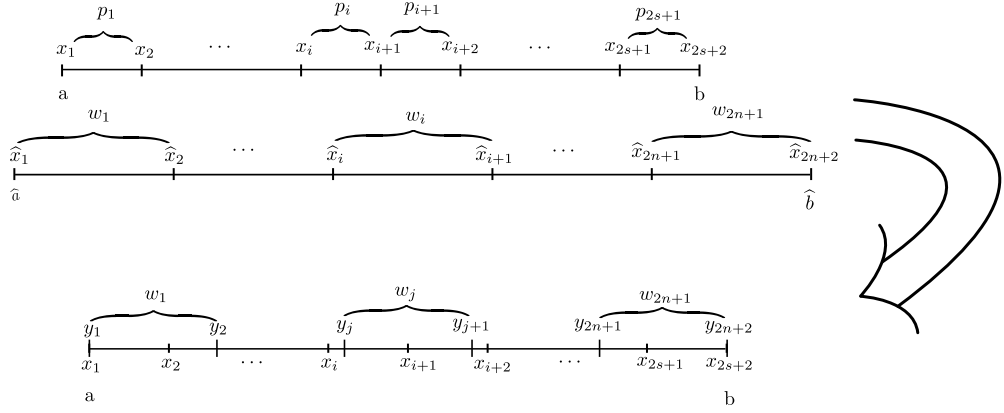}
                \caption{Example of spatial scaling and shifting of the interval $[\hat{a}, \hat{b}]$ to match the interval $[a,b]$.}\label{fig:SpatialScaling}
 \end{center}
\end{figure}

Given the  numerical solution $p(x)$ for the steady state of (\ref{equation16}) computed at the discrete points $\{x_1,\, x_2,$ $\,\ldots ,\, x_{2s+2}\}$, where $s$ is a large natural number and $x_1=a$, $x_{2s+2}=b$, the value $p_i$, $i=1,\,2,\,\ldots,\,2s+1$, represents the weight of the filter in the interval $[x_i,x_{i+1})$ and the sum of weights in all the intervals equals $1$.
We want to compute the numerical solution for the steady state of (\ref{equation16}) for a new interval $[\hat{a}, \hat{b}]$ with a different discretization $\{\hat{x}_1,\, \hat{x}_2,\, \ldots ,\, \hat{x}_{2n+2}\}$, where $n<s$ and $\hat{x}_1=\hat{a}$, $\hat{x}_{2n+2}=\hat{b}$, as shown in Figure \ref{fig:SpatialScaling}. We call this new numerical solution $w(x)$. The idea is to make use of the previously computed values $p_i$ to approximate the new numerical solution $w(x)$.
We start mapping the interval $[\hat{a}, \hat{b}]$ to the interval $[a,b]$ by linear scaling and shifting.
As a result, the points $\{\hat{x}_1, \hat{x}_2, \ldots , \hat{x}_{2n+2}\}$ are mapped into $\{y_1, y_2,\ldots, y_{2n+2}\}$ with $y_1=a$, $y_{2n+2}=b$, ref. Figure \ref{fig:SpatialScaling}.
The values $w_j, j=1,2,\ldots,2n+1$, in the interval $[\hat{a},\hat{b}]$ will be the weights in the intervals $[y_j, y_{j+1})\subseteq [a,b]$, $j=1,2,\ldots, 2n+1$, given by

 \begin{equation}\label{equ192}
 w_j=\int_{y_j}^{y_{j+1}}p(x)\dx
 \end{equation}

The integral in (\ref{equ192}) can be approximated by the Riemann sum based on the discrete points $x_i$ and the weights $p_i$, $i=1,2,\ldots,2s+1$. If $y_j$ falls in the interval $(x_{m_1-1}$,\, $x_{m_1})$ and $y_{j+1}$ in $(x_{m_2}$,\, $x_{m_2+1})$, then
the filter weight $w_j$ will be
\begin{equation}\label{equ194}
w_j = p_{m_1-1}(x_{m_1}-y_j)+ \sum_{i=m_1}^{m_2-1}p_i+ p_{m_2}(y_{j+1}-x_{m_2})
\end{equation}
Using this special interpolation method, the shapes of filters with different lengths are the same. Moreover, the filter length can be any positive real number, also a non integer one.

In summary to produce FP filters with a fixed shape for different support lengths we can either solve the PDE each time or we can solve it once and then interpolate the computed filter to produce the new ones. Solving PDEs numerically is much slower than interpolating existing filters, therefore in our simulations we use the interpolation strategy to get FP filters with same shape for different lengths.

\section{A Different View of Instantaneous frequency and phase}\label{sec:InstFreq}

We first consider the instantaneous frequency definition proposed by Huang et al. which allows to construct the spectrum of a signal \cite{huang1999new, huang2009instantaneous, wu2004study}.
This definition is based on the Hilbert transform of a signal $f(x)$, which is given by
\begin{equation}\label{equ1}
H(f)(x) = \frac{1}{\pi}\text{p.v.} \int_{-\infty}^\infty\frac{f(\tau)}{x-\tau}\textrm{d}\tau
\end{equation}
provided this integral exists as a principal value \cite{hahn1996hilbert}. It is well known that
$z(x) = f(x) + i H(f)(x)$ is an analytic function \cite{gabor1946theory} and one can write it as
\begin{equation}\label{equ2}
z(x) = f(x)+i H(f)(x) = a(x)e^{i \theta(x)}
\end{equation}
where $a(x)$ and $\theta(x)$, both real functions, represent the amplitude and the phase of $z(x)$ respectively. The instantaneous frequency $w(x)$ for the signal $f(x)$ is defined as
\begin{equation}\label{equ3}
w(x)=\frac{\textrm{d}\theta(x)}{\dx}
\end{equation}

This definition may be controversial, first of all because it may lead to negative frequencies, which are not meaningful in practice.
To overcome this issue we simply need to first decompose the signal into IMFs, whose statistical significance was studied in \cite{wu2005statistical}. The IMFs, which can be produced using for instance EMD, IF or ALIF algorithm are guaranteed to have a well behaved positive instantaneous frequency due to their properties: number of extrema and number of zero crossings that is either equal or differ at most by one and the mean value of the upper and lower envelope is zero at any point.

The other problem is that the Hilbert transform is a global operator which is not ideal for a local time--frequency analysis. In this section, we present a new definition  which relies only on local information.

Before we move on to the new definition of the instantaneous frequency, we first review the behavior of a system of second ordinary differential equations (ODEs) in the polar coordinates, where the rotation speed corresponds to one equation of the system. For example, consider the linear system of ODEs $({x}' = \cos{x} ,{y}' = -\sin{x} )$.
By changing the coordinate from $(x,y)$ to $(r,\theta)$ by
$x = r\cos\theta$, $y =-r\sin\theta$,
this linear system can be written as $({\theta}' = 1, {r}' = 0 )$.
The rotation speed of this ODE system is given directly as ${\theta}'$ in the first equation.
Since  ${\theta}'$ is a constant, the rotation speed is a constant, which is consistent with the previous analysis.
Let us consider also the case of a non--linear ODE such as the van der Pol oscillator given as
${x}''+\alpha(x^2-1){x}' + x =0, \alpha>0$.
The standard first--order form of it is $({x}' = y, {y}' = -x -\alpha(x^2 -1)y)$.
Using polar coordinates $x= r\cos\theta$, $y=-r\sin\theta$, this non--linear system can be written as $({r}' = -\alpha(r^2\cos^2\theta-1)r\sin^2\theta , {\theta}' = 1 + \alpha(r^2\cos^2\theta-1)r\sin\theta\cos\theta )$.
The second equation gives the rotation speed ${\theta}'$ which is not a constant any more. When $\alpha<< 1$, ${\theta}'$ is positive.

Let $f(x)$ be an IMF that represents some pattern of oscillations. We treat this $f(x)$ as the $x$ coordinate of some second order ODE, we can get the frequency naturally as the derivative of the phase angle in the polar coordinate ${\theta}'$. It is not necessary to derive the corresponding ordinary differential equation, instead we can compute the phase angle and the rotation speed directly. The procedure contains  two steps: first $f(x)$ is mapped to  $\theta(x)$ in the polar coordinate and second the rotation speed ${\theta}'$ is computed. We mainly have to deal with step one since the second step consists simply in taking a derivative.  When mapping $f(x)$ to $\theta(x)$, we should get rid of the impact of $r$ since
 $r$ and $\theta$ are independent in the polar coordinate. The way we do it is to normalize both $f(x)$ and $f'(x)$ by their envelopes. Thus after normalization, $(f(x),f'(x))$ is the unit circle or a perturbation of the unit circle if the normalization is not perfect.
 In the latter case,
although the perturbation is not a perfect unit circle, we can still see its rotation standing at the center.
Thus the rotation speed ${\theta}'$ is a positive function.

Based on these observations, we propose a new local definition of instantaneous phase and frequency as the phase angle and the rotation speed of the phase angle respectively.

Let $f(x)$ be a function satisfying the IMF requirements, there exists an envelope function $q(x)$ of $f(x)$  such that
\begin{equation}\label{equ10}
F_1(x):=f(x)/q(x) \in[-1,1]
\end{equation}
Considering the derivative of $f(x)$, there exists an envelope function $r(x)$ of $f'(x)$ such that
\begin{equation}\label{equ11}
F_2(x):=f'(x)/r(x) \in[-1,1]
\end{equation}
Functions $q(x)$ and $r(x)$ are not unique. They can be taken for instance to be the cubic splines connecting the local extrema in $f(x)$ and $f'(x)$ respectively.
If we define
\begin{equation}\label{equ12}
F(x)=F_1(x) + i F_2(x)
\end{equation}
then $F(x)$ corresponds to a curve in $[-1,1]\times[-1,1]$ on the complex plane. $F(x)$ is a perturbation of the unit circle and we define the angle for the rotation of $F(x)$ as
\begin{equation}\label{equ13}
\theta(x)=-\arctan{\frac{F_2(x)}{F_1(x)}}
\end{equation}
which corresponds to the instantaneous phase of $f(x)$. While its instantaneous frequency can be defined as
\begin{equation}\label{equ14}
w(x)=\frac{\textrm{d}\theta(x)}{\dx}
\end{equation}

We observe that even though the envelope functions $q(x)$ and $r(x)$ are not unique, in general the instantaneous phase and frequency we just defined depend scarcely on their choice.

Another observation is that if the magnitude of the IMF $f(x)$ or its derivative change significantly in a short time,
the envelopes $q(x)$ and $r(x)$ might not follow closely the changes of the IMF and its derivative. This may cause unexpected errors in the
instantaneous phase and frequency. So it is advisable to identify and properly handle these sudden
changes when constructing the envelopes. To address this problem we can make use of the essentially non--oscillatory (ENO) technique for shock capturing developed in computational fluid dynamics \cite{harten1987uniformly,shu1999high}.

The ENO technique is a {\it divide et impera} method. First, based on the differences of consecutive extrema of a given signal, we detect possible sudden changes in the magnitude using a preselected threshold. If a sudden change is detected between two consecutive extreme points we compute the differences of consecutive sample points in between those extrema. The sudden change is assumed to happen where the left difference differs most from
the right difference. We use this point to divide the signal into two parts and to construct two separate envelopes one for the left and one for the right hand side.

Let us consider two test examples where we compute the instantaneous frequency using both the new definition and  the definition based on Hilbert transform.\\
\textbf{Test 1} The signal in Figure \ref{fig2-1} is given by
\begin{equation}\label{signal_1}
f(x)=(1+0.2\cos{(0.06\pi x)})\sin{[(1+0.1x)x]}, \quad x\in[0,40]
\end{equation}
The amplitude of the signal changes slowly, both frequency analysis methods show the gradual change in the instantaneous frequency as shown in Figure \ref{fig2}. However when there is a significant change in the amplitude of the signal the instantaneous frequencies defined by the two approaches are different.
\begin{figure}[h]
        \begin{subfigure}[b]{0.3\textwidth}
                \centering                
                \includegraphics[width=\textwidth]{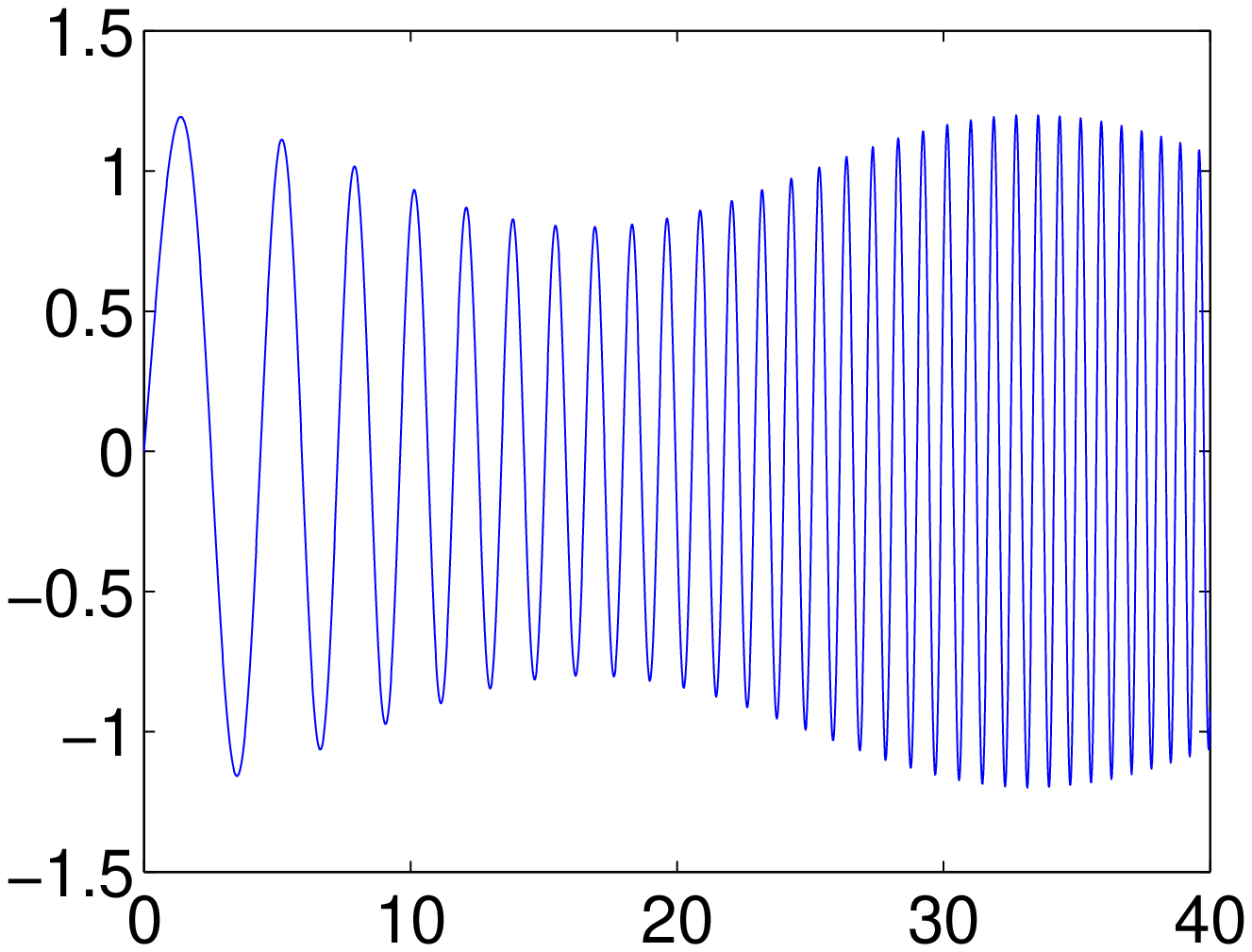}
                \caption{}
                \label{fig2-1}
        \end{subfigure}%
        ~
        \begin{subfigure}[b]{0.3\textwidth}
                \centering
                \includegraphics[width=\textwidth]{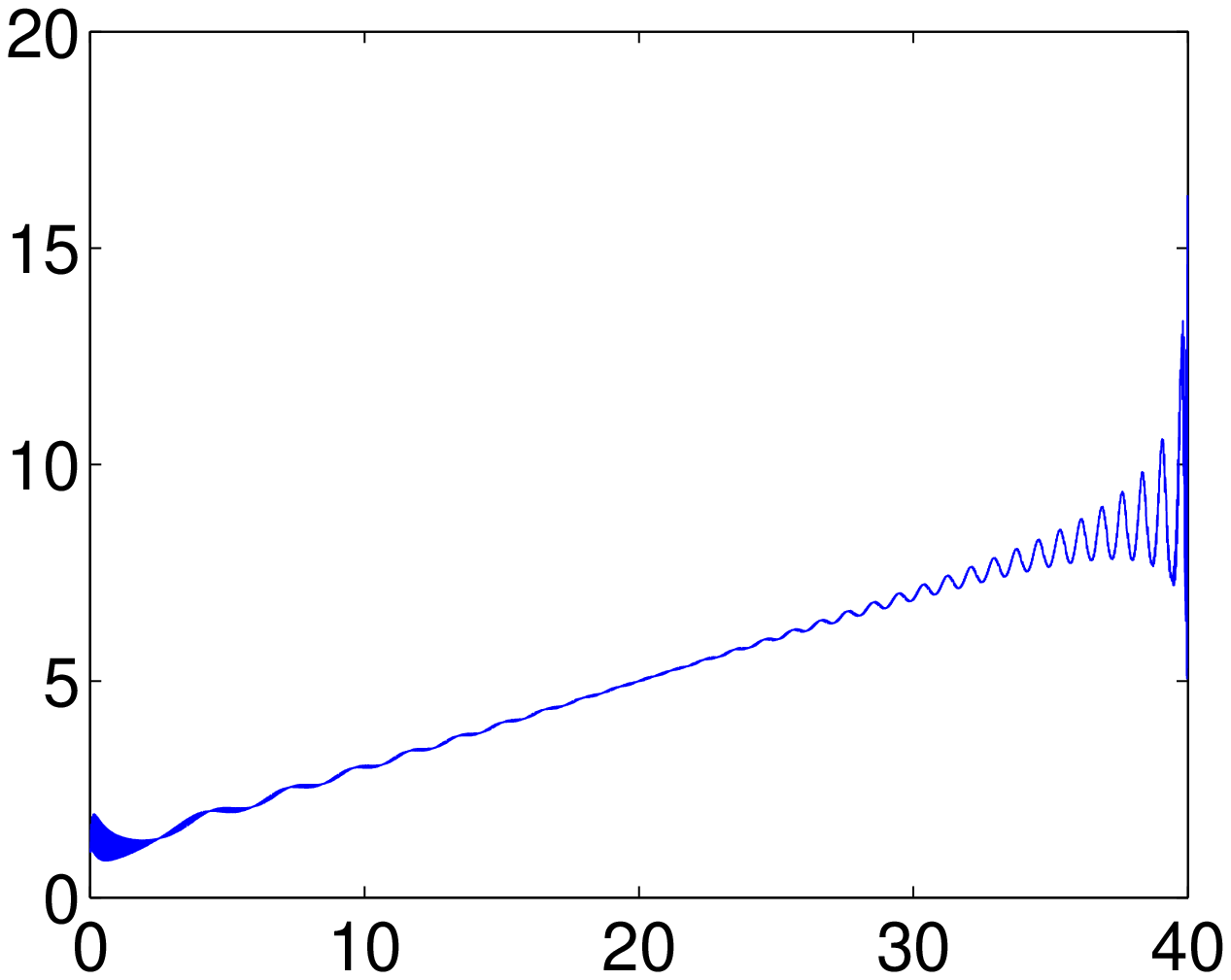}
                \caption{}
                \label{fig2-2}
        \end{subfigure}
        ~
        \begin{subfigure}[b]{0.3\textwidth}
                \centering
                \includegraphics[width=\textwidth]{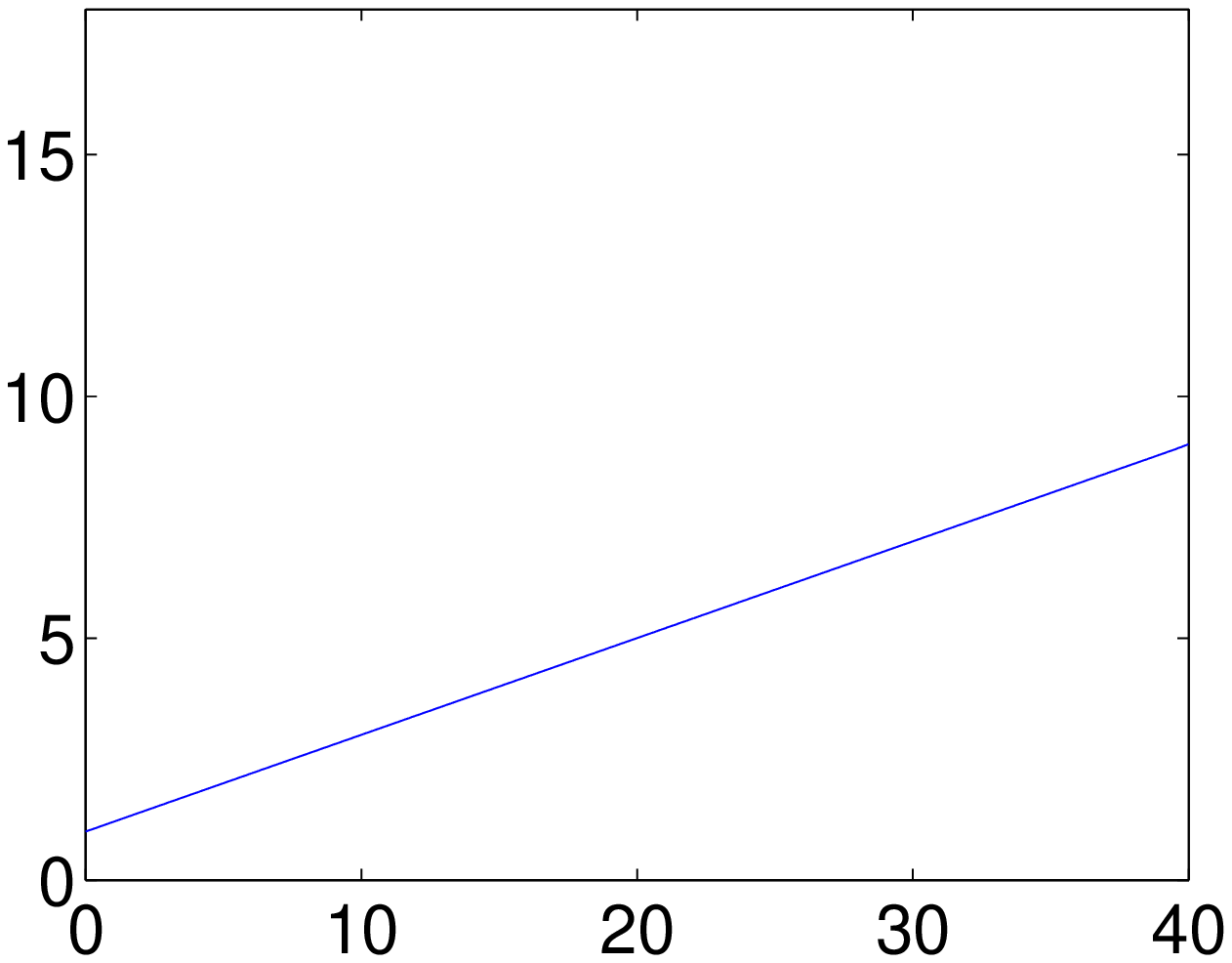}
                \caption{}
                \label{fig2-3}
        \end{subfigure}
        \caption{Test 1. $(a)$ The signal defined in (\ref{signal_1}). The oscillation gradually becomes faster and the amplitude changes mildly. $(b)$  Instantaneous frequency  computed using Hilbert transform. It shows the gradual change in the instantaneous frequency but has some oscillations. $(c)$ Instantaneous frequency computed using the proposed method. It does show the gradual change in the instantaneous frequency and it has almost no oscillations.  }\label{fig2}
\end{figure}\\
\textbf{Test 2} The signal in Figure \ref{fig3-1}  is generated by
\begin{equation}\label{signal_2}
f(x)=a(x)\sin{(2 \pi x)}, \text{where} \quad a(x)=1-0.9\chi_{[3,6]},\quad x\in[0,10]
\end{equation}
 It has a sudden change in the amplitude at time $3$ and $6$. Except for these two positions, the signal is a constant frequency signal. We expect the time frequency analysis to return an instantaneous frequency which is almost a constant. Using the instantaneous frequency definition based on Hilbert transform, the transitory change in the amplitude affects faraway positions and this leads to strange behaviors in the instantaneous frequency. On the other hand, using the proposed method together with the ENO technique we get an instantaneous frequency which is almost a constant, ref. Figure \ref{fig3-3}.

\begin{figure}[h]
        \begin{subfigure}[b]{0.3\textwidth}
                \centering
                \includegraphics[width=\textwidth]{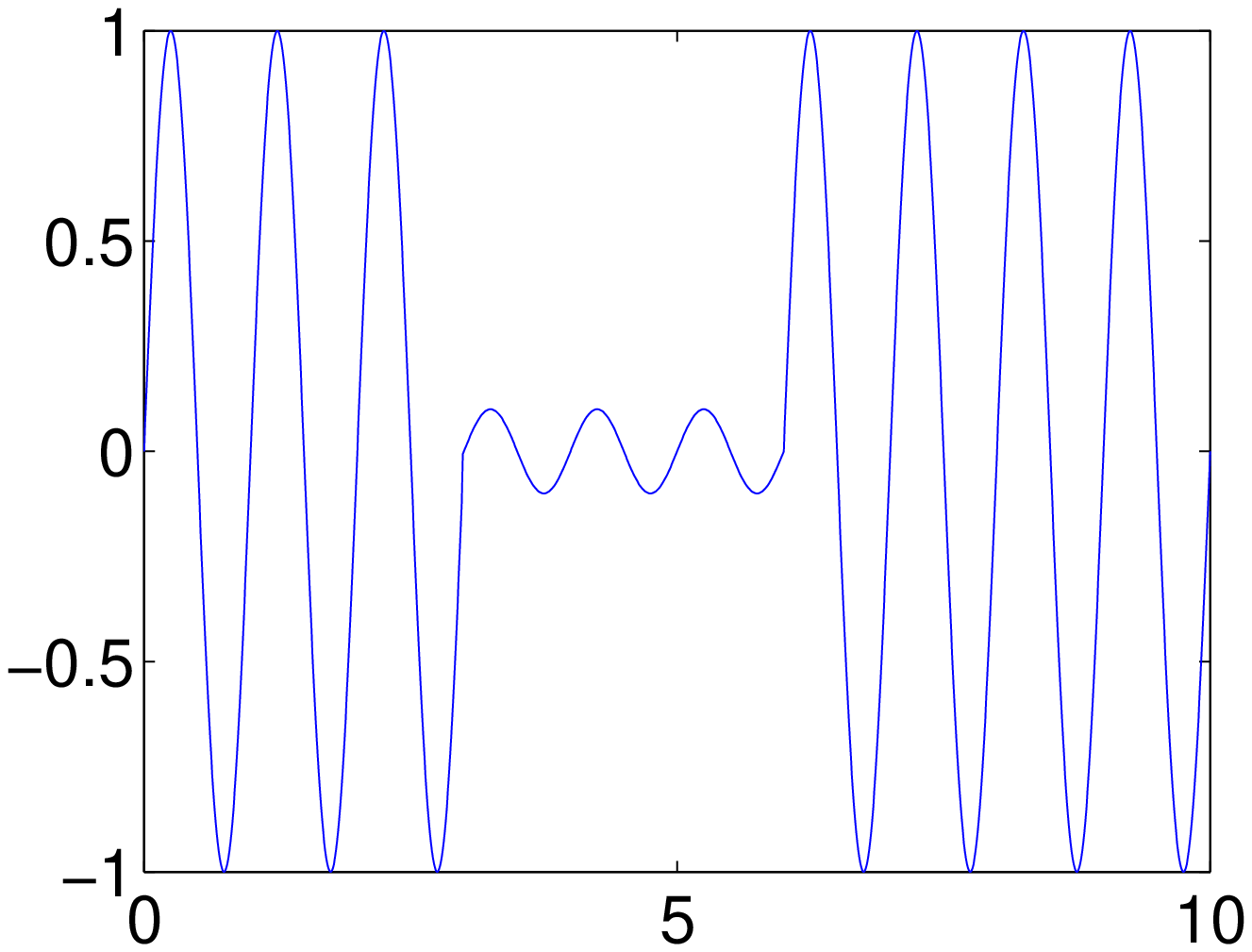}
                \caption{}
                \label{fig3-1}
        \end{subfigure}%
        ~ 
        \begin{subfigure}[b]{0.3\textwidth}
                \centering
                \includegraphics[width=\textwidth]{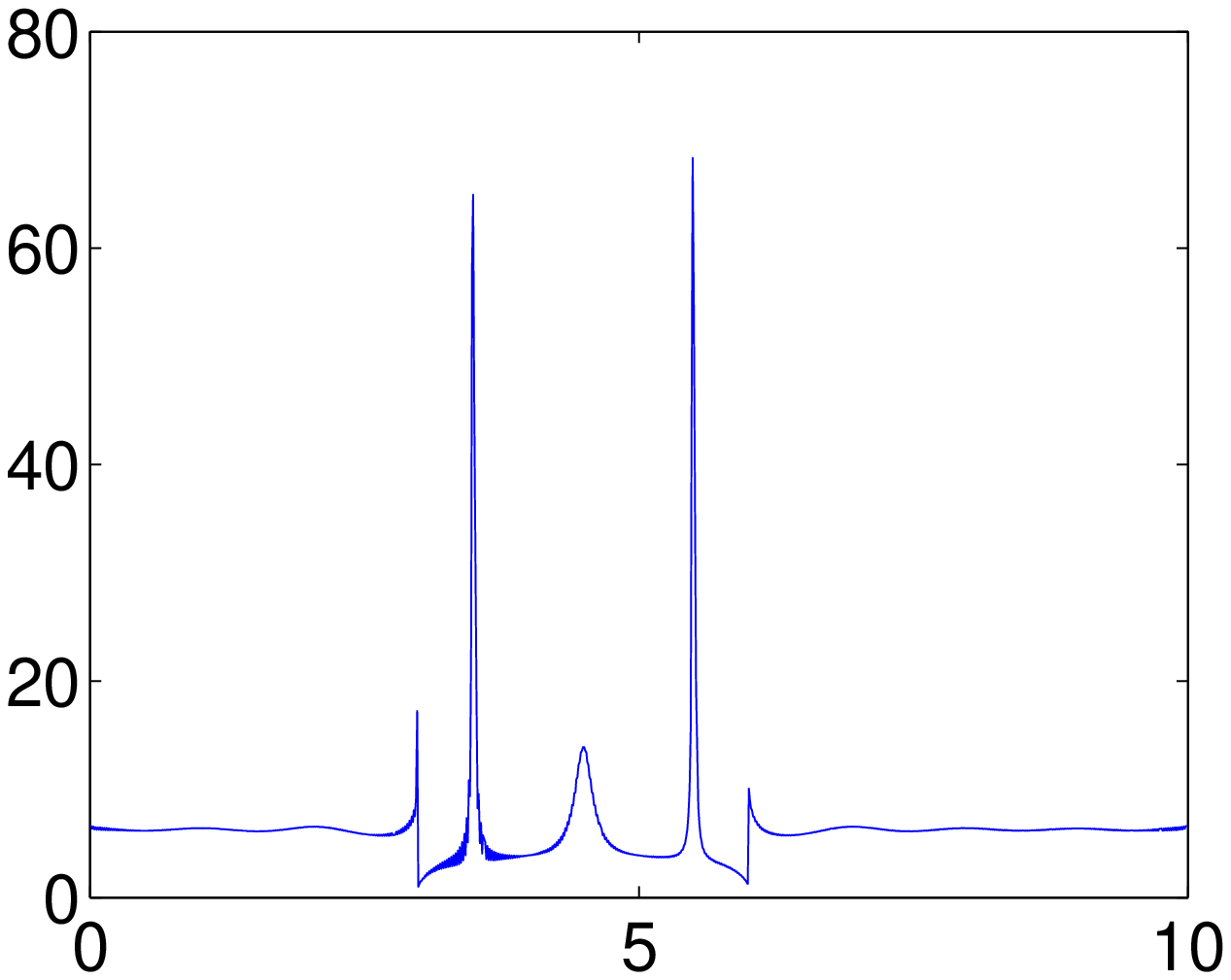}
                \caption{}
                \label{fig3-2}
        \end{subfigure}
        ~ 
        \begin{subfigure}[b]{0.3\textwidth}
                \centering
                \includegraphics[width=\textwidth]{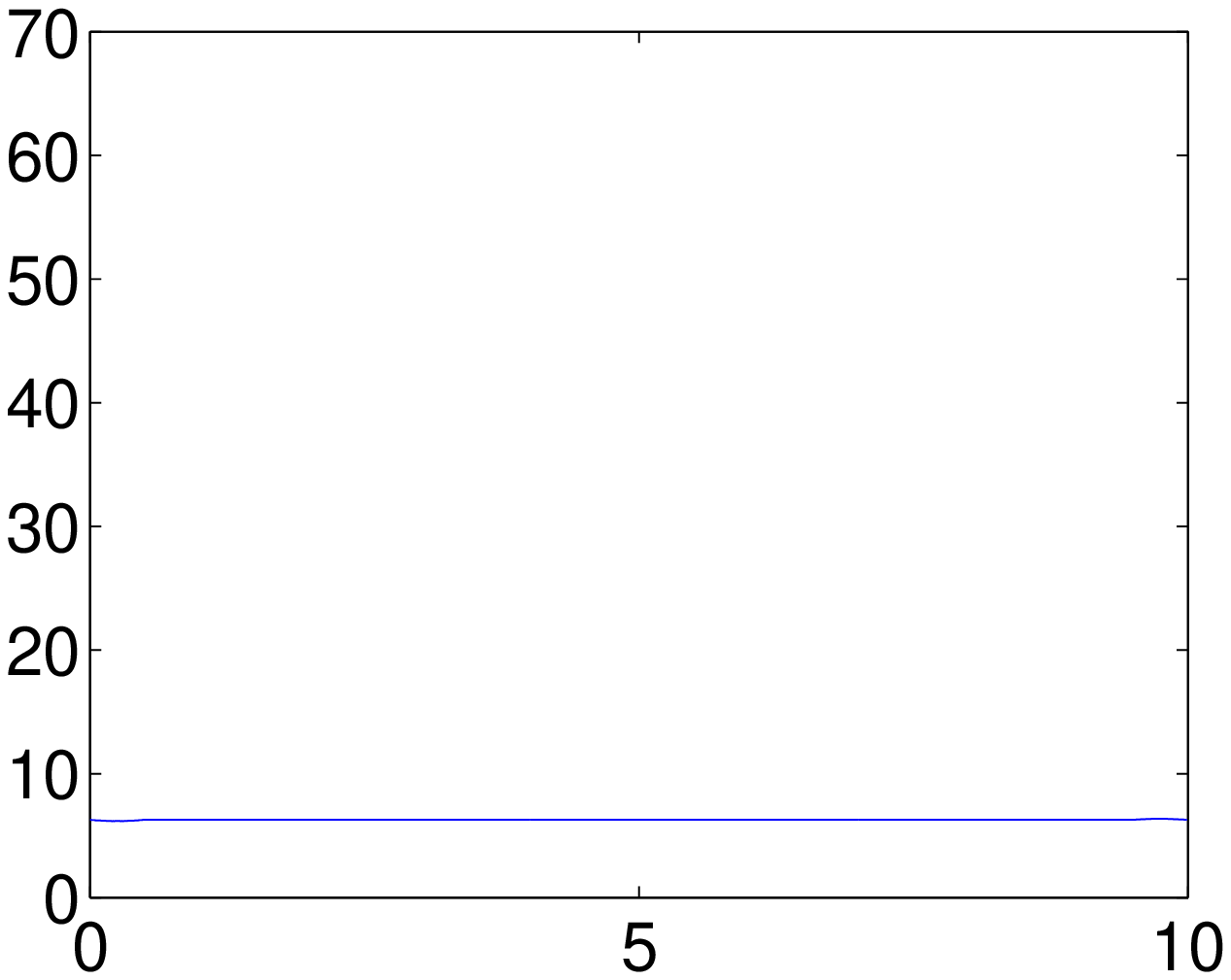}
                \caption{}
                \label{fig3-3}
        \end{subfigure}
        \caption{Test 2. $(a)$ The signal defined in  (\ref{signal_2}) with sudden changes in the amplitude. $(b)$ Instantaneous frequency computed using Hilbert transform $(c)$ Instantaneous frequency computed using the proposed method. }\label{fig3}
\end{figure}

\section{Numerical Experiments}\label{sec:Experiments}

In this section we test both the IF method and the proposed ALIF algorithm on artificial and real life signals. For both techniques we use the FP filter, given in Figure \ref{fig6-1}, convolved with itself. Such FP filter is based on the PDE model described in Section \ref{sec:filters}, where h(x), g(x) are the functions shown in Figure \ref{fig4} and the coefficients are $\alpha=0.005$, and $\beta=0.09$. We point out that FP Filters with different lengths are computed using the special interpolation method introduced in Section \ref{sec:filters}.

As observed previously, a good decomposition method should capture all the finest oscillations around a moving average. That means the IMFs should satisfy at least the following condition: all the local maximal values are positive and all the local minimal values are negative, as shown in Figure \ref{fig6-3}. It is important to note that using an iterative filtering method and tuning the stopping criterion described in (\ref{eq:SD}) we can get an IMF that looks like Figure \ref{fig6-2} to be like Figure \ref{fig6-3}. In the following experiments we generally set the threshold for SD to be around $10^{-5}$.

\begin{figure}
        \begin{subfigure}[b]{0.3\textwidth}
                        \centering
                        \includegraphics[width=\textwidth]{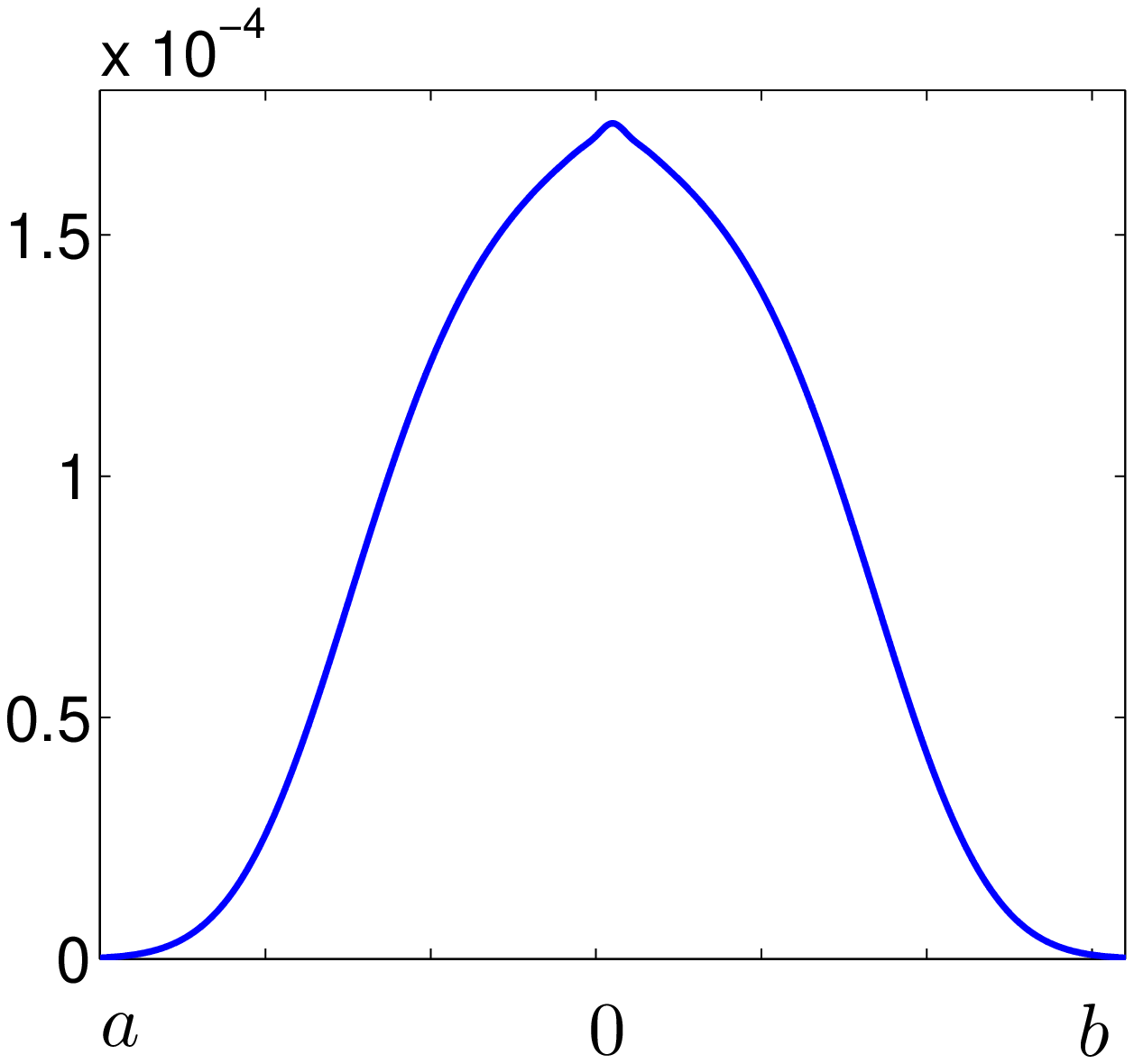}
                        \caption{}
                        \label{fig6-1}
                \end{subfigure}%
                ~ 
                \begin{subfigure}[b]{0.3\textwidth}
                        \centering
                        \includegraphics[width=\textwidth]{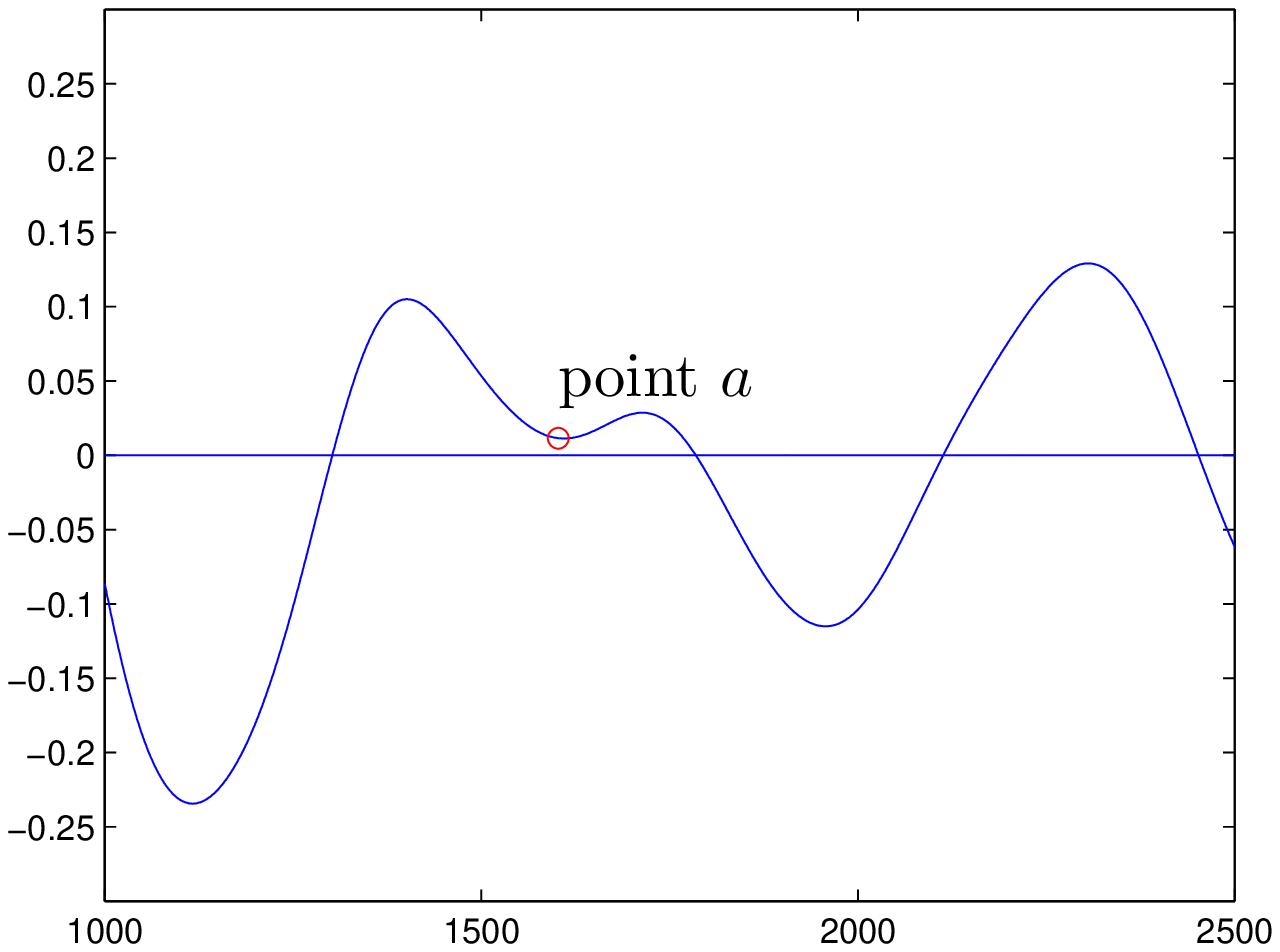}
                        \caption{}
                        \label{fig6-2}
                \end{subfigure}
                ~ 
                \begin{subfigure}[b]{0.3\textwidth}
                        \centering
                        \includegraphics[width=\textwidth]{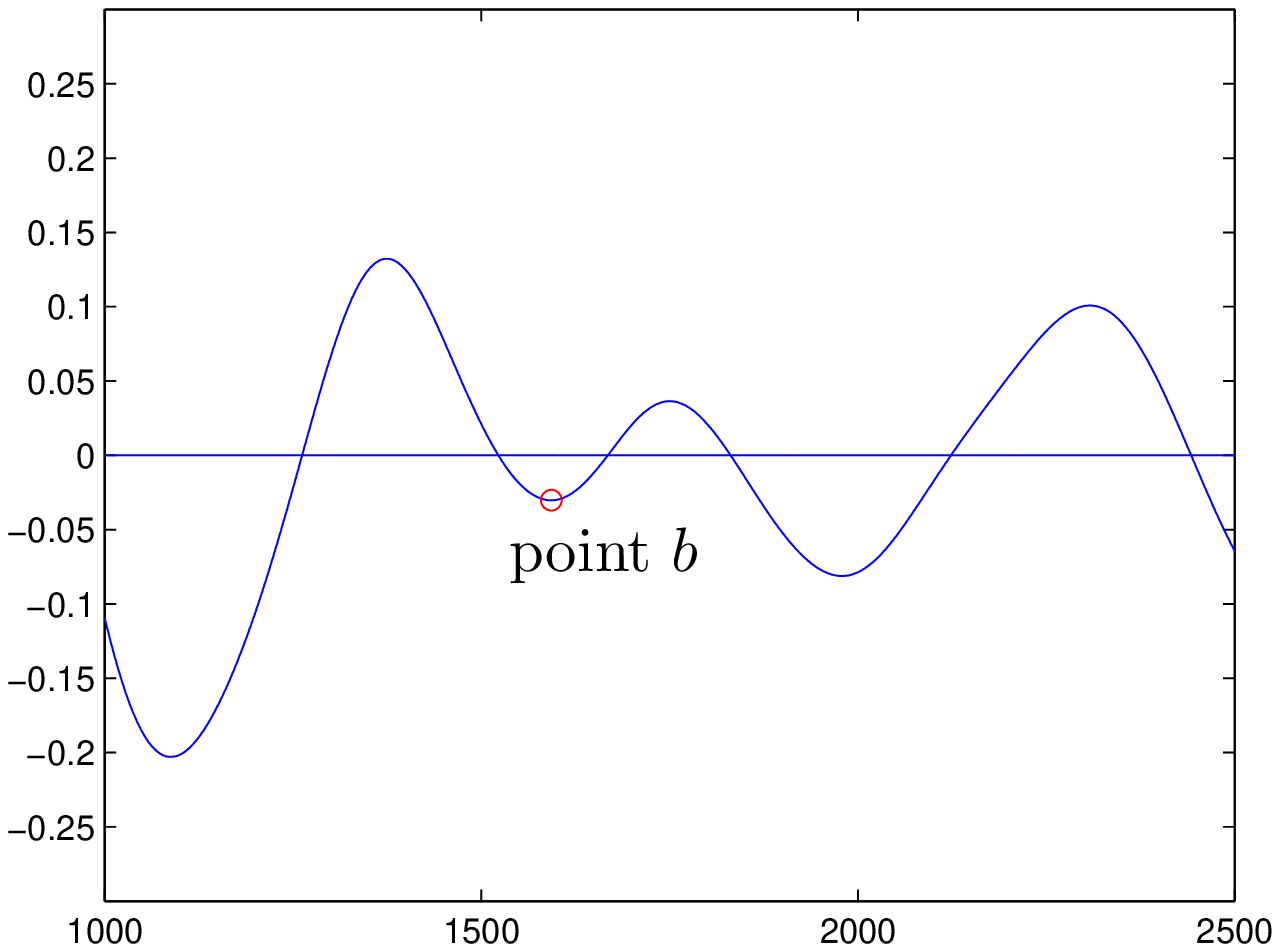}
                        \caption{}
                        \label{fig6-3}
                \end{subfigure}
                \caption{(a) The FP filter we use in the numerical implementations. (b) After $3$ iterations the local minimum point $a$ is above $0$. (c) After $5$ iterations the local minimum point $b$ is below $0$. Other parts do not change significantly from (\subref{fig6-2}) to (\subref{fig6-3}).}  \label{fig6}
\end{figure}

It is time now to test the methods on six artificial and three real--life data sets. We show that both IF and ALIF are stable under perturbations.

All the data sets as well as the Matlab codes used in this section are available at the first author webpage\footnote{\url{www.cicone.com}} or can be requested to the same author\footnote{\url{antonio.cicone@univaq.it}}.

\noindent \textbf{Example 1} We apply the IF algorithm on the non--stationary frequency modulated signal
 \begin{equation}\label{equ24}
 f(x)=4(x-0.5)^2+(2(x-0.5)^2+0.2)\sin{\left((20\pi +0.2\cos{(40\pi x)}) x\right)},\quad x\in[0,1]
 \end{equation}
 From Figure \ref{fig7}, we see that the IF method decompose $f(x)$ into two components. The first is the frequency modulated signal $(2(x-0.5)^2+0.2)\sin{\left((20\pi +0.2\cos{(40\pi x)}) x\right)}$ and the second is the so called trend $4(x-0.5)^2$.
\begin{figure}[H]
\begin{center}
        \begin{subfigure}[b]{0.48\textwidth}
                \centering
                \includegraphics[width=\textwidth]{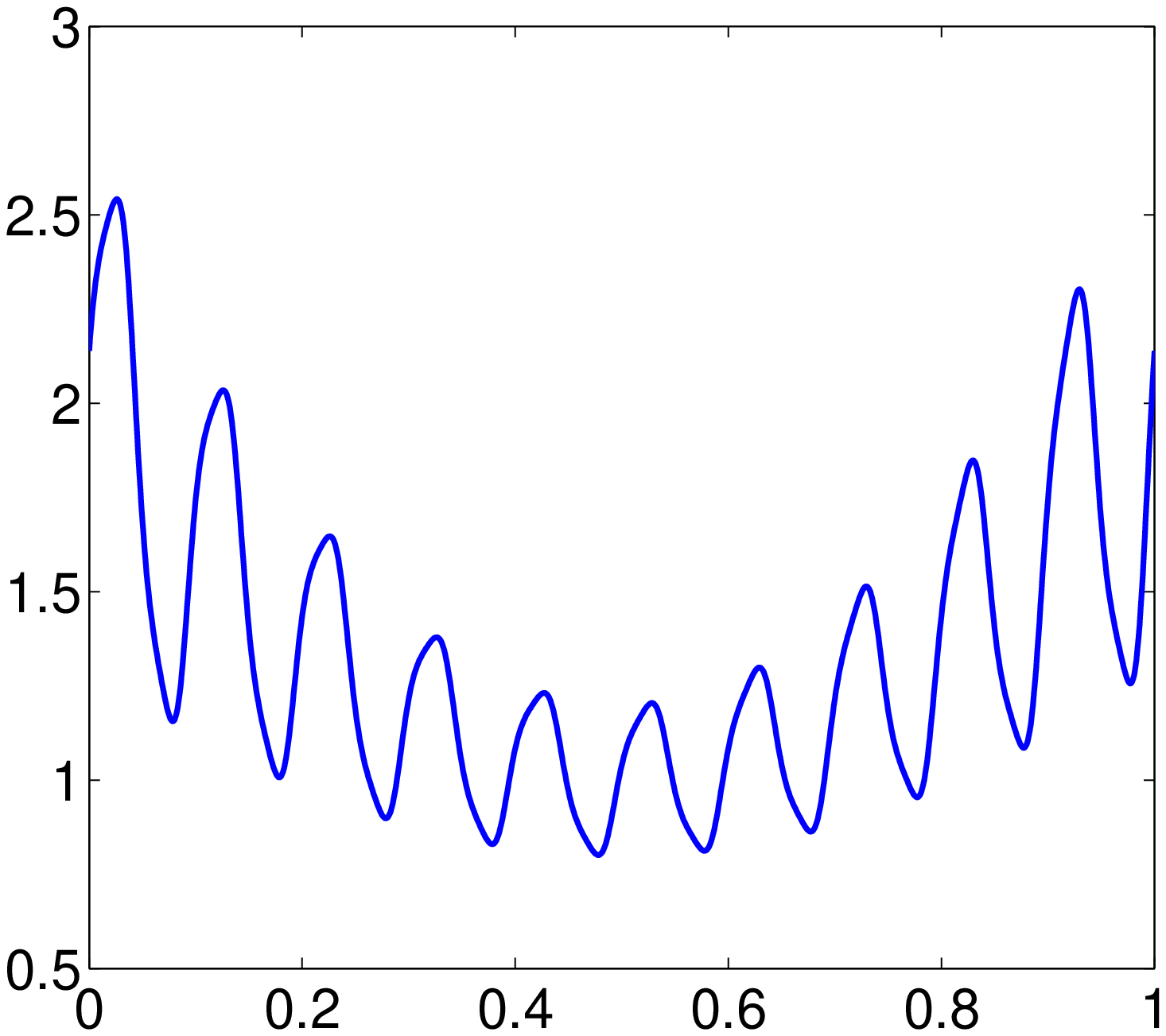}
                \caption{}
                \label{fig7-1}
        \end{subfigure}%
        ~ 
        \begin{subfigure}[b]{0.48\textwidth}
                \centering
                \includegraphics[width=\textwidth]{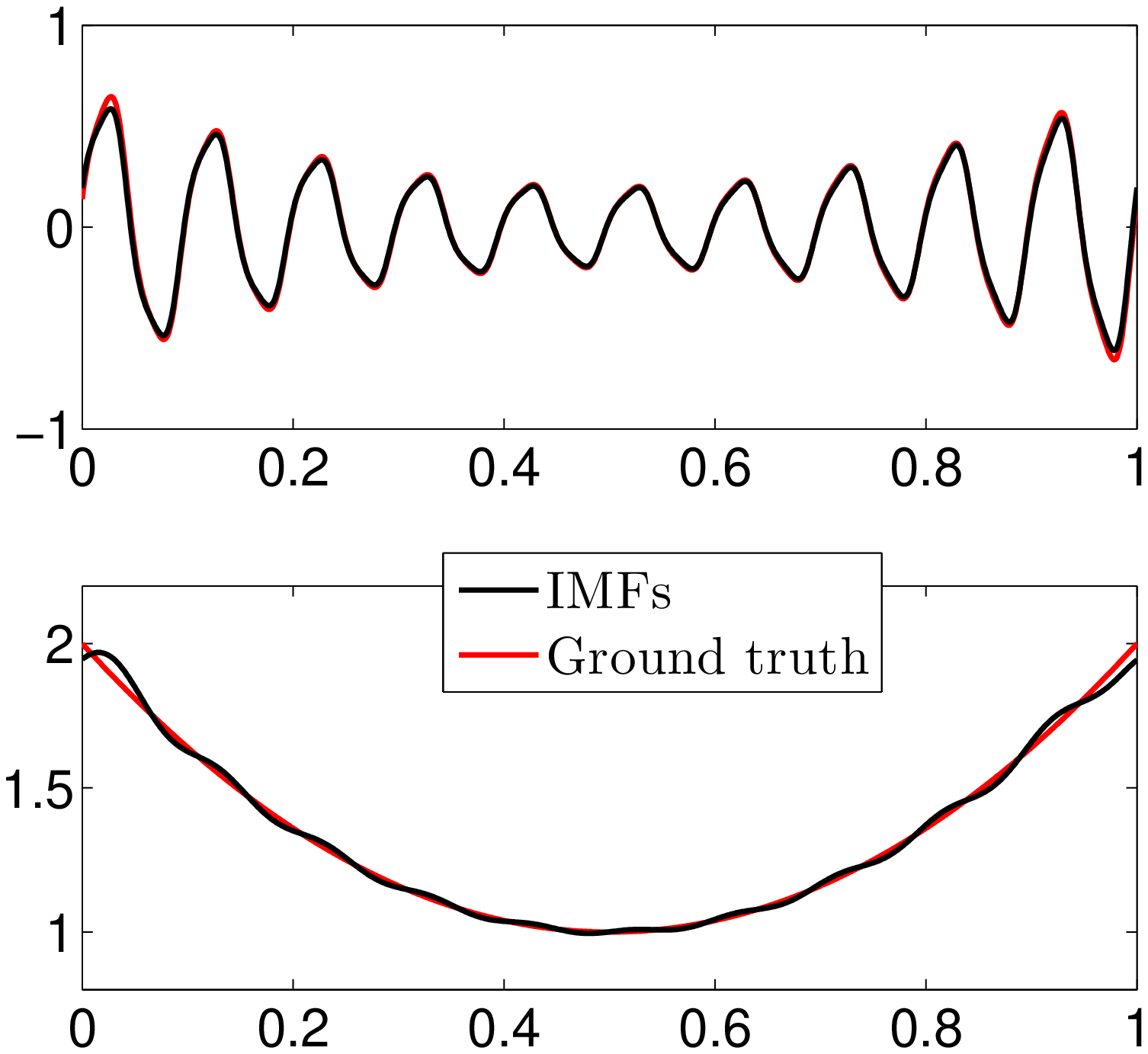}
                \caption{}
                \label{fig7-2}
        \end{subfigure}
\end{center}
        \caption{ (\subref{fig7-1}) The signal given in (\ref{equ24}). (\subref{fig7-2}) The two components in the IF decomposition.}\label{fig7}
\end{figure}

\noindent \textbf{Example 2} We test the IF technique on the highly non--stationary signal
 \begin{equation}\label{equ26}
 f(x)=\sin{(4\pi x)}+0.5\cos{(50\pi|x|-40\pi x^2)},\quad  x\in[-0.4,0.4]
 \end{equation}
After the decomposition, we compute the instantaneous frequency of each component by the method proposed in Section \ref{sec:InstFreq}.
As shown in Figures \ref{fig9} and \ref{fig9bis}, $f(x)$ is separated into two IMFs which reproduce with good accuracy the ground truth components, ref. Figure \ref{fig9-4}. One has a varying instantaneous frequency and the other has a constant instantaneous frequency, ref. Figure \ref{fig9-3}.

 \begin{figure}[H]
        \begin{subfigure}[b]{0.48\textwidth}
                \centering
                \includegraphics[width=\textwidth]{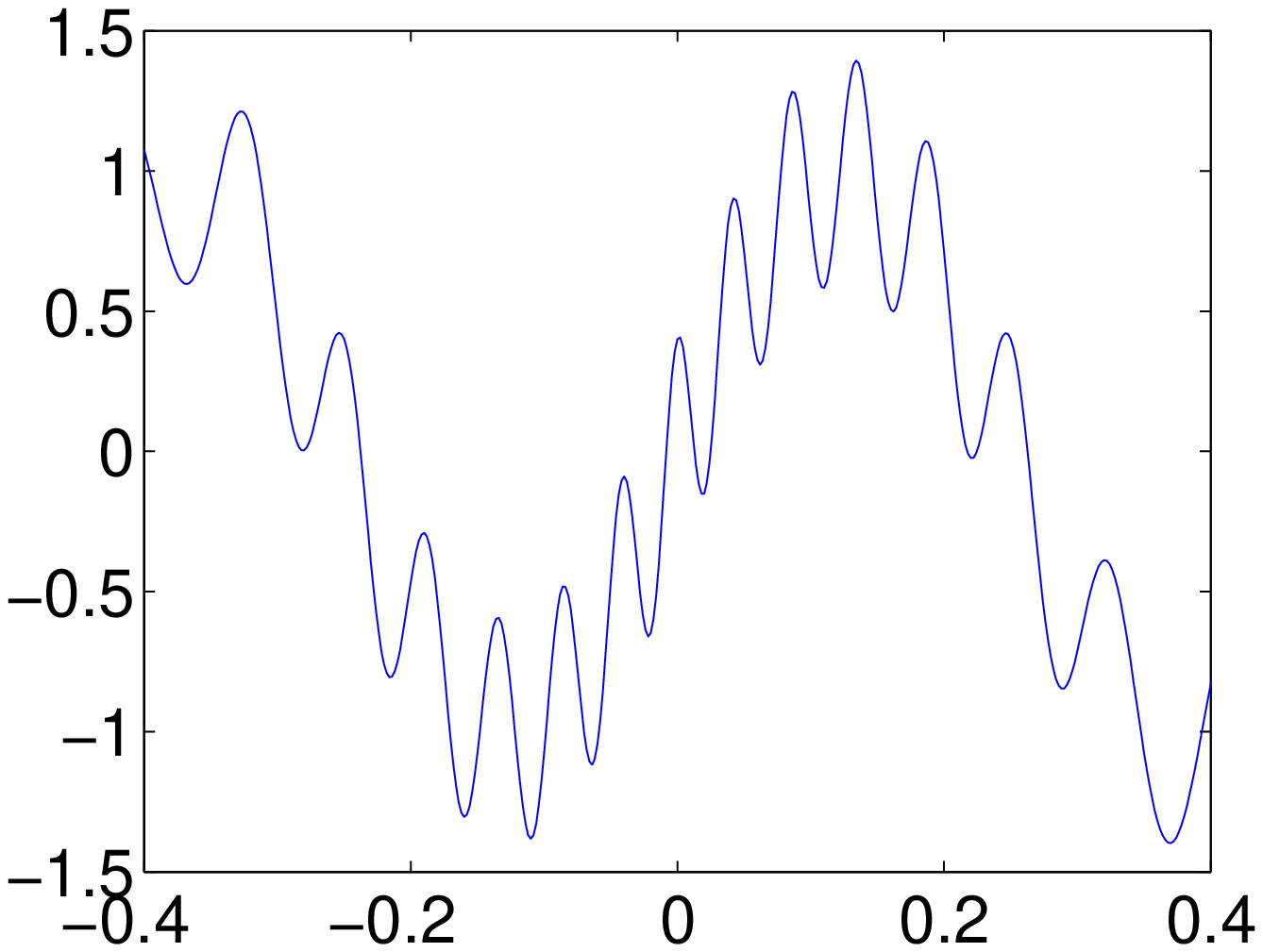}
                \caption{}
                \label{fig9-1}
        \end{subfigure}
        ~\begin{subfigure}[b]{0.48\textwidth}
                \centering
                \includegraphics[width=\textwidth]{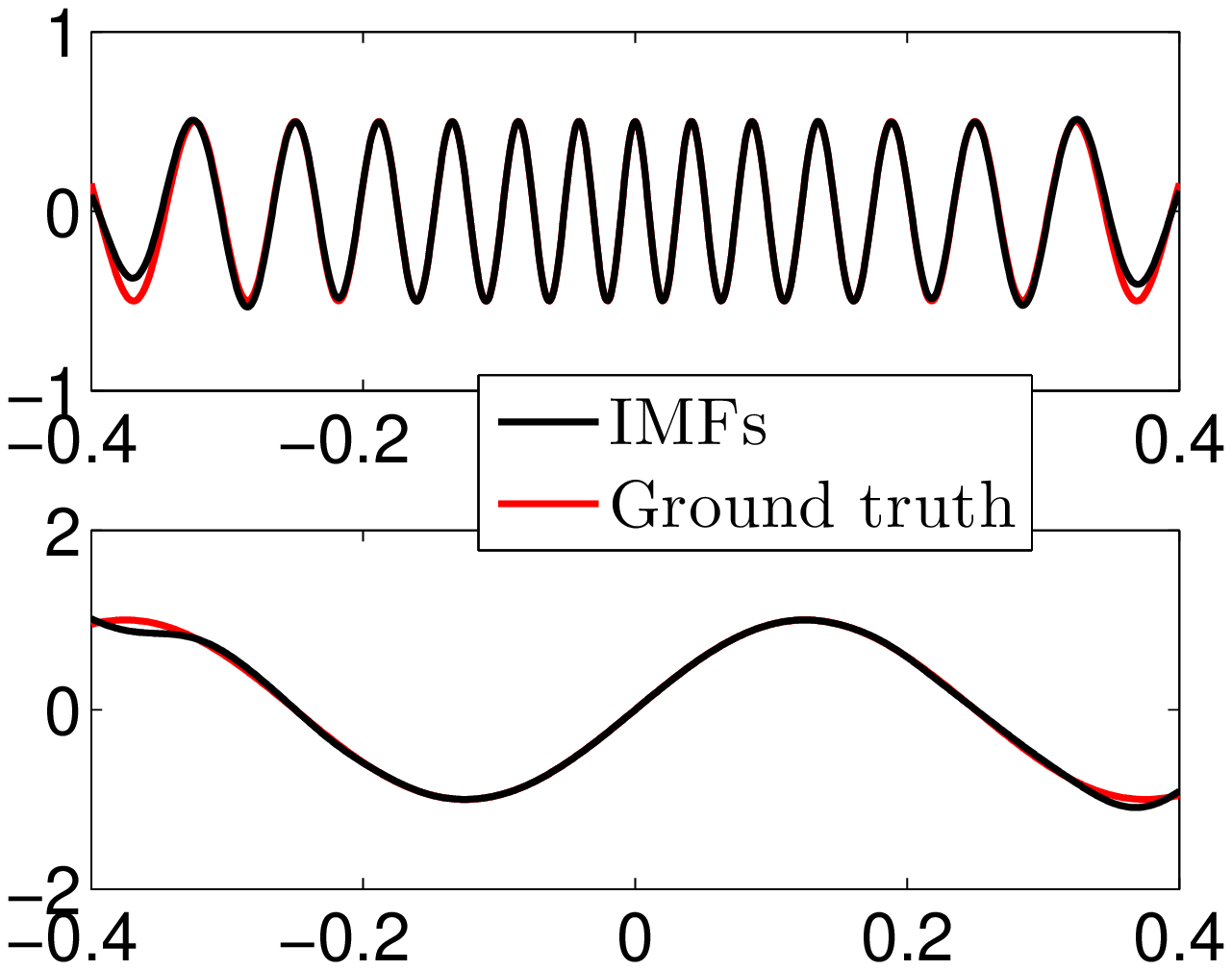}
                \caption{}
                \label{fig9-2}
        \end{subfigure}
        \caption{(\subref{fig9-1}) The signal given in (\ref{equ26}). (\subref{fig9-2}) The components obtained from the IF decomposition.}\label{fig9}
\end{figure}

\begin{figure}[H]
        \begin{subfigure}[b]{0.48\textwidth}
                \centering
                \includegraphics[width=\textwidth]{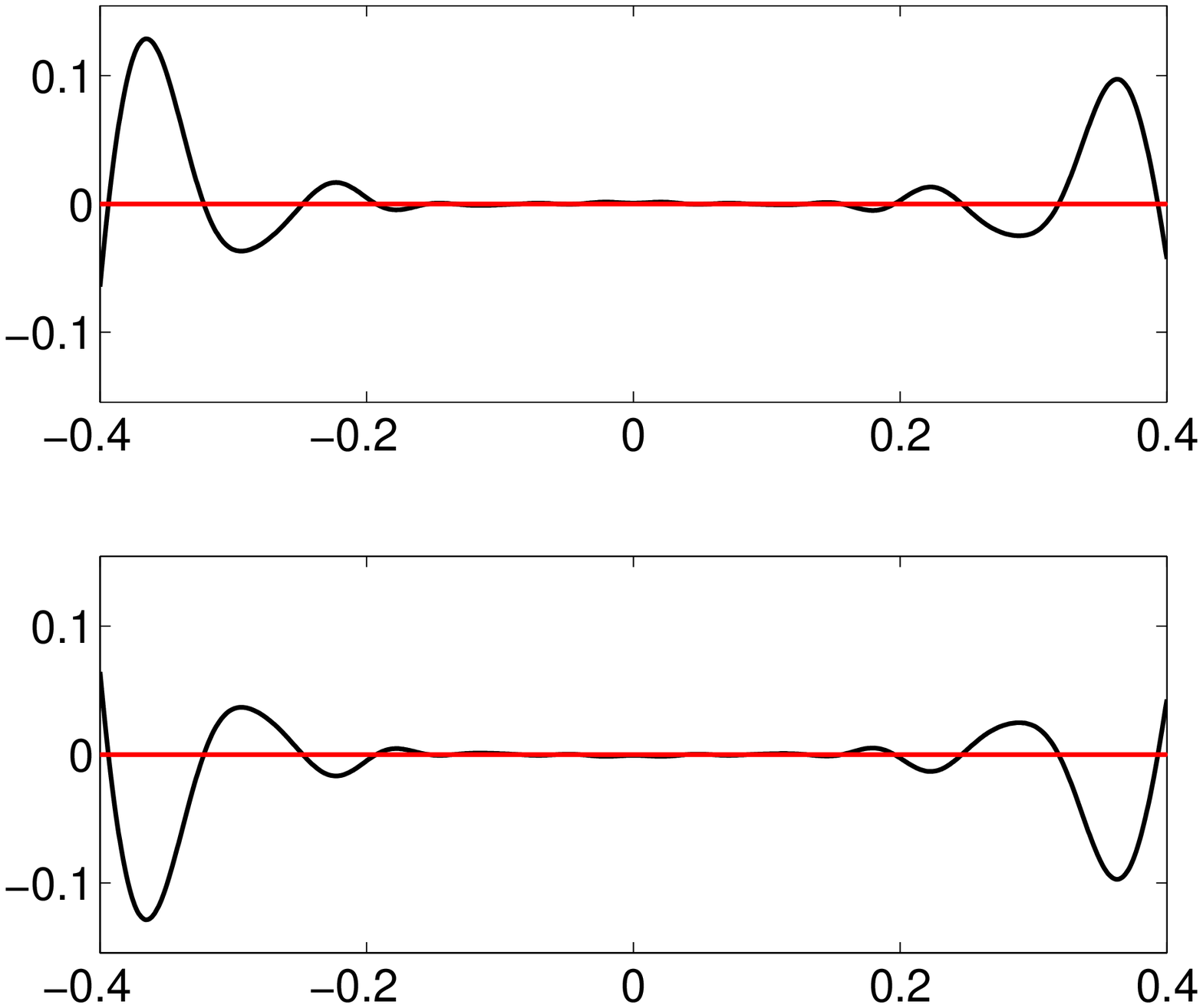}
                \caption{}
                \label{fig9-4}
        \end{subfigure}~
        \begin{subfigure}[b]{0.48\textwidth}
                \includegraphics[width=\textwidth]{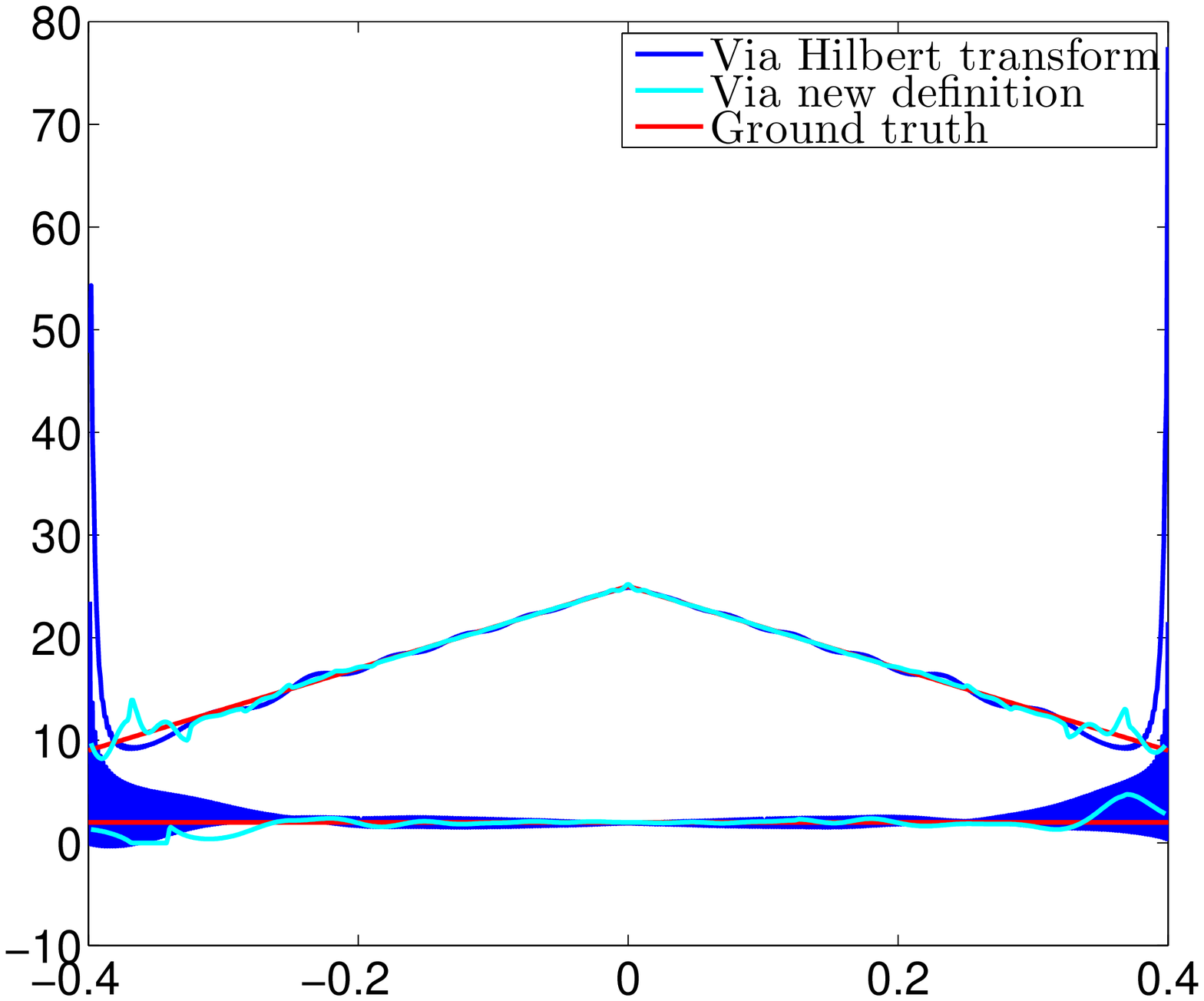}
                \caption{}
                \label{fig9-3}
        \end{subfigure}
        \caption{(\subref{fig9-4}) Differences between IMFs and the ground truth components. (\subref{fig9-3}) The instantaneous frequency for the two components depicted in Figure \ref{fig9-2}.  }\label{fig9bis}
\end{figure}

\noindent \textbf{Example 3}\label{ex3} This time we study the highly non--stationary signal shown in Figure \ref{fig30b}
obtained by adding together the following non-stationary signals $f_1(x)$ and $f_2(x)$, plotted in Figure \ref{fig30a}, and then shifted upward of one.

 \begin{equation}\label{eq:doubleHat}
 \begin{array}{ccc}
   f_1(x) & = & \cos{(-\frac{8}{\pi} x^2-20|x|)} \\
   f_2(x) & = & \cos{(-\frac{8}{\pi} x^2-4|x|)}\\
    f(x) & = & f_1(x) + f_2(x)+1,\quad  x\in[0,2\pi]
 \end{array}
 \end{equation}

  \begin{figure}[H]
        \begin{subfigure}[b]{0.5\textwidth}
                \centering
                \includegraphics[width=\textwidth]{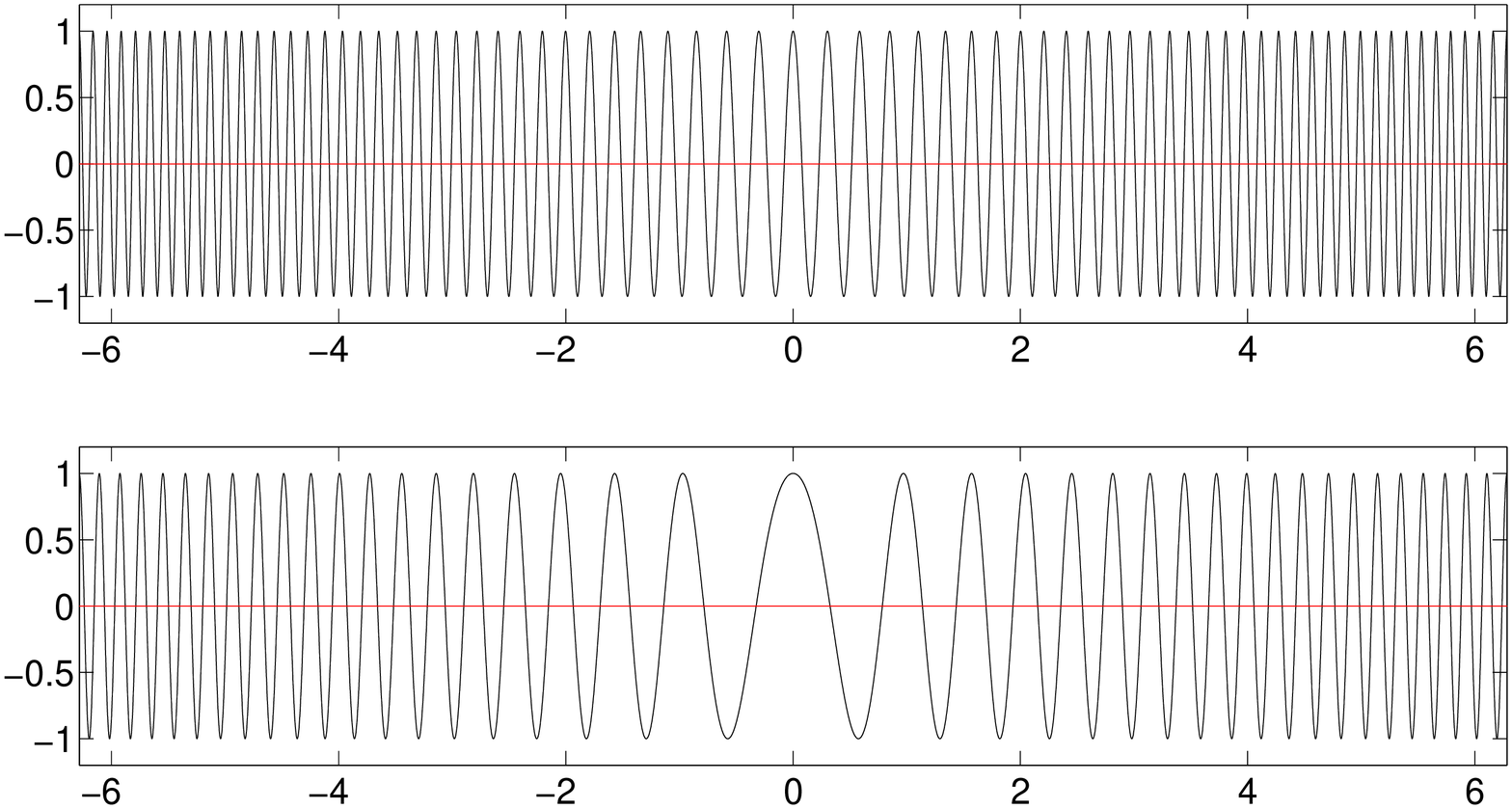}
                \caption{}
                \label{fig30a}
        \end{subfigure}%
        ~ 
          \begin{subfigure}[b]{0.5\textwidth}
                \centering
                \includegraphics[width=\textwidth]{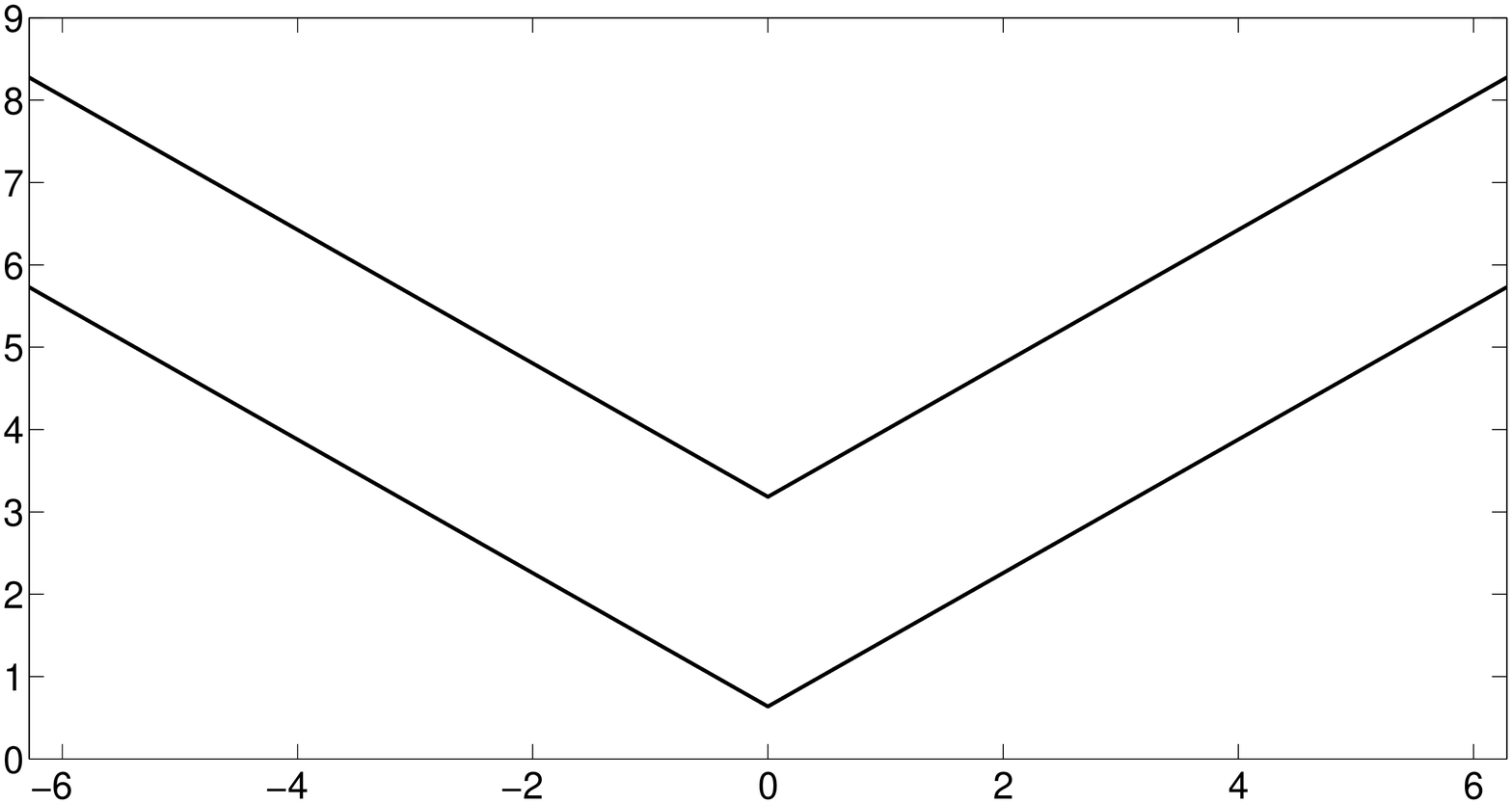}
                \caption{}
                \label{fig31b}
        \end{subfigure}
        \caption{ (\subref{fig30a}) Components $f_1(x)$ and $f_2(x)$. (\subref{fig31b}) Theoretical instantaneous frequencies of $f_1(x)$ and $f_2(x)$.}\label{fig30}
\end{figure}

The two components $f_1(x)$ and $f_2(x)$ are ideal IMFs for the signal $f(x)$. However if we apply the IF technique to decompose the signal what we get is far away from the expected IMFs. In fact, no matter what low pass filter we select and the mask length we choose, the IMFs we get from IF are going to look like something similar to what shown in Figure \ref{fig31a}.

\begin{figure}[H]
\begin{subfigure}[b]{0.5\textwidth}
                \centering
                \includegraphics[width=\textwidth]{./Signal_double_hat.eps}
                \caption{}
                \label{fig30b}
        \end{subfigure}
        ~
        \begin{subfigure}[b]{0.5\textwidth}
                \centering
                \includegraphics[width=\textwidth]{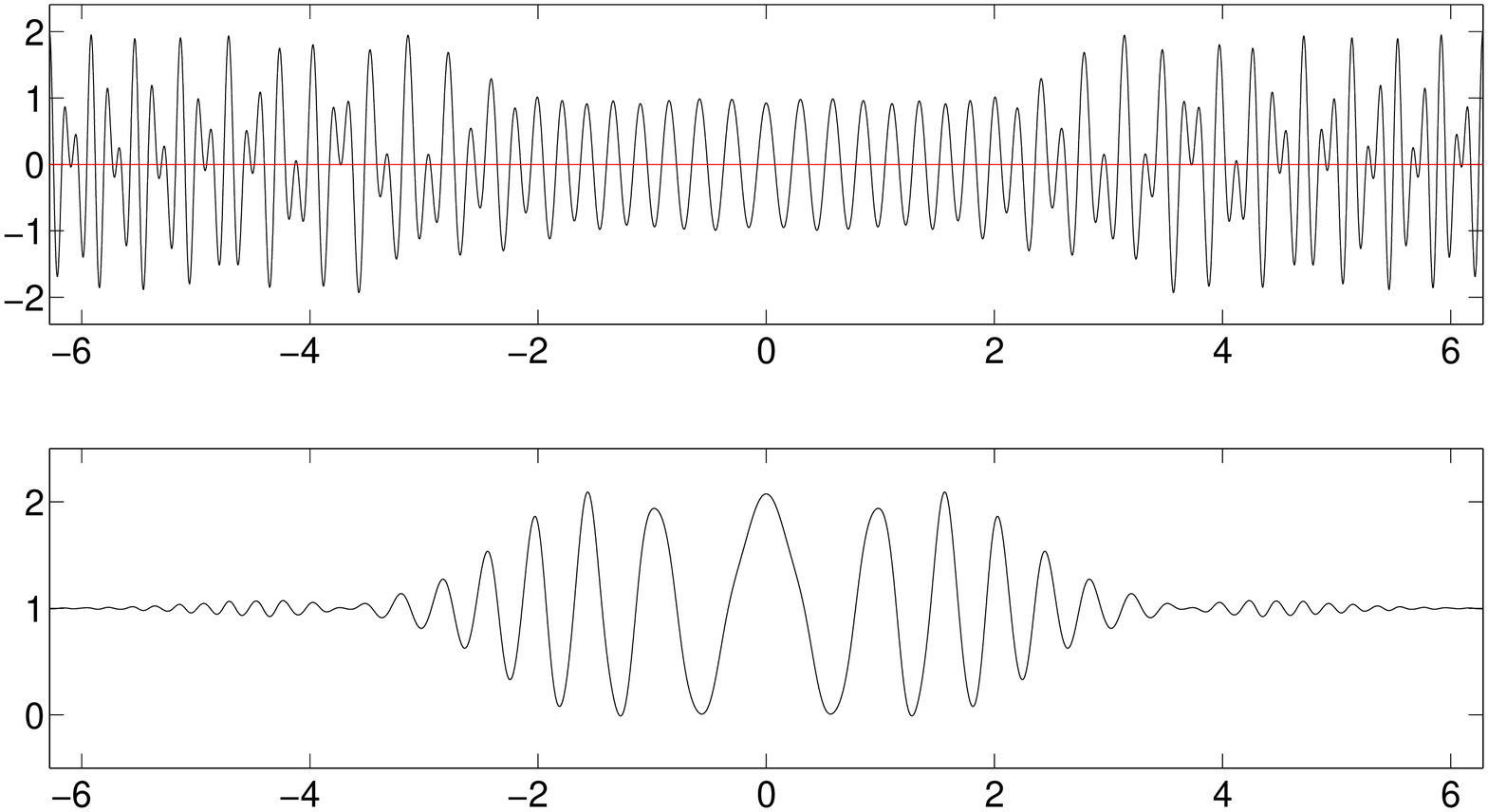}
                \caption{}
                \label{fig31a}
        \end{subfigure}
        \caption{(\subref{fig30b}) Signal $f(x)$ given in (\ref{eq:doubleHat}). (\subref{fig31a}) IF decomposition.}\label{fig31}
\end{figure}

To explain this behavior we take advantage of the fact that the components of the signal $f(x)$ are known a priori, so we can analyze them. Following Section \ref{sec:InstFreq}, we compute the instantaneous frequencies of $f_1$ and $f_2$. The result is plotted in Figure \ref{fig31b}.
We observe that these instantaneous frequencies are well separated at each instant of time, however their ranges overlap. This is the very reason why the IF method fails to decompose this signal: IF can only separate components whose instantaneous frequencies have well separated ranges.
In fact in IF, to produce a single IMF, we use a fixed mask length. Hence the algorithm can produce only an IMF that contains all the components of a signal whose instantaneous frequencies are inside a well defined interval, regardless if these components were originated by different signals or a single one.

One way to decompose the signal $f(x)$ to produce the components $f_1$ and $f_2$ is to use a technique that allows to adaptively change the mask length from point to point. This is done by the ALIF method.

Using the mask length plotted in Figure \ref{fig32b}, which has been derived from the signal itself as described in Section \ref{sec:ALIF}, the ALIF algorithm produces the decomposition shown in Figure \ref{fig32a}. It is important to point out here that to produce such a decomposition the ALIF technique does not make use of any a priori knowledge on the given signal.

\begin{figure}[H]
\begin{subfigure}[b]{0.5\textwidth}
                \centering
                \includegraphics[width=\textwidth]{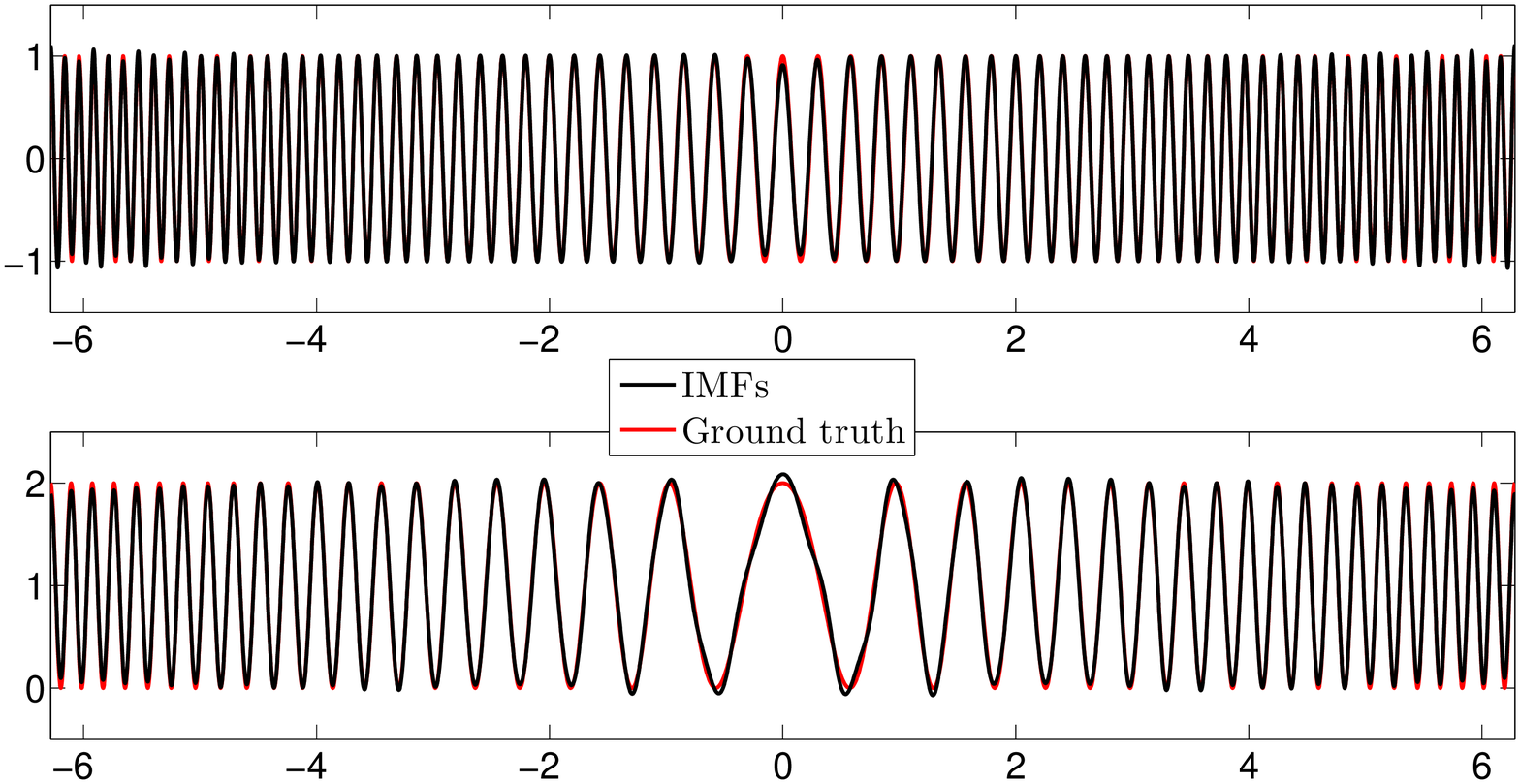}
                \caption{}
                \label{fig32a}
        \end{subfigure}
        ~
        \begin{subfigure}[b]{0.5\textwidth}
                \centering
                \includegraphics[width=\textwidth]{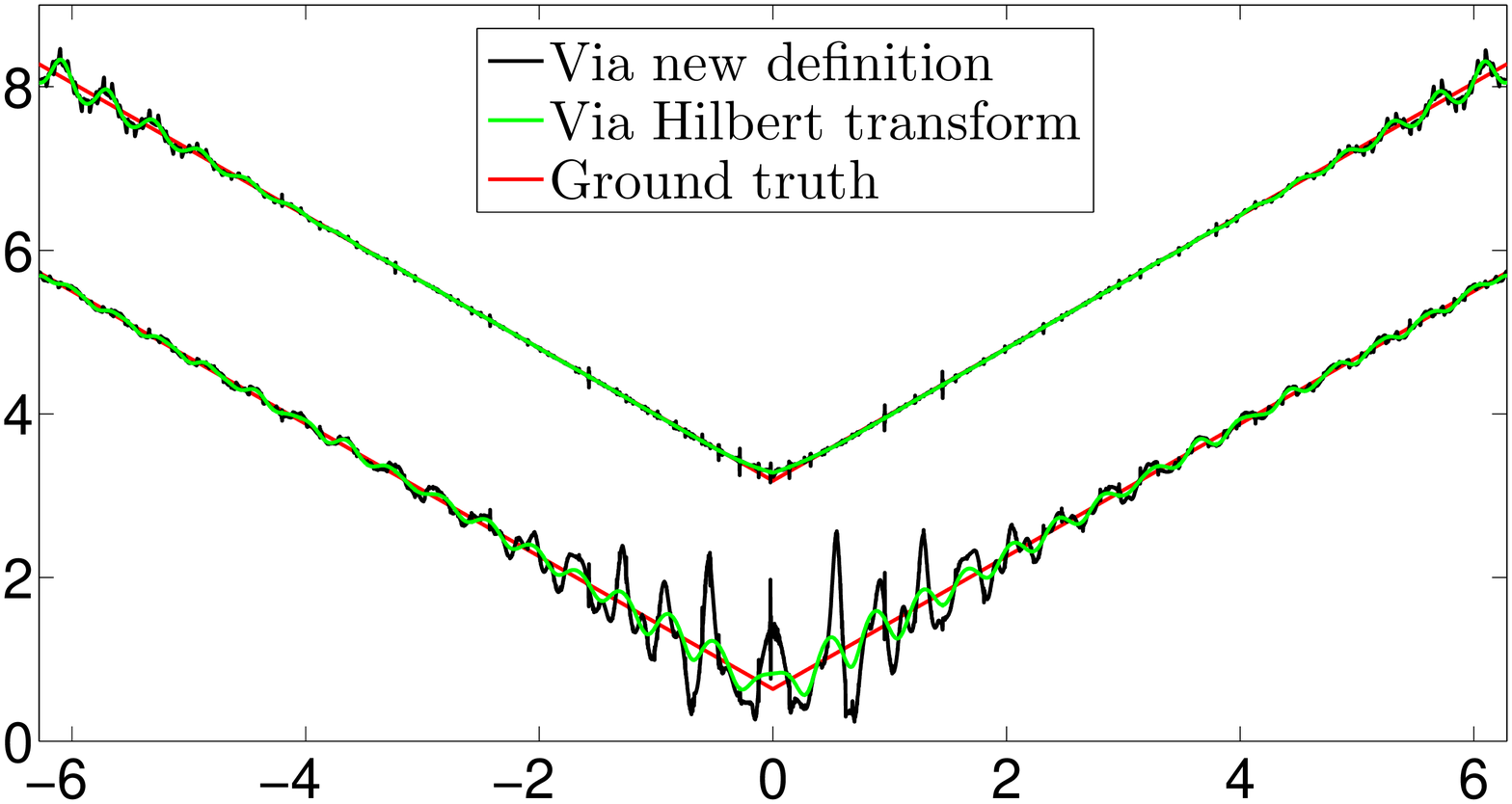}
                \caption{}
                \label{fig32c}
        \end{subfigure}
        \caption{(\subref{fig32a}) ALIF decomposition. (\subref{fig32c}) Instantaneous frequencies of the IMFs.}
\end{figure}

This example allows to show the ability of ALIF with respect to IF method in decomposing properly a non--stationary signal simply based on the local information from the signal itself.

Regarding the instantaneous frequency computation, in Figure \ref{fig32c} we compare the values computed using the standard technique, based on Hilbert Transform \cite{huang1998empirical}, the instantaneous frequency obtained using the new definition proposed in Section \ref{sec:InstFreq} and the known ground truth. To compute the instantaneous frequency of the second component we first subtract from it its average value 1.

Looking at this picture we observe that, for the second IMF, the new method is producing a more oscillatory instantaneous frequency than the one computed using the traditional definition. This is because such signal has intra--wave modulation, which is due to the presence of concurrent oscillations with different frequencies \cite{huang1999new}. This is evident if we look at Figure \ref{fig32a} where the IMFs are compared with the ground truth. 
In fact, looking at the central section of the bottom curves, the black one, that represents the remainder, crosses several times the red one, which depicts the ground truth.

Thanks to the scaling of the signal and its derivative, ref. \eqref{equ10} and \eqref{equ11}, the new definition of instantaneous frequency proves to be more sensitive than the traditional one to intra--wave modulation allowing to identify time intervals where this phenomenon happens.

For some applications such sensitiveness proves to be extremely valuable, for instance in the decomposition of the solution of the Duffing equation as well as other classical nonlinear systems \cite{huang1998empirical}.

In some other cases, like the example under study, such intra--wave modulations are simply related to small amplitude oscillations due to the numerical calculations. In this case we may want to reduce the contribution of such small amplitude oscillations. One way is, by means of the ALIF technique itself, to filter out such small amplitude oscillations from the signal and its derivative, as shown in Figure~\ref{fig:DoubleHatNewFreq}.

\begin{figure}[H]
\centering
                \includegraphics[width=0.5\textwidth]{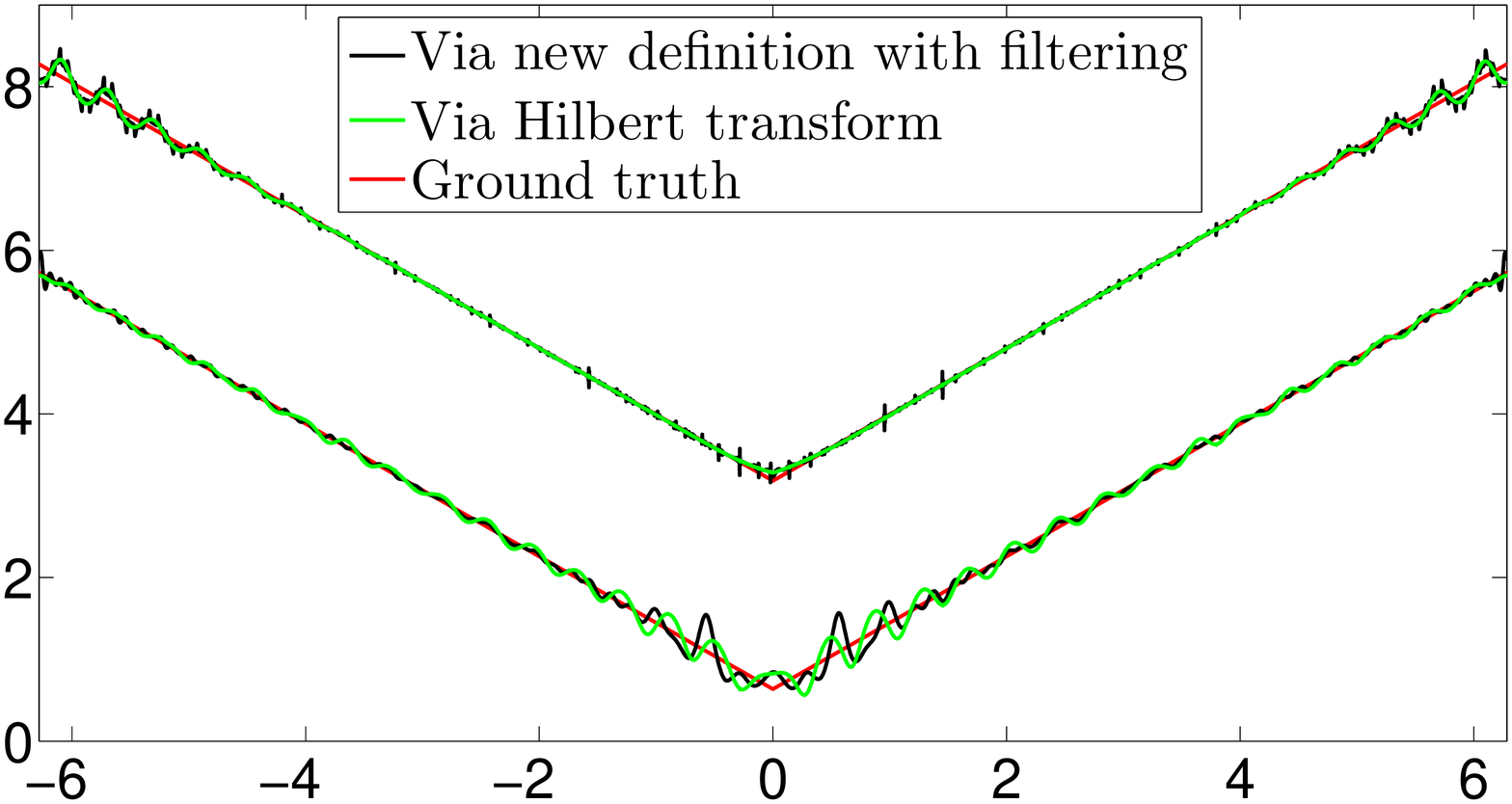}
                \caption{Instantaneous frequencies of the IMFs computed using the new definition coupled with the ALIF technique to remove small scale oscillations.}
                \label{fig:DoubleHatNewFreq}
\end{figure}

\noindent \textbf{Example 4} We consider now a signal and its perturbation with white noise to show the stability of the IF algorithm and to introduce the identifiability problem \cite{Wu2014Ident}. The signals are
\begin{equation}\label{equ25_0}
f_0(x)=\sin{(\pi x)}+\sin{(4\pi x)},\phantom{+n_2(x)}\qquad  x\in[0,5]
\end{equation}
\begin{equation}\label{equ25}
 f_i(x)=\sin{(\pi x)}+\sin{(4\pi x)}+n_i(x),\qquad x\in[0,5],\quad i=1,\, 2
\end{equation}
 where $n_1(x)$  and $n_2(x)$ are white noise:  $n_1(x) \sim  N(0,0.01)$  and $n_2(x) \sim  N(0,1)$ for every $x\in\R$. We apply the IF algorithm on  $f_0(x)$, $f_1(x)$ and $f_2(x)$.
From Figure \ref{fig8}, we see that $f_0(x)$ is separated into two IMFs, which correspond to the components $\sin{(4\pi x)}$ and $\sin{(\pi x)}$ respectively.
$f_1(x)$ is decomposed into seven IMFs as shown in Figure \ref{fig88}. The first few IMFs come from the impact of the noise and the last two IMFs reveal the two sinusoidal functions $\sin{(4\pi x)}$ and $\sin{(\pi x)}$. $f_2(x)$ is decomposed into nine IMFs and only the last seven are shown in Figure \ref{fig888}. As for $f_1(x)$, the first few IMFs come from the impact of the noise and the last two IMFs reveal the two sinusoidal functions $\sin{(4\pi x)}$ and $\sin{(\pi x)}$.
 Via this example, the IF algorithm is shown to be robust to white noise. What is more important, denoising is automatically achieved by getting rid of the first few highly oscillatory IMFs.\\
 \begin{figure}[H]
 \begin{center}
        \begin{subfigure}[b]{0.3\textwidth}
                \centering
                \includegraphics[width=\textwidth]{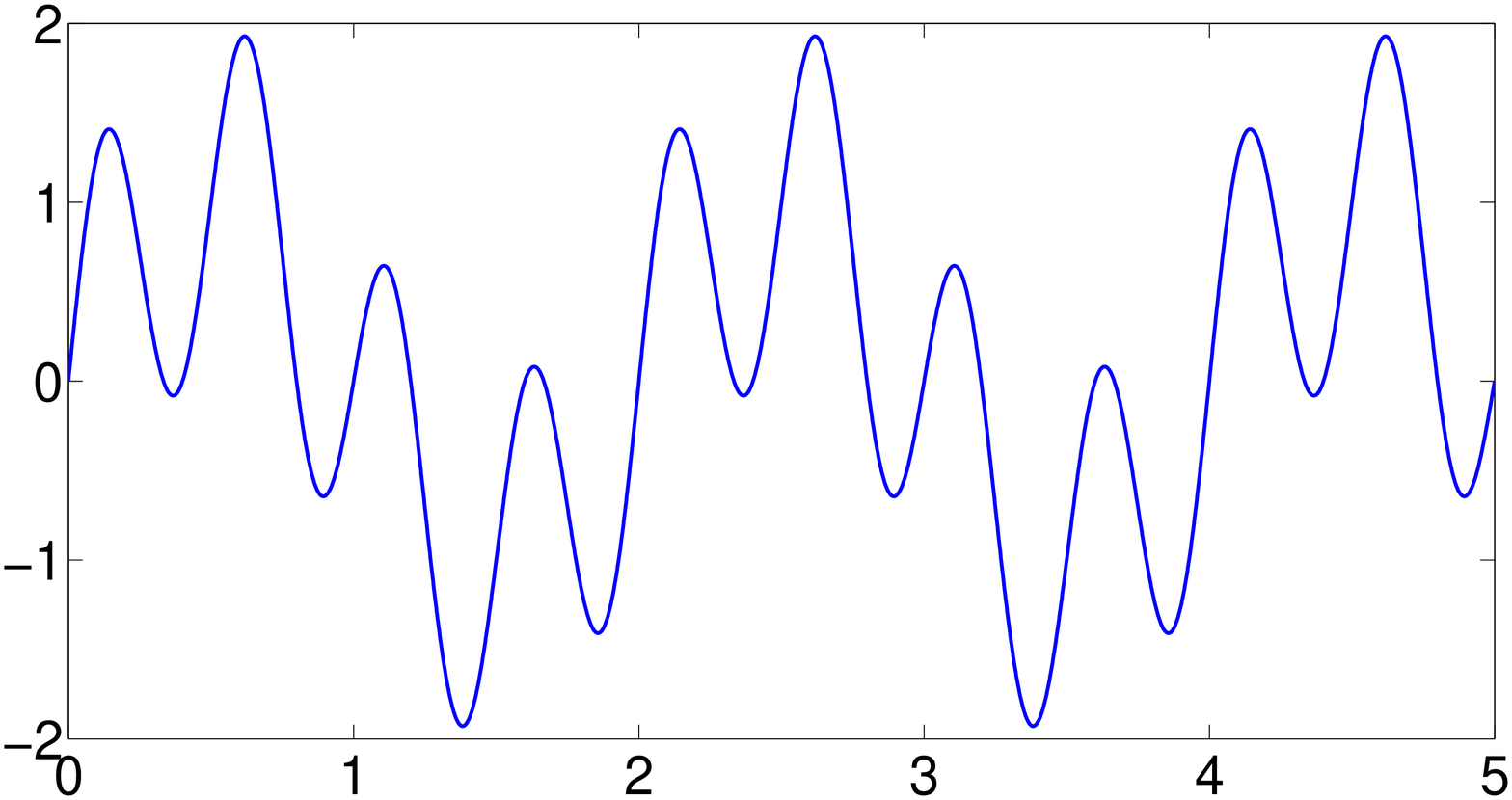}
                \caption{}
                \label{fig8-1}
        \end{subfigure}%
        ~
        \begin{subfigure}[b]{0.3\textwidth}
                \centering
                \includegraphics[width=\textwidth]{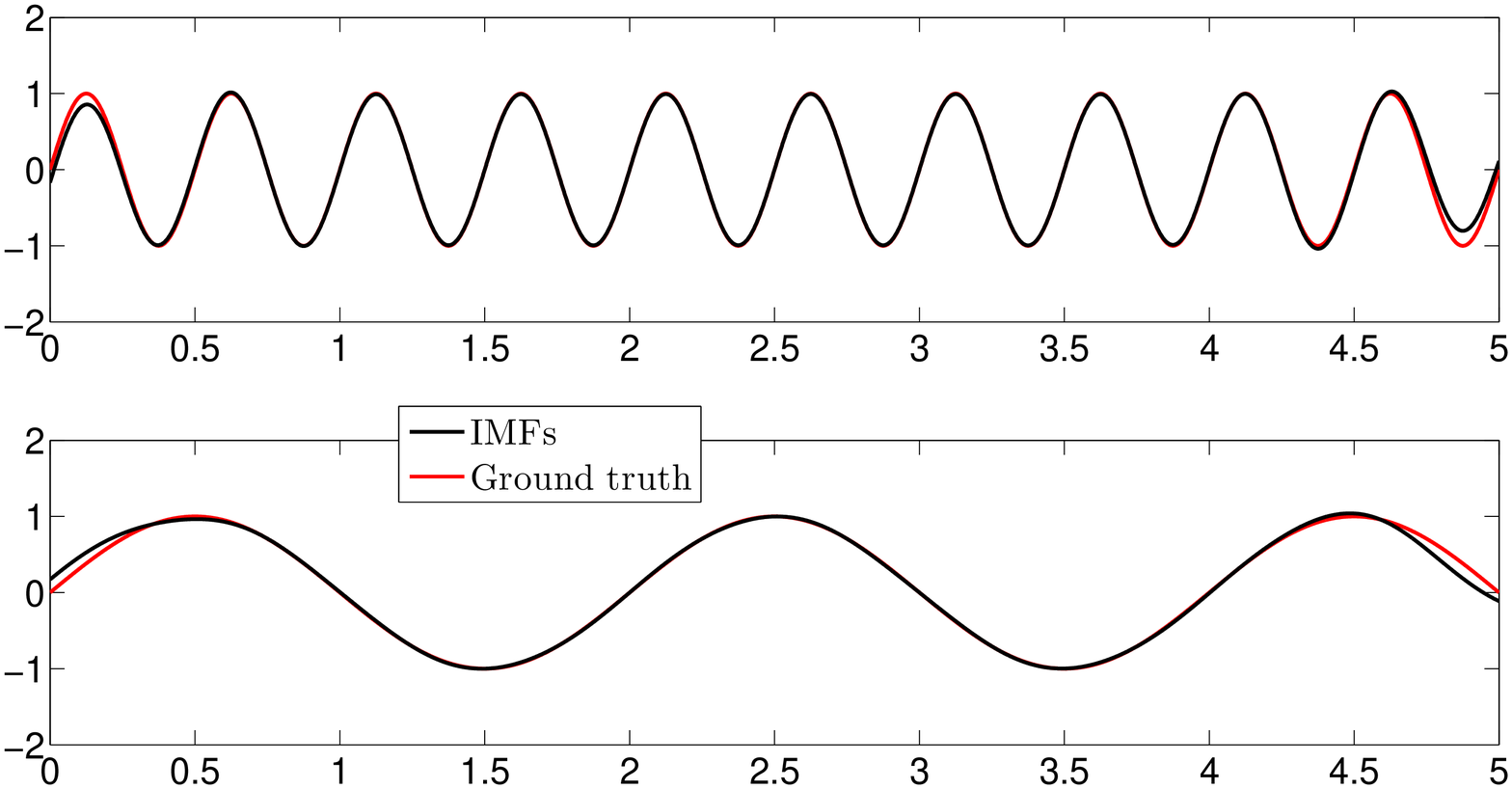}
                \caption{}
                \label{fig8-2}
        \end{subfigure}
        ~
        \begin{subfigure}[b]{0.3\textwidth}
                \centering
                \includegraphics[width=\textwidth]{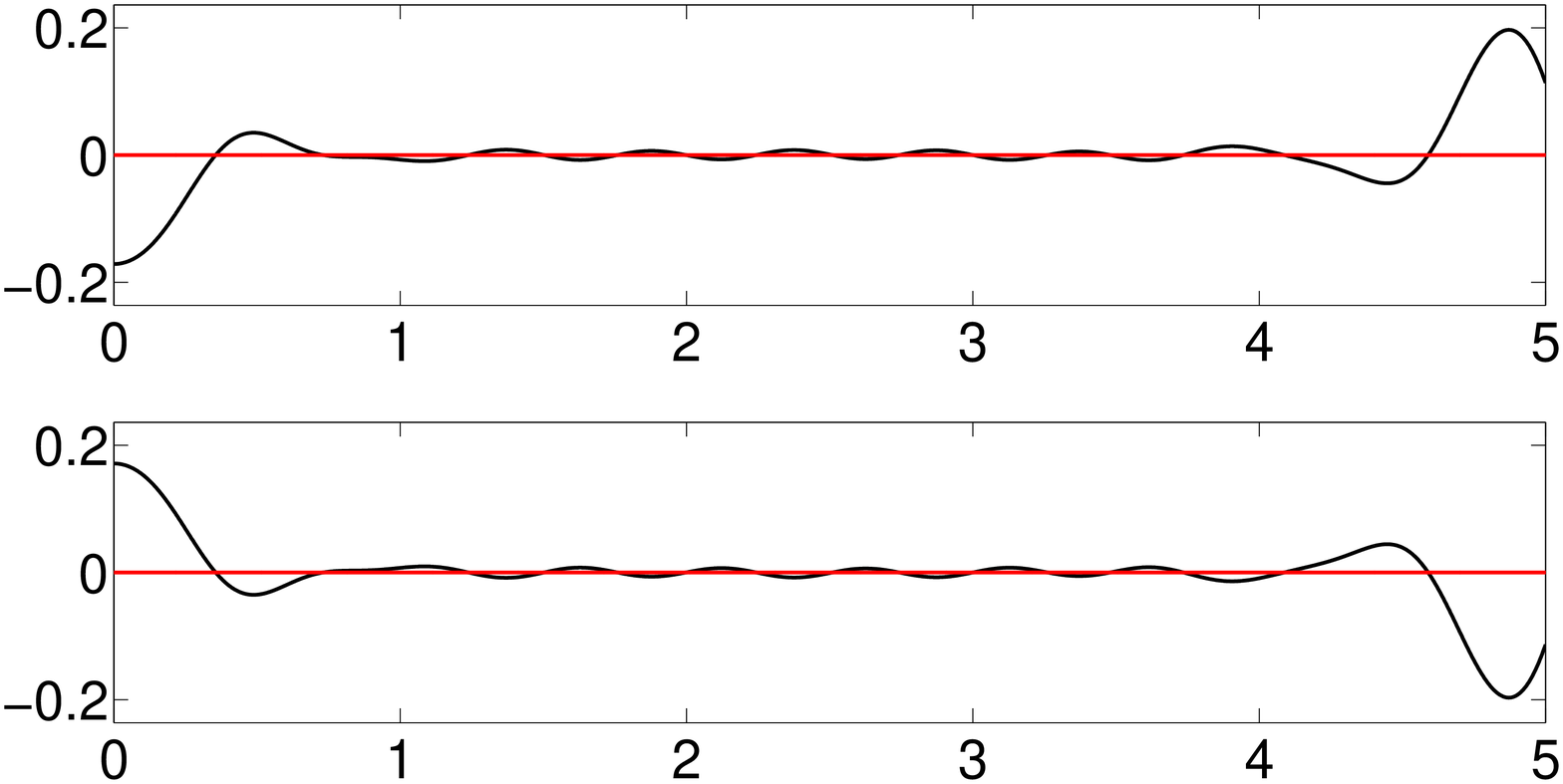}
                \caption{}
                \label{fig8-3}
        \end{subfigure}
\end{center}
        \caption{ (\subref{fig8-1}) The signal $f_0(x)$ given in (\ref{equ25_0}). (\subref{fig8-2}) The components in the IF decomposition. (\subref{fig8-3}) Error values. }\label{fig8}
\end{figure}

 \begin{figure}[H]
 \begin{center}
        \begin{subfigure}[b]{0.3\textwidth}
                \centering
                \includegraphics[width=\textwidth]{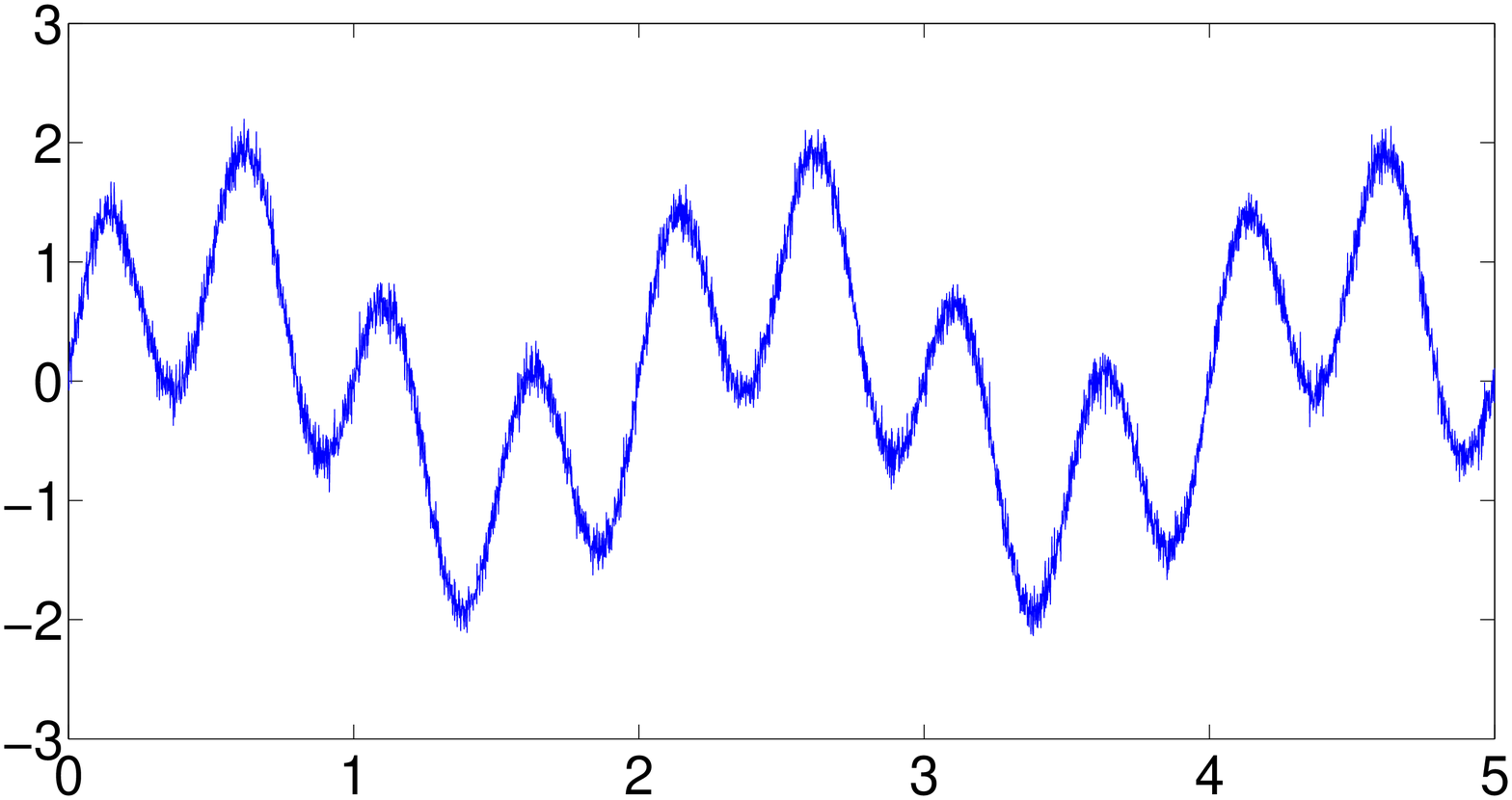}
                \caption{}
                \label{fig88-1}
        \end{subfigure}%
        ~
        \begin{subfigure}[b]{0.3\textwidth}
                \centering
                \includegraphics[width=\textwidth]{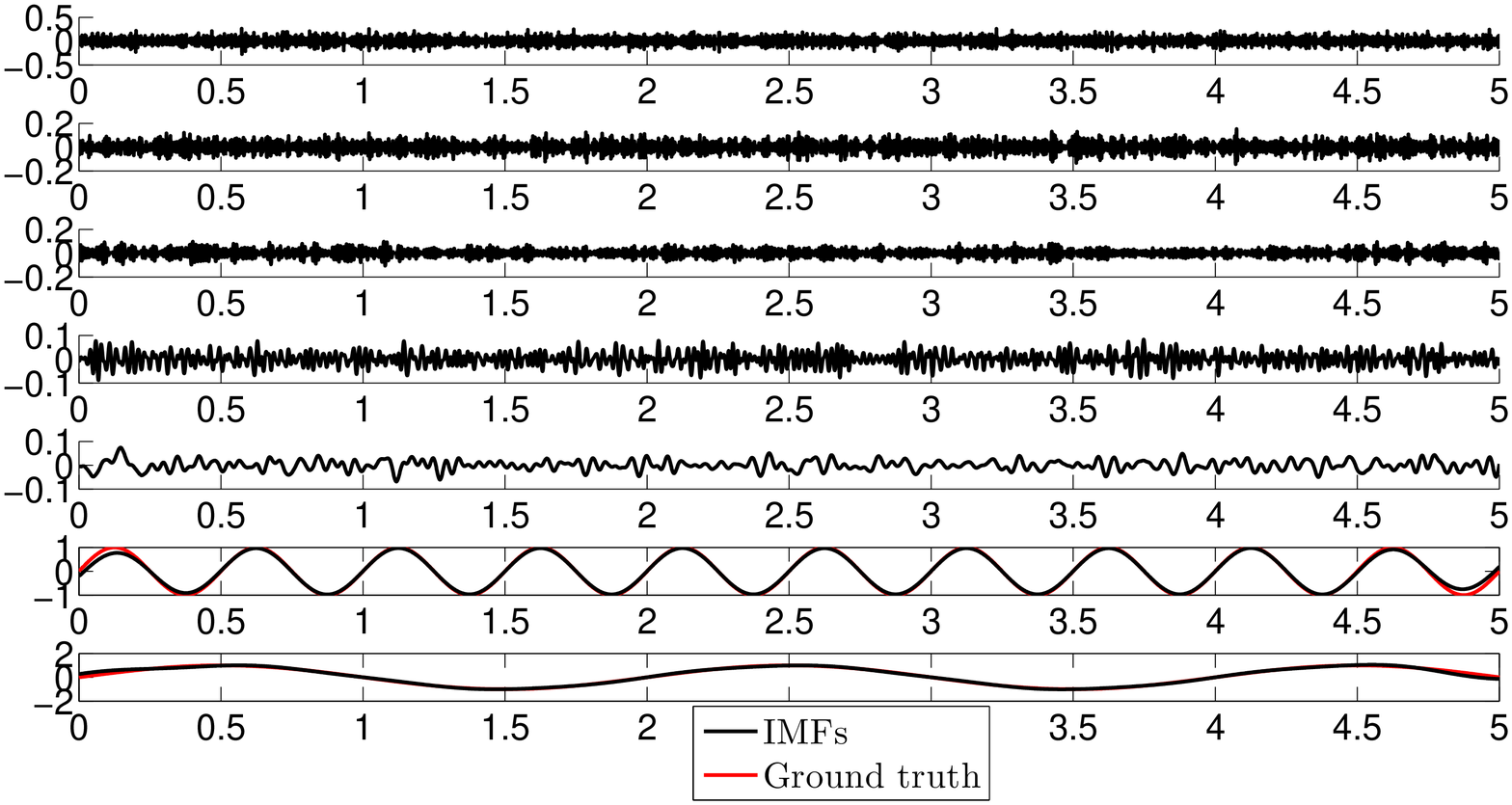}
                \caption{}
                \label{fig88-2}
        \end{subfigure}
        ~
        \begin{subfigure}[b]{0.3\textwidth}
                \centering
                \includegraphics[width=\textwidth]{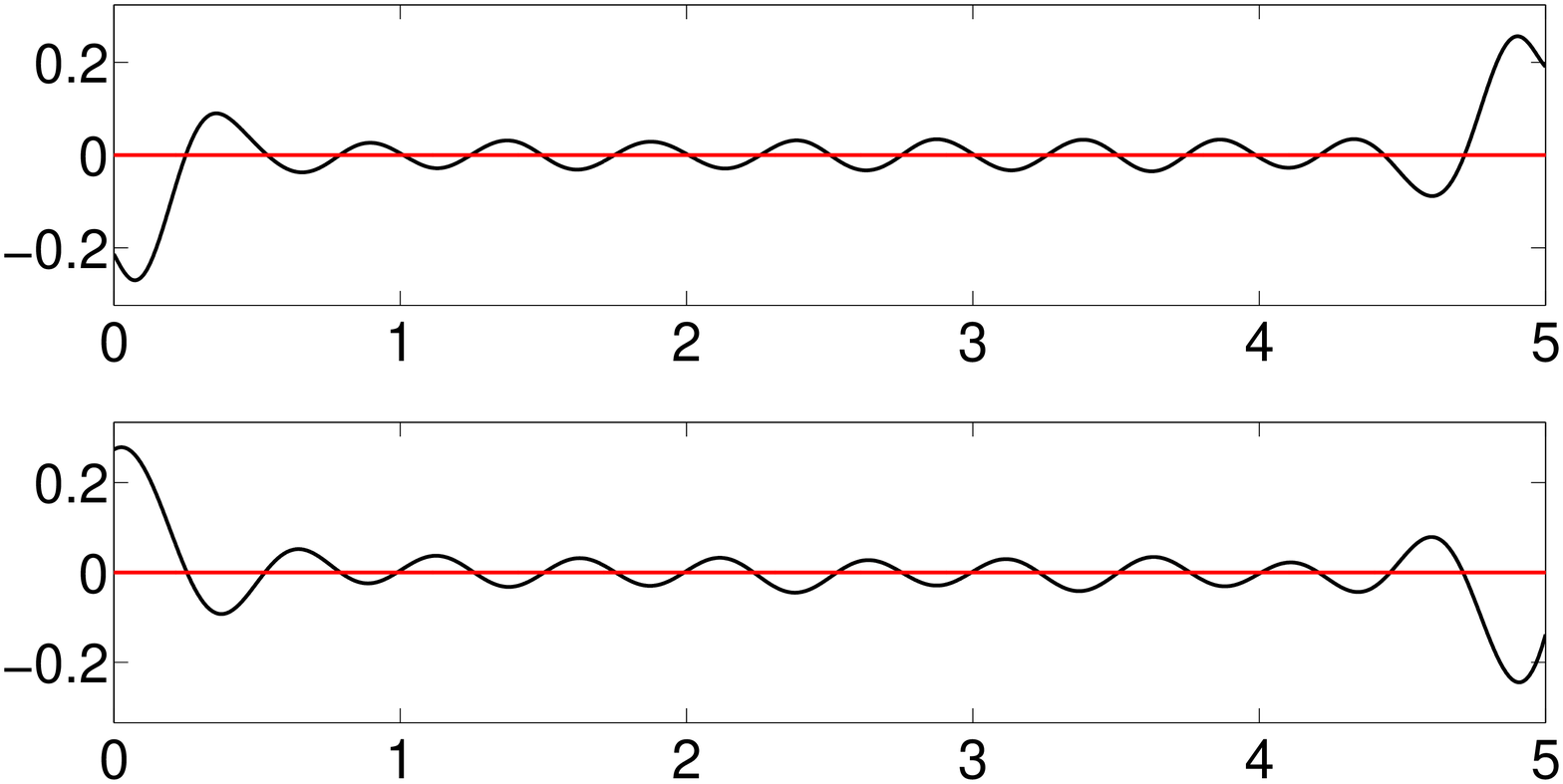}
                \caption{}
                \label{fig88-3}
        \end{subfigure}
\end{center}
        \caption{ (\subref{fig88-1}) The signal $f_1(x)$ given in (\ref{equ25}). (\subref{fig88-2})  The IMFs produced using IF. (\subref{fig88-3})  Error values computed as differences between the ground truth components and the last two IMFs in the decomposition. }\label{fig88}
\end{figure}

 \begin{figure}[H]
 \begin{center}
        \begin{subfigure}[b]{0.3\textwidth}
                \centering
                \includegraphics[width=\textwidth]{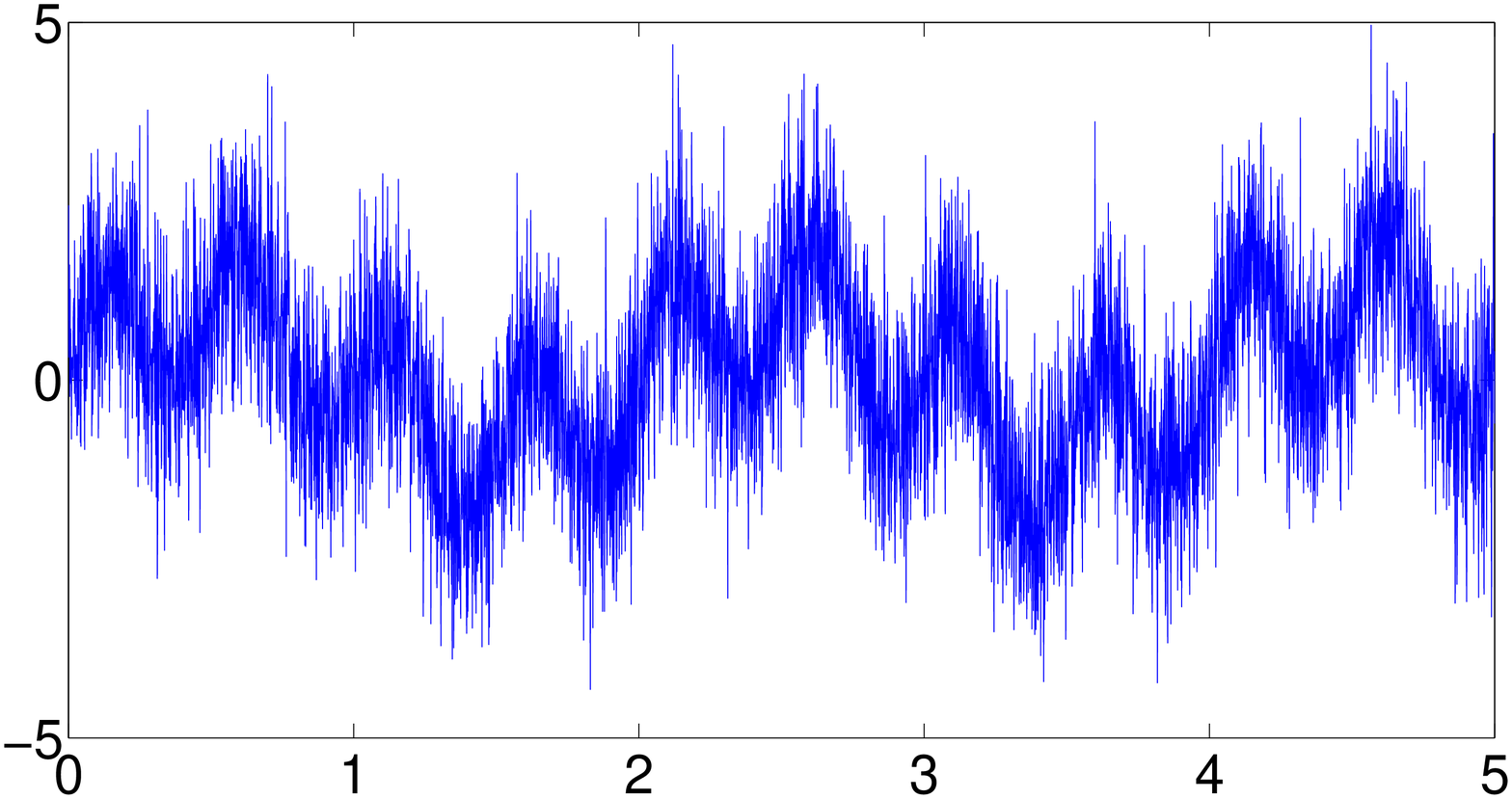}
                \caption{}
                \label{fig888-1}
        \end{subfigure}%
        ~
        \begin{subfigure}[b]{0.3\textwidth}
                \centering
                \includegraphics[width=\textwidth]{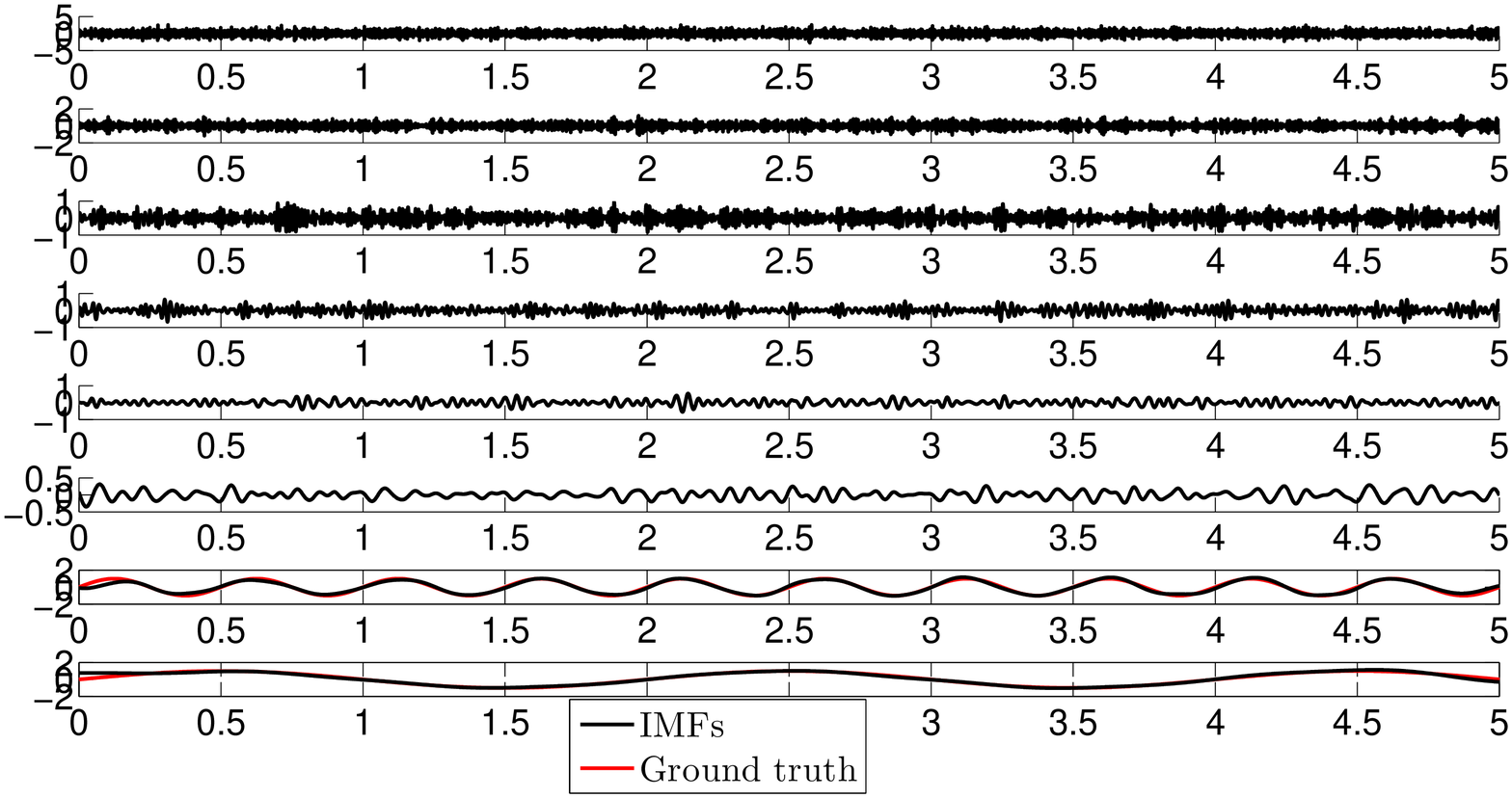}
                \caption{}
                \label{fig888-2}
        \end{subfigure}
        ~
        \begin{subfigure}[b]{0.3\textwidth}
                \centering
                \includegraphics[width=\textwidth]{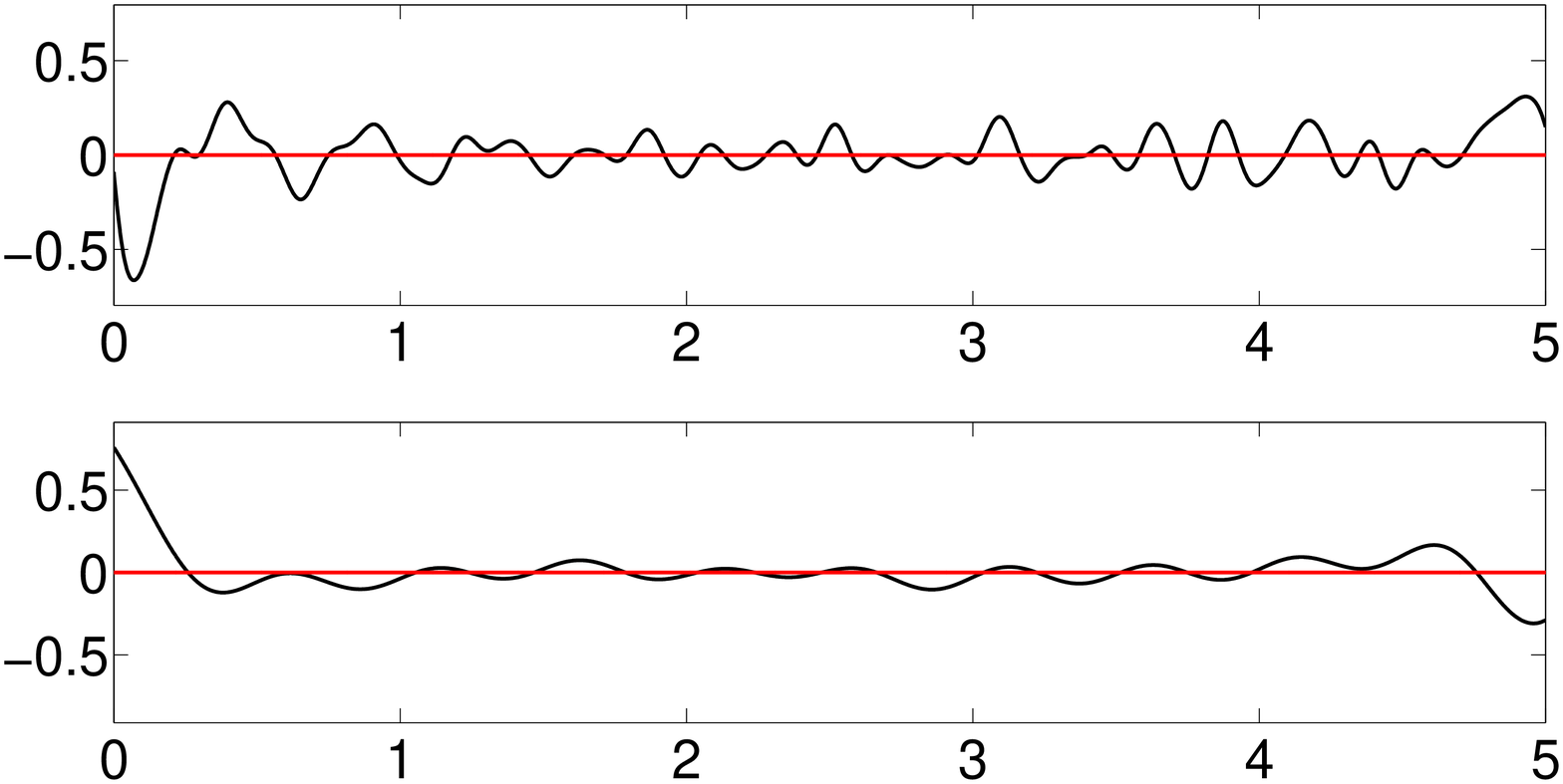}
                \caption{}
                \label{fig888-3}
        \end{subfigure}
\end{center}
        \caption{ (\subref{fig888-1})  The signal $f_2(x)$ given in (\ref{equ25}). (\subref{fig888-2}) The IF decomposition. (\subref{fig888-3})  Error values. }\label{fig888}
\end{figure}

Another important issue is that, given a signal, there might be more than one valid representation of it. For example, a purely harmonic function can be represented also as a time varying amplitude and time varying phase function. Therefore is very natural to ask which are all the possible different representations of a signal $f(x)$ that can be potentially obtained using the chosen decomposition method, and how much can they differ from each other. In the literature this is known as \emph{identifiability problem} \cite{genton2007statistical,daubechies2011synchrosqueezed,Wu2014Ident}. In \cite{Wu2014Ident} Chen at al. analyze the problem from the Synchrosqueezed Wavelet Transforms prospective, in particular they define a functional class $\mathcal A$ of single--component periodic functions having `slowly varying amplitude modulation and instantaneous frequency', as well as a class $\mathcal C$ of superpositions of single--component periodic functions in $\mathcal A$ whose instantaneous frequencies are different each other at every instant of time. For more details on these two classes we refer the reader to \cite{Wu2014Ident}. In that paper the authors prove that, given a signal in $\mathcal C$, which is a superposition of single--component `slowly varying' periodic functions, and assuming there is another representation of it which is also in $\mathcal C$, the difference of the two representations is bounded by a constant which depends only on the characteristics of the signal itself.

What happen when we consider the IF method? What we know from the results of the current example is that, given the signal $f_0(x)$, which belongs to the class $\mathcal C$, we can decompose it with high accuracy using the IF method, even if we perturb it with white noise.

In the next example we show that the IF method can decompose and identify with good accuracy also the components of a signal which does not belong to the class $\mathcal C$.

\noindent \textbf{Example 5} We consider the signal $f(x)$ which does not belong to the class $\mathcal C$, described in the previous example. $f(x)$ is defined as follows
\begin{eqnarray*}
  f_1(x) &=& (\sin(4\pi x)+1.5)\cos(50 \pi x) \\
  f_2(x) &=& \left(5\sin\left(2\pi\frac{x+1}{6}+\pi\right)+5.6\right)\sin(2\pi (10x+0.03\cos(40\pi x))) \\
  f_3(x) &=& (2 \cos(1.4\pi x)+5)\sin(4\pi x)
\end{eqnarray*}
\begin{equation}\label{eq:Ex9}
f(x)=f_1(x)+f_2(x)+f_3(x)
     ,\qquad  x\in[-1,\ 2]
\end{equation}
For this signal all the theorems proved by Chen et al. in \cite{Wu2014Ident} do not apply anymore.

Nevertheless, we can apply the IF method to $f(x)$ and produce a remarkably good decomposition, as shown in Figure \ref{figEx9-2}. The differences between the expected components and the one obtained via IF are plotted in Figure \ref{figEx9bis}.

\begin{figure}[H]
 \begin{center}
        \begin{subfigure}[b]{0.48\textwidth}
                \centering
                \includegraphics[width=\textwidth]{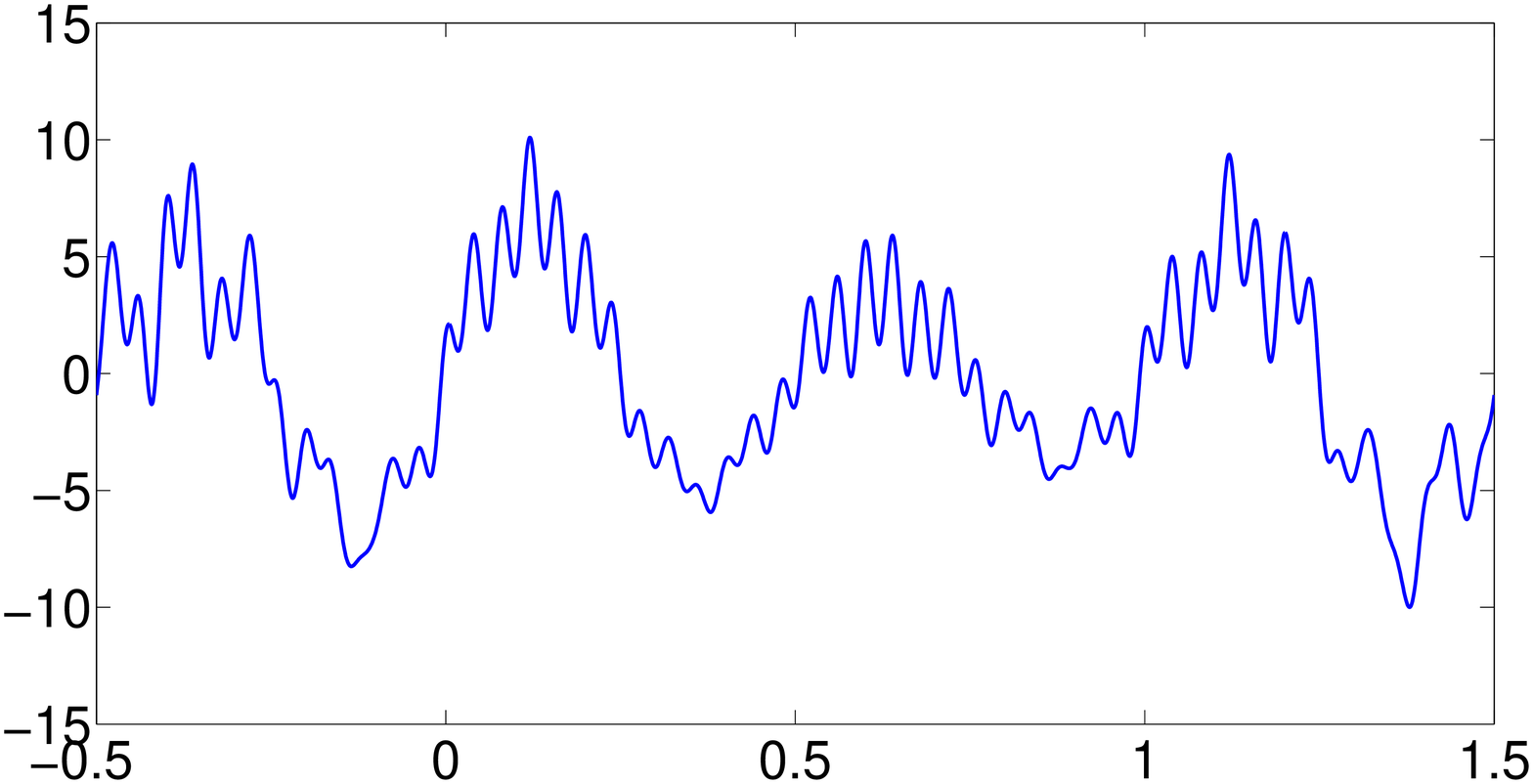}
                \caption{}
                \label{figEx9-1}
        \end{subfigure}%
        ~
        \begin{subfigure}[b]{0.48\textwidth}
                \centering
                \includegraphics[width=\textwidth]{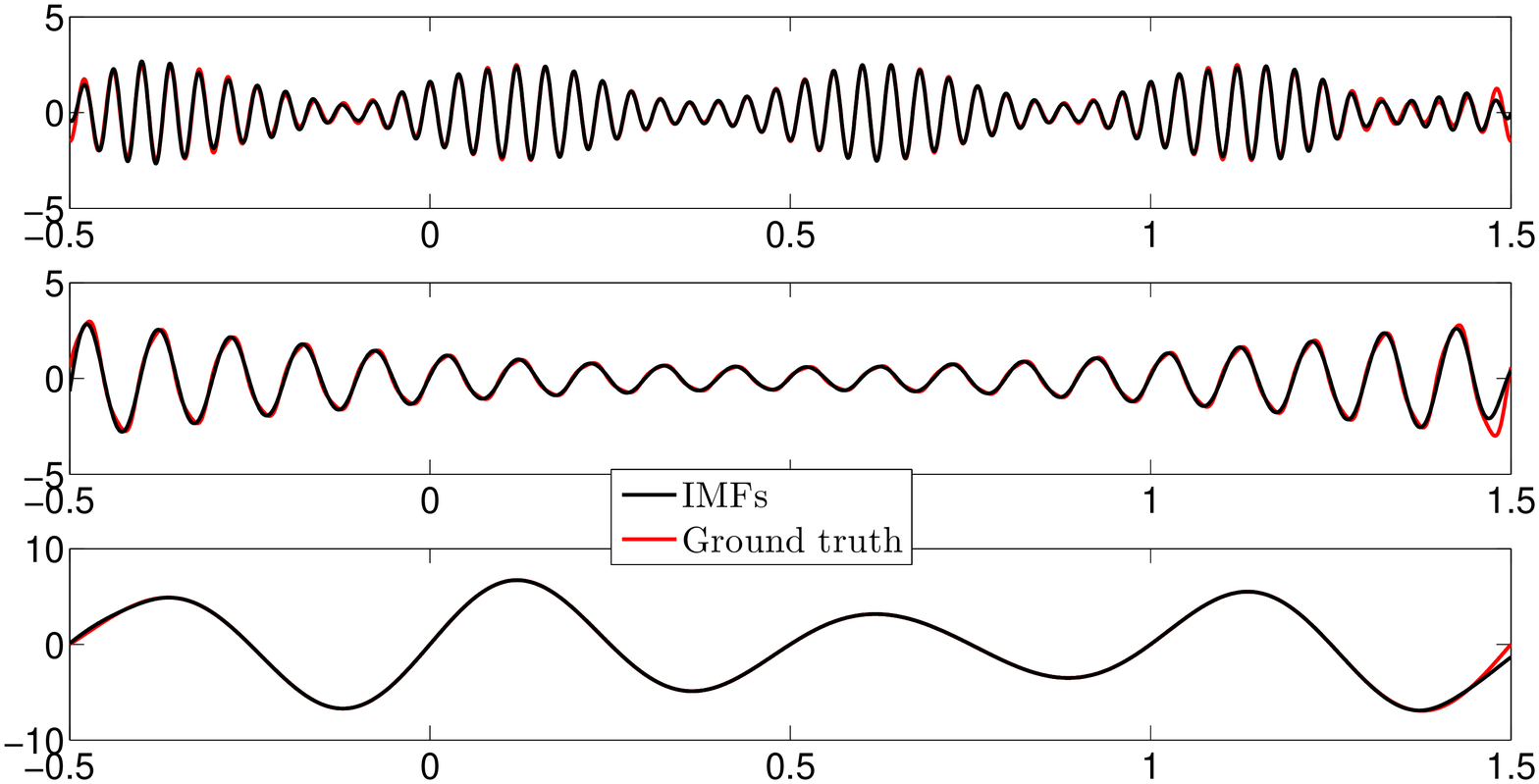}
                \caption{}
                \label{figEx9-2}
        \end{subfigure}
 \end{center}
        \caption{ (\subref{figEx9-1})  The signal given in (\ref{eq:Ex9}). (\subref{figEx9-2}) The IF decomposition.  }\label{figEx9}
\end{figure}
\begin{figure}[H]
 \begin{center}
        \begin{subfigure}[b]{0.48\textwidth}
                \centering
                \includegraphics[width=\textwidth]{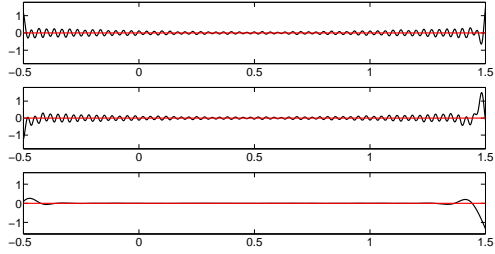}
        \end{subfigure}
\end{center}
        \caption{ The differences between the computed IMFs and the known ground truth.}\label{figEx9bis}
\end{figure}

From this example and the previous ones it is evident that the IF technique can identify functions in broader classes than the ones defined by Chen et al. in \cite{Wu2014Ident}. However it is not yet clear at this point what should be the properties of such classes.

Furthermore all the obtained results suggest that, in general, when we apply the IF method we can retrieve the components contained in a signal with good accuracy. It remains an open problem to derive bounds on the expected accuracy in the decomposition of a given signal. We plan to tackle all these problems in the next future.

\noindent \textbf{Example 6} This time we analyze an artificial signal and its perturbation with white noise to show the stability of the ALIF method. We consider
\begin{equation}\label{eq:ALIFnoise_no}
f_0(x)=\cos\left[20\cos\left(\frac{x}{10}\right)-4x\right]+\cos\left[20\cos\left(\frac{x}{10}\right)-7x\right]+1,\phantom{+\alpha_i n(x)}\qquad  x\in[0,20\pi]
\end{equation}
plotted in Figure \ref{fig40a}, which belongs to the functional class $\mathcal C$, it is a superposition of single--component periodic functions having slowly varying amplitude modulation and instantaneous frequency and whose instantaneous frequencies are different each other at every instant of time, ref. \cite{Wu2014Ident}. Furthermore we consider
\begin{equation}\label{eq:ALIFnoise}
f_i(x)=\cos\left[20\cos\left(\frac{x}{10}\right)-4x\right]+\cos\left[20\cos\left(\frac{x}{10}\right)-7x\right]+1+\alpha_i n(x), \qquad  x\in[0,20\pi], \; i=1,\,2,
\end{equation}
where $n(x)$ is white noise and $\alpha_i$, $i=1,\,2$, are constants so that the corresponding Signal to Noise Ratio, computed as $\textrm{SNR}=20\log(\|\textrm{signal}\|_2/\|\textrm{noise}\|_2)$, is around 0 and $-10$ respectively.

The IF algorithm can decompose the signal $f_0(x)$, but fails to reproduce the two basic components $g_1(x)=\cos\left[20\cos\left(\frac{x}{10}\right)-4x\right]$ and $g_2(x)=\cos\left[20\cos\left(\frac{x}{10}\right)-7x\right]$ contained in it. As it was for Example 3, the reason is these two components have corresponding instantaneous frequencies with overlapping ranges as shown in Figure \ref{fig40c}.

\begin{figure}[H]
        \begin{subfigure}[b]{0.3\textwidth}
                \centering
                \includegraphics[width=\textwidth]{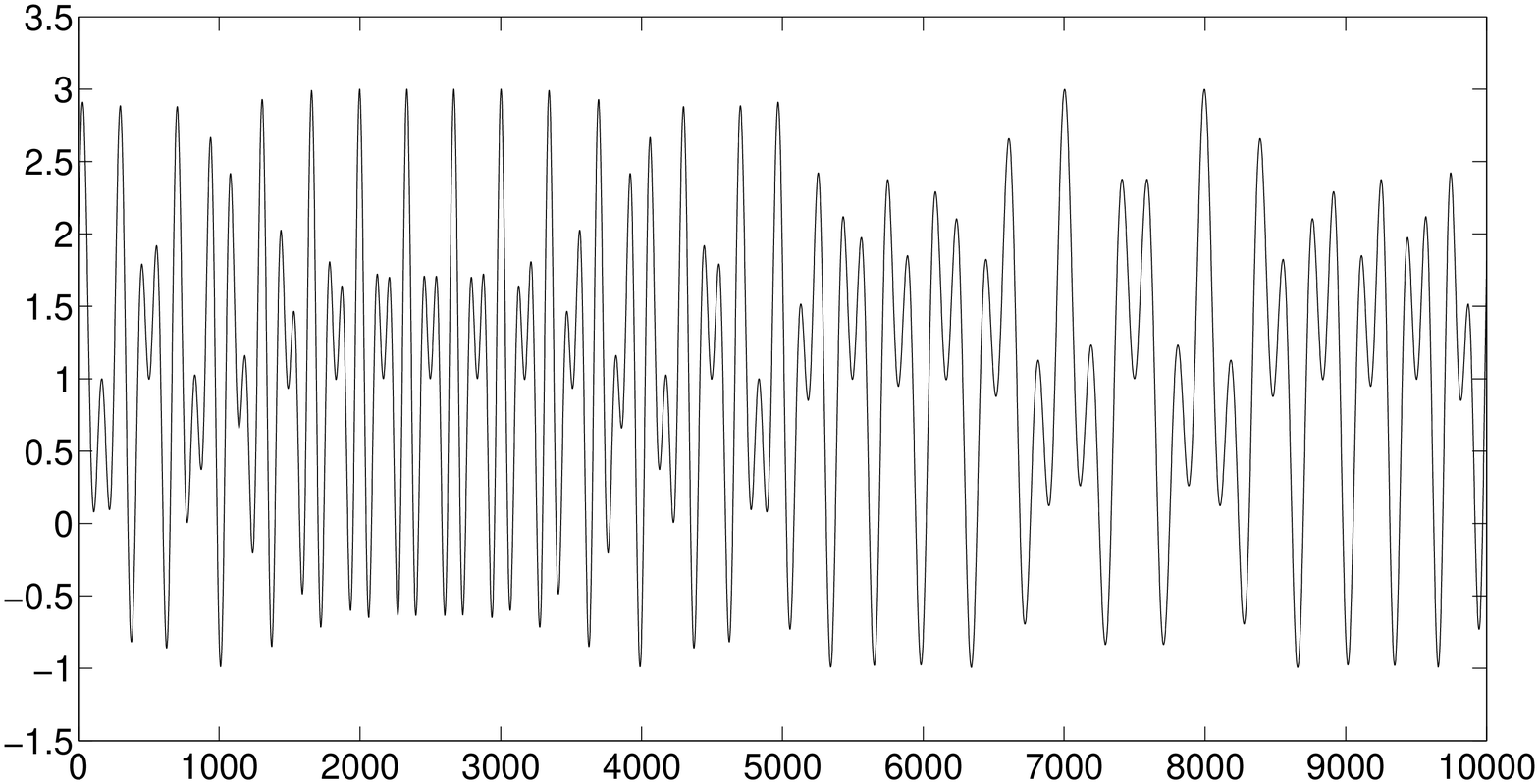}
                \caption{}
                \label{fig40a}
        \end{subfigure}
        ~ 
        \begin{subfigure}[b]{0.3\textwidth}
                \centering
                \includegraphics[width=\textwidth]{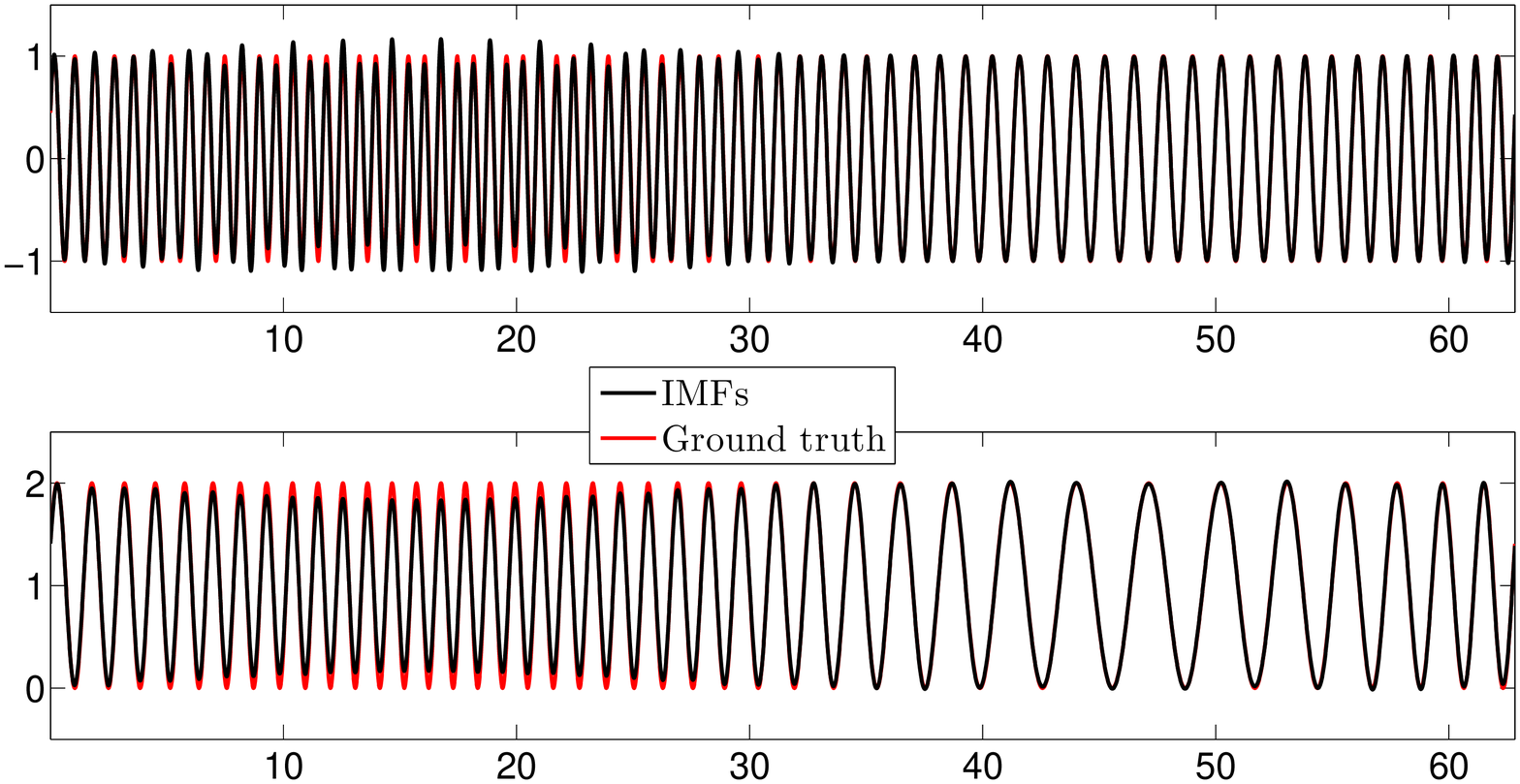}
                \caption{}
                \label{fig40b}
        \end{subfigure}
        ~ 
        \begin{subfigure}[b]{0.3\textwidth}
                \centering
                \includegraphics[width=\textwidth]{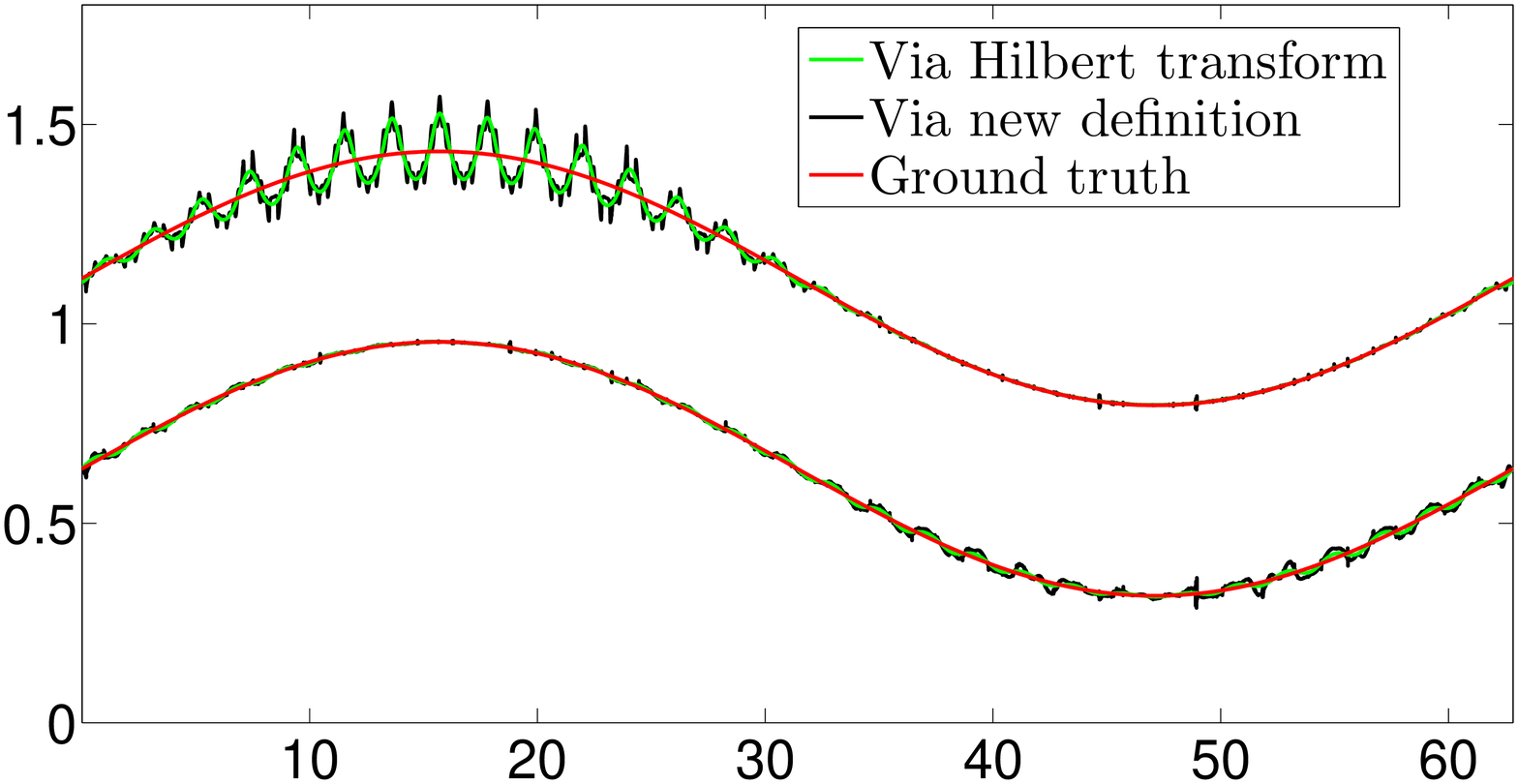}
                \caption{}\label{fig40c}
        \end{subfigure}%

         \caption{  (\subref{fig40a}) Signal $f_0(x)$ given in (\ref{eq:ALIFnoise_no}). (\subref{fig40b}) The decomposition produced by ALIF. (\subref{fig40c}) Instantaneous frequencies.}\label{fig40}
\end{figure}

To overcome the problem we make use of the ALIF code. The ALIF decomposition is plotted in \ref{fig40b}.

The method can handle high levels of noise as shown in Figure \ref{fig41} where SNR is around 0. As plotted in Figure \ref{fig41b} the first components in the decomposition contain purely noise, while the last two contain mainly the signal we want to recover.

\begin{figure}[H]
 \begin{center}
        \begin{subfigure}[b]{0.5\textwidth}
                \centering
                \includegraphics[width=\textwidth]{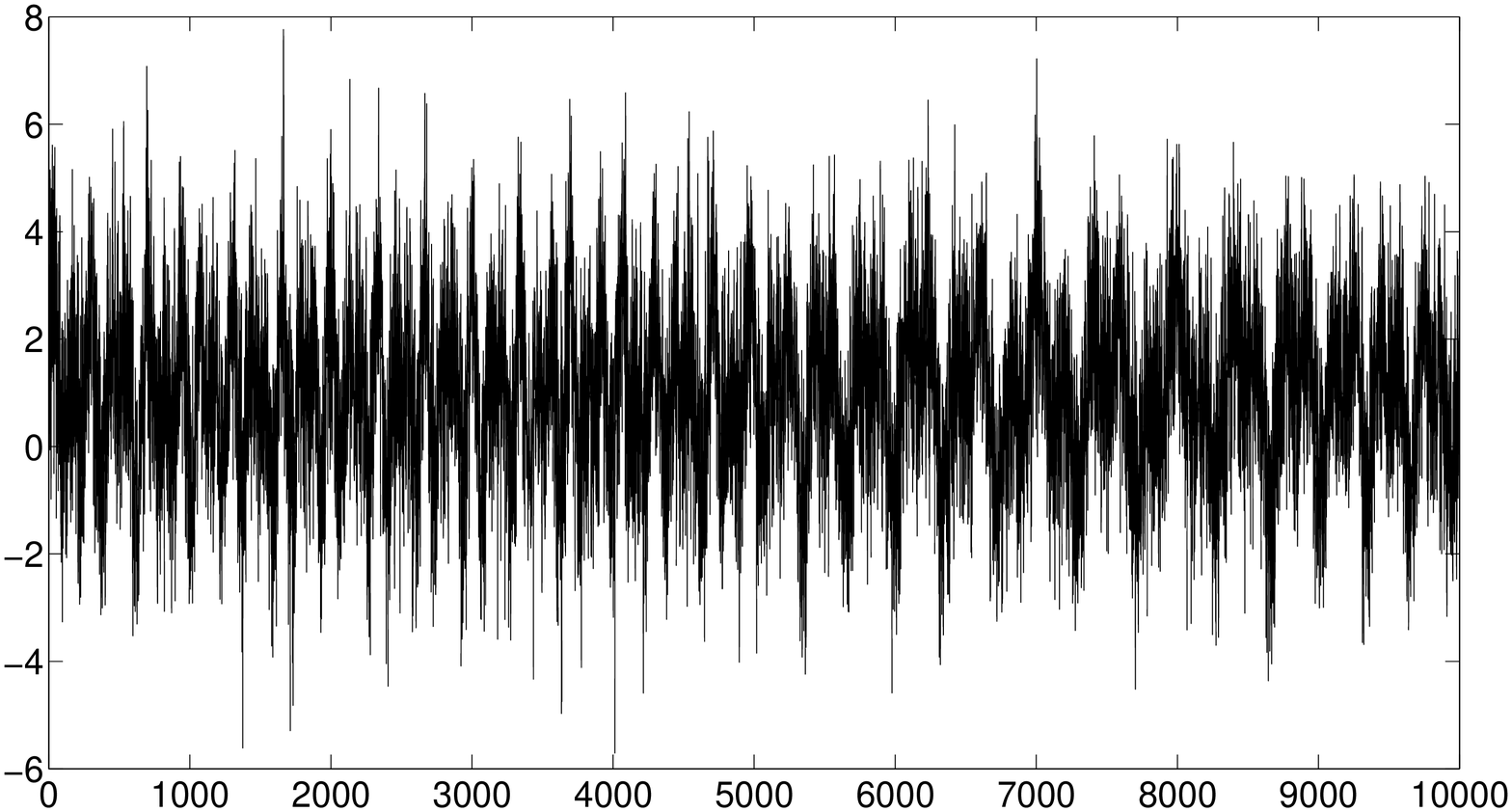}
                \caption{}
                \label{fig41a}
        \end{subfigure}%
        ~
        \begin{subfigure}[b]{0.5\textwidth}
                \centering
                \includegraphics[width=\textwidth]{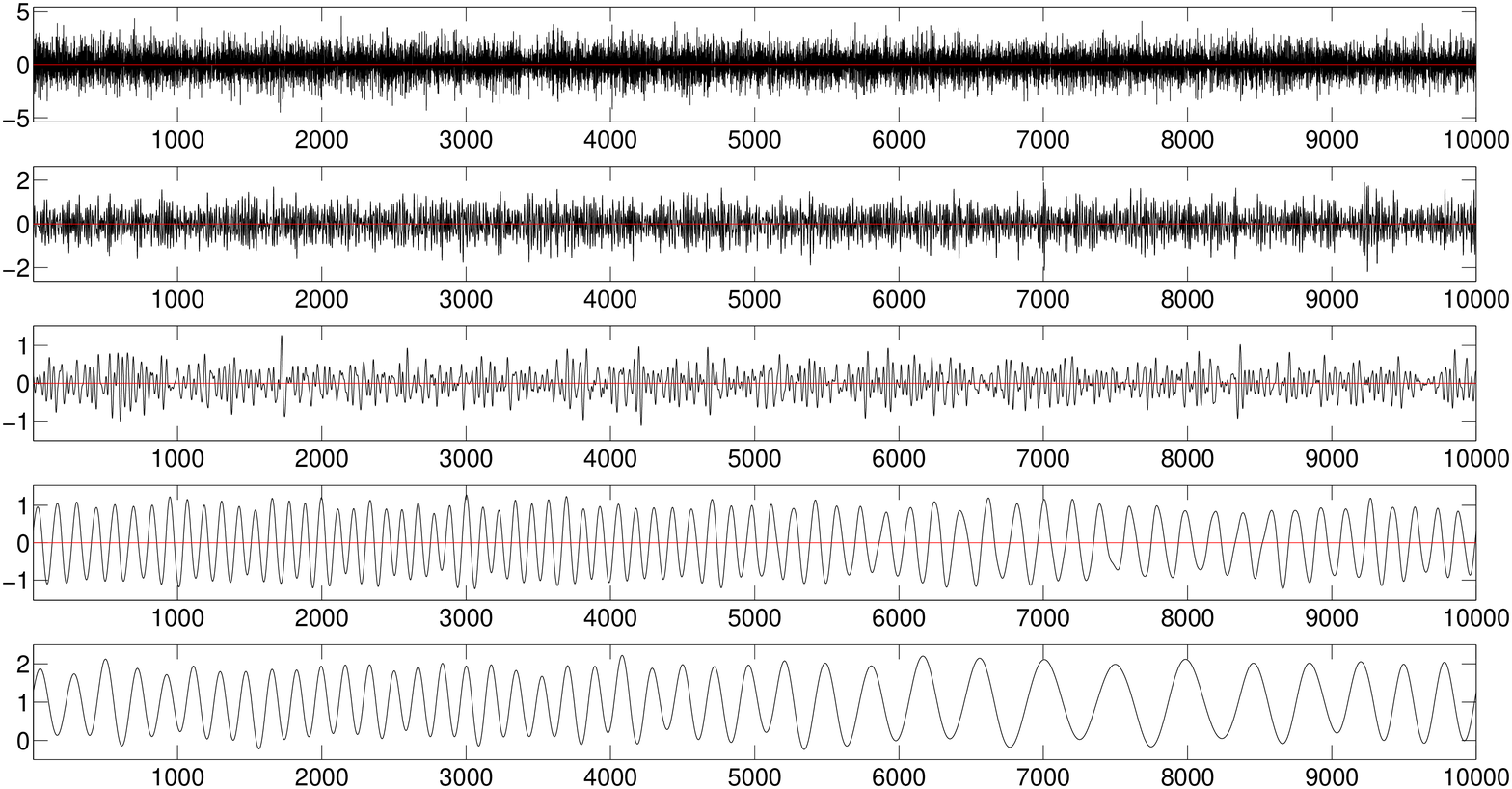}
                \caption{}
                \label{fig41b}
        \end{subfigure}
\end{center}
        \caption{ (\subref{fig41a}) Signal $f_1(x)$ with SNR around 0. (\subref{fig41b}) The decomposition produced by ALIF.  }\label{fig41}
\end{figure}

For higher levels of noise the method can still produce a meaningful decomposition. However the low frequency IMFs components start being corrupted by noise, ref. Figure \ref{fig42}. We note that also in this example ALIF does not use any a priori knowledge on the components embedded in the given signal.

\begin{figure}[H]
 \begin{center}
        \begin{subfigure}[b]{0.5\textwidth}
                \centering
                \includegraphics[width=\textwidth]{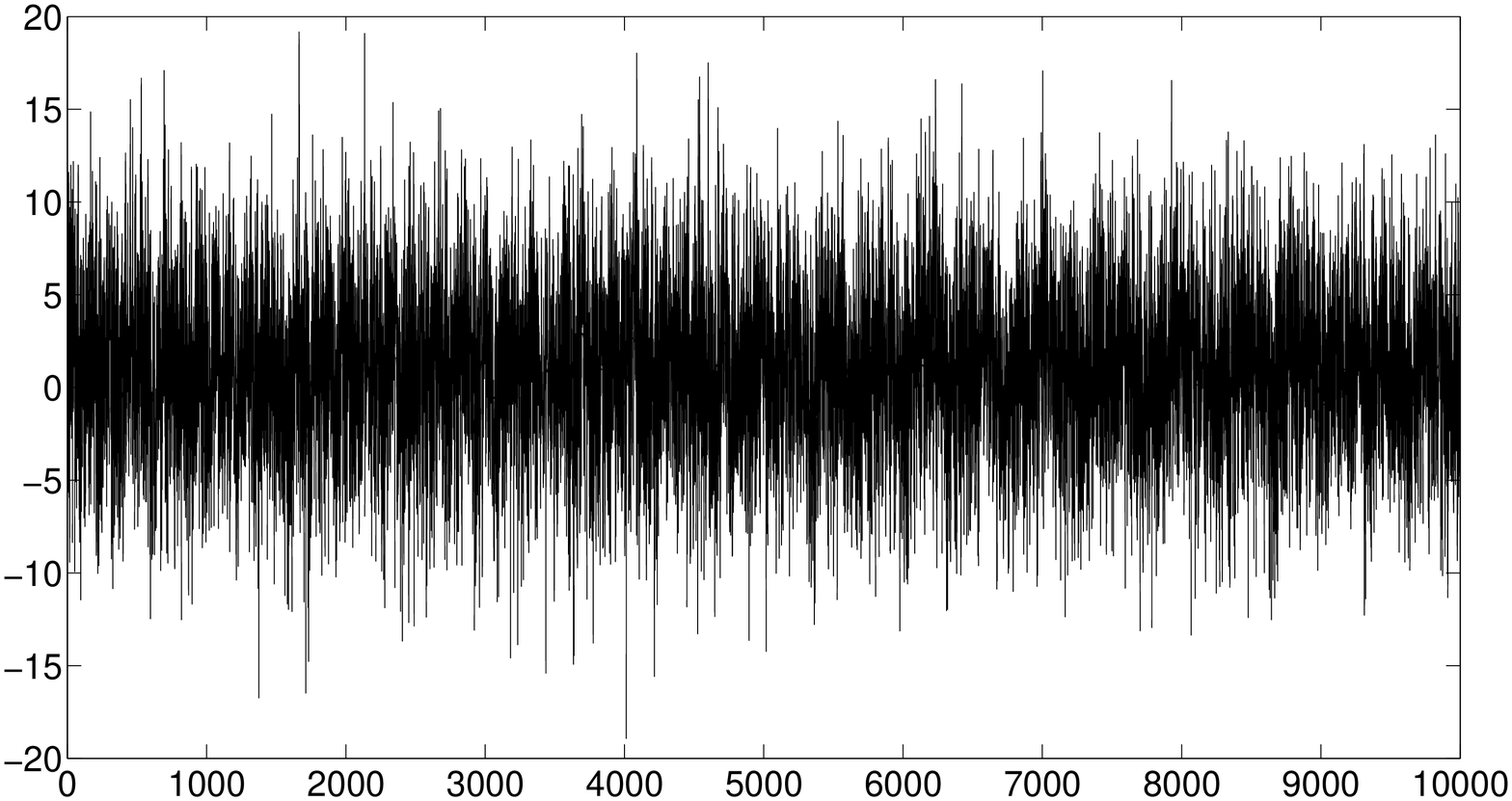}
                \caption{}
                \label{fig42a}
        \end{subfigure}%
        ~
        \begin{subfigure}[b]{0.5\textwidth}
                \centering
                \includegraphics[width=\textwidth]{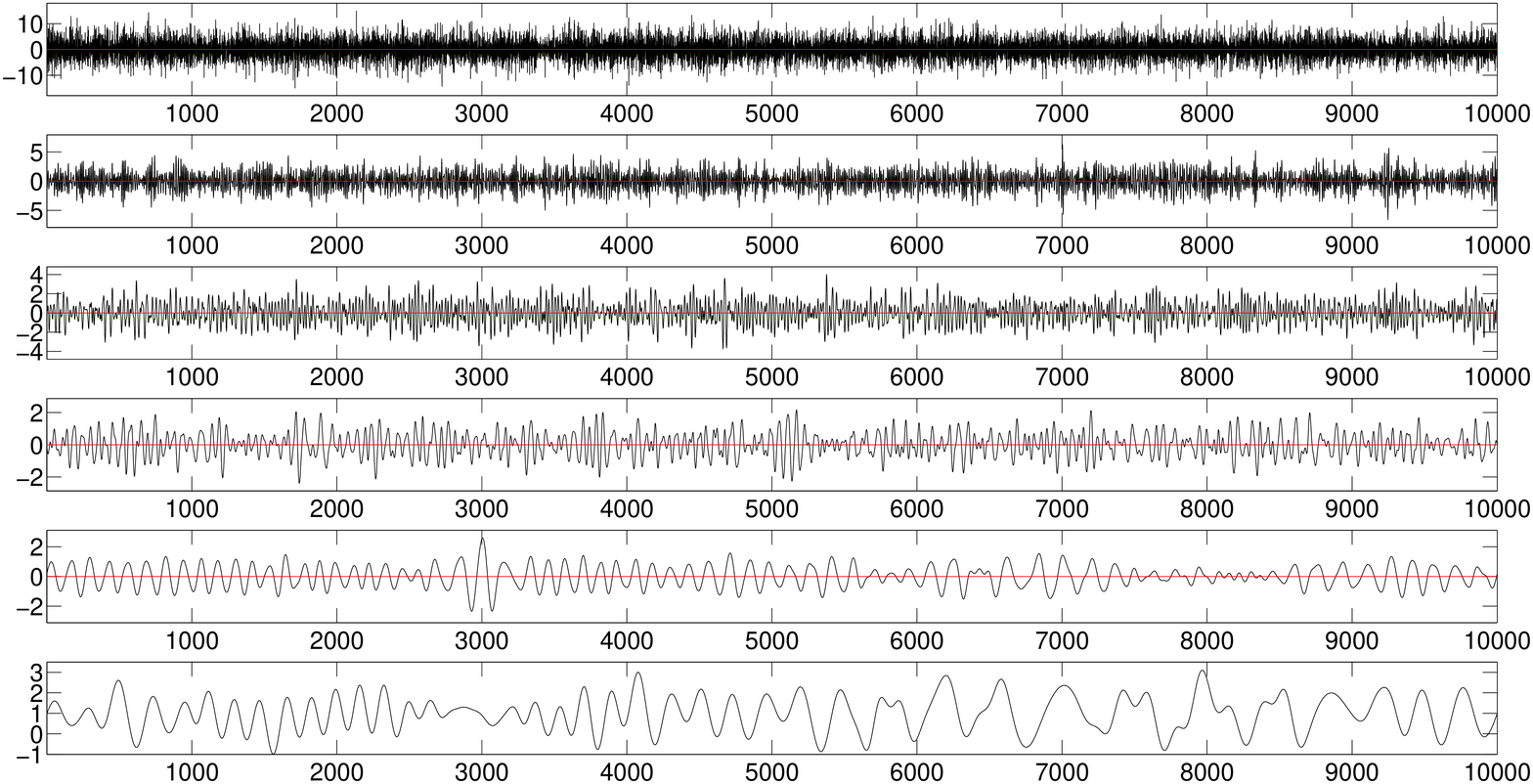}
                \caption{}
                \label{fig42b}
        \end{subfigure}
\end{center}
        \caption{ (\subref{fig42a}) Signal $f_2(x)$ with SNR around $-10$. (\subref{fig42b}) The decomposition produced by ALIF.  }\label{fig42}
\end{figure}

In Figures \ref{fig43} and \ref{fig44} we compare the last two components in the ALIF decompositions of $f_1$ and $f_2$ with the signals $g_1$ and $g_2+1$. These examples show the stability of this technique to white noise.

\begin{figure}[H]
 \begin{center}
                \includegraphics[width=0.7\textwidth]{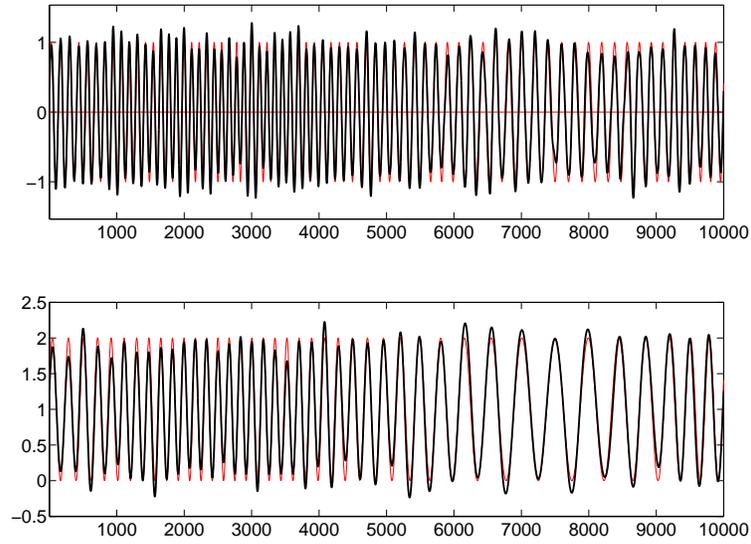}
                \caption{Comparison of the last two components in the ALIF decomposition of $f_1$, in solid black, with $g_1$ and $g_2+1$, in dotted red}\label{fig43}
 \end{center}
\end{figure}

\begin{figure}[H]
 \begin{center}
                \includegraphics[width=0.7\textwidth]{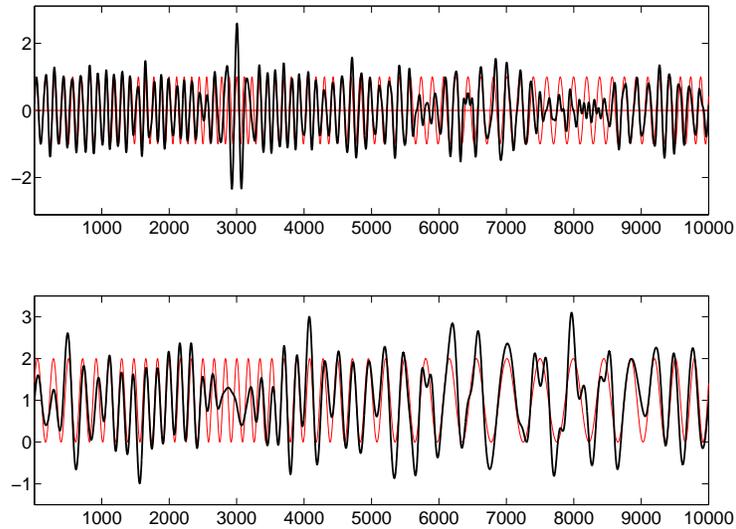}
                \caption{The last two components in the ALIF decomposition of $f_2$, in solid black, compared with $g_1$ and $g_2+1$, in dotted red}\label{fig44}
 \end{center}
\end{figure}

Furthermore these examples suggest that also the ALIF technique can identify with good accuracy signals belonging to the functional class $\mathcal C$, like signal $f_0(x)$, as well as signals that are clearly not in $\mathcal C$, like $f_1(x)$ with its SNR around 0 dB.

It remains an open problem to derive theoretical bounds on the expected accuracy of the ALIF decomposition of a given signal. We plan to study this problem in the next future.

\noindent \textbf{Example 7} We apply the IF algorithm to the deviation of the length of the day (LOD) data\footnote{LOD data set \url{http://hpiers.obspm.fr/eoppc/eop/eopc04/eopc04.62-now}} for $1000$ days from $01/01/1973$ to $09/28/1976$. This data is decomposed into $5$ components as shown in Figure \ref{fig10} where $4$ of them are IMFs and the last one is the trend. From the four IMFs, we can see very regular patterns: the half monthly change pattern, the monthly change pattern, the half yearly change pattern as well as the yearly change pattern.\\

 \begin{figure}[H]
        \begin{subfigure}[b]{0.5\textwidth}
                \centering
                \includegraphics[width=\textwidth]{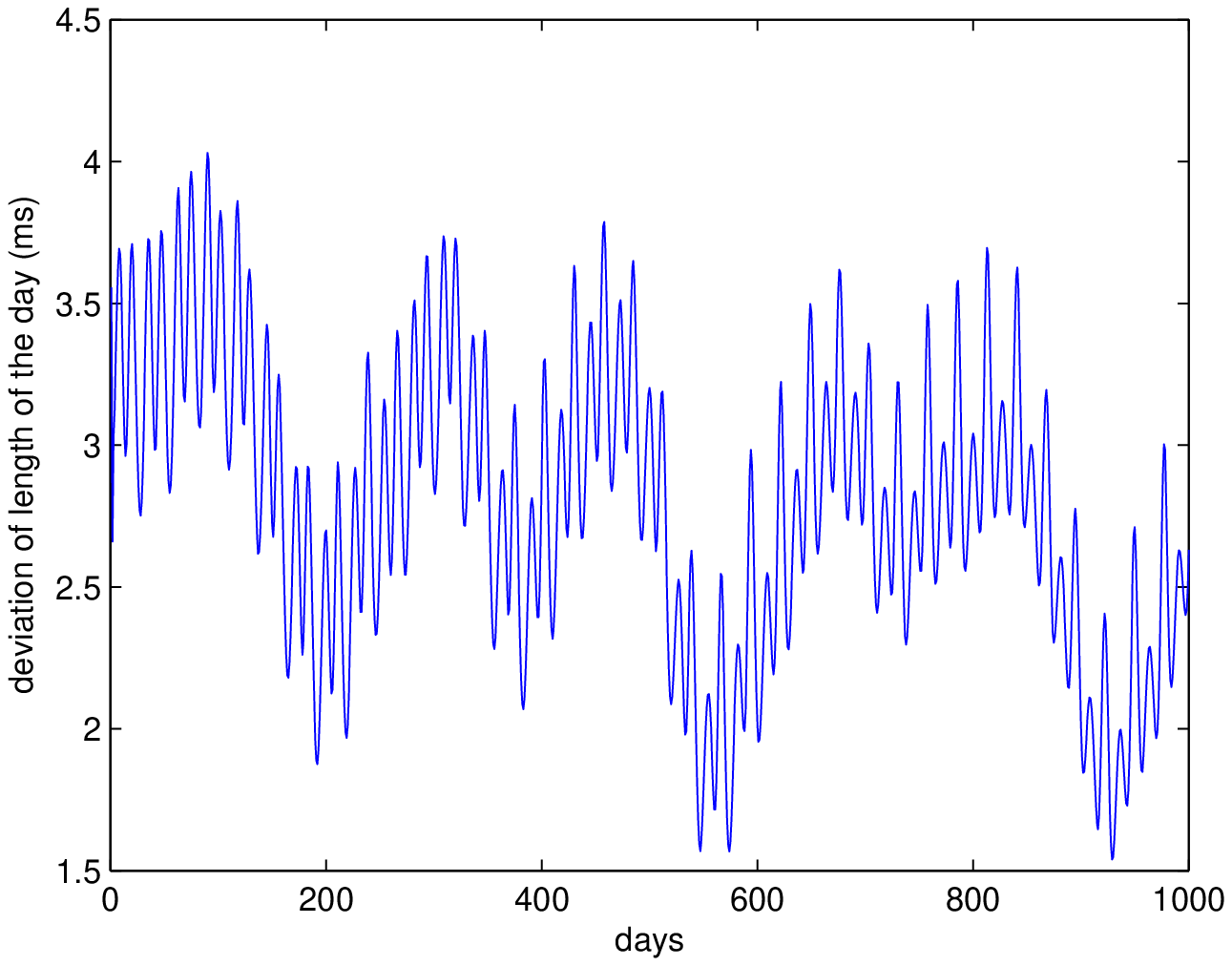}
                \caption{}
                \label{fig10-1}
        \end{subfigure}
        \begin{subfigure}[b]{0.5\textwidth}
        \includegraphics[width=\textwidth]{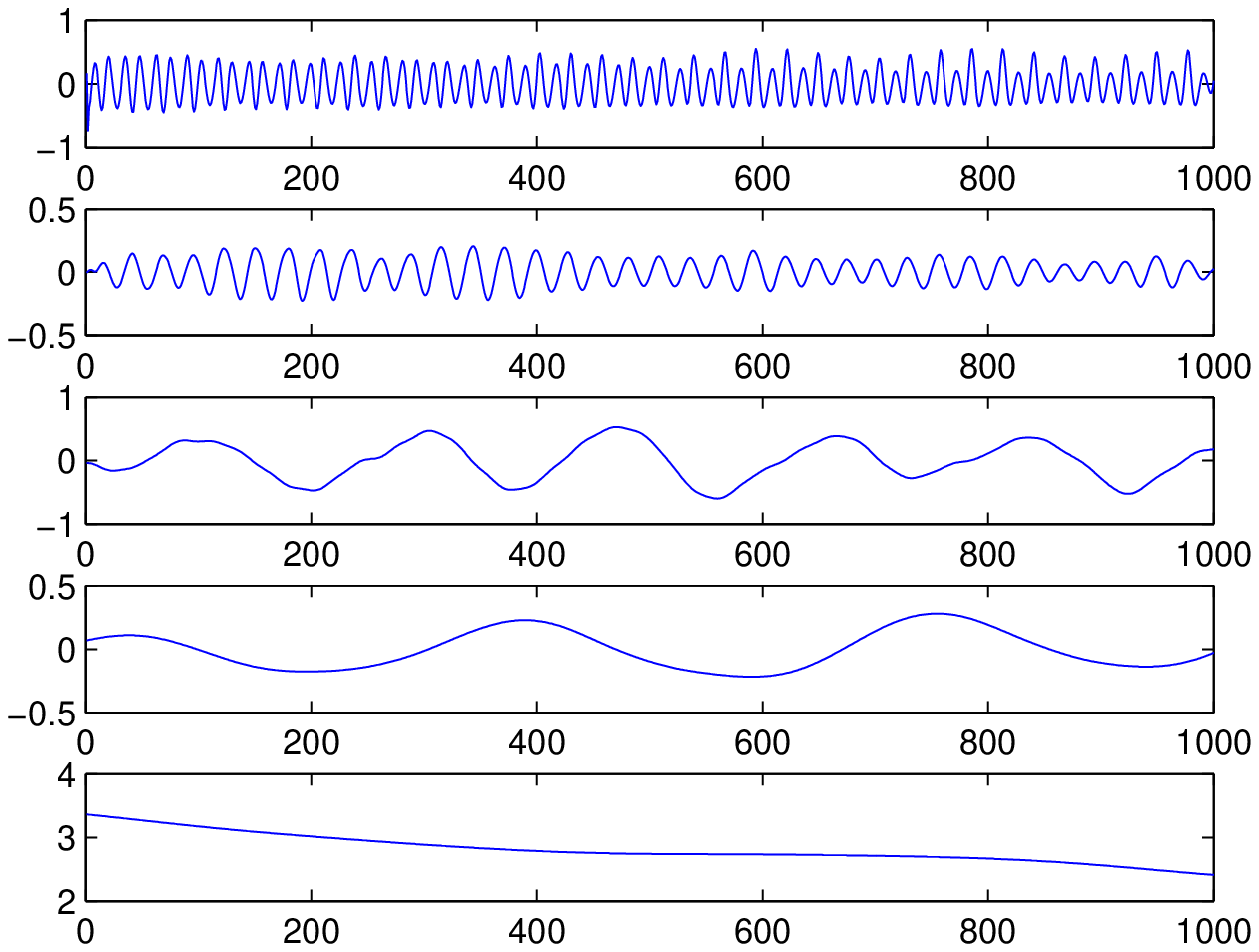}
        \caption{}
        \label{fig10-2}
        \end{subfigure}
        \caption{Length of day (LOD) signal and its decomposition. (\subref{fig10-1}) The LOD signal. (\subref{fig10-2}) The $5$ components in the IF decomposition.} \label{fig10}
\end{figure}

\noindent \textbf{Example 8} We test the IF algorithm on the water level data\footnote{Honshu earthquake and tsunami data set \url{
http://ntwc.arh.noaa.gov/previous.events/?p=03-11-11_Honshu}} recorded at Kawaihae, Hawaii, HI for 72 hours from March $11$, $2011$ to March $13$, $2011$ when the 2011 Honshu earthquake and tsunami occurred. The data is decomposed into several components as shown in Figure \ref{fig11} where the first two IMFs represent the transient signals associated with the impact of the tsunami. The last four components instead reveal the basic wave height with its regular patterns with periods of approximately 12, 24, 36 and 72 hours.

 \begin{figure}[H]
        \begin{subfigure}[b]{0.5\textwidth}
                \centering
                \includegraphics[width=\textwidth]{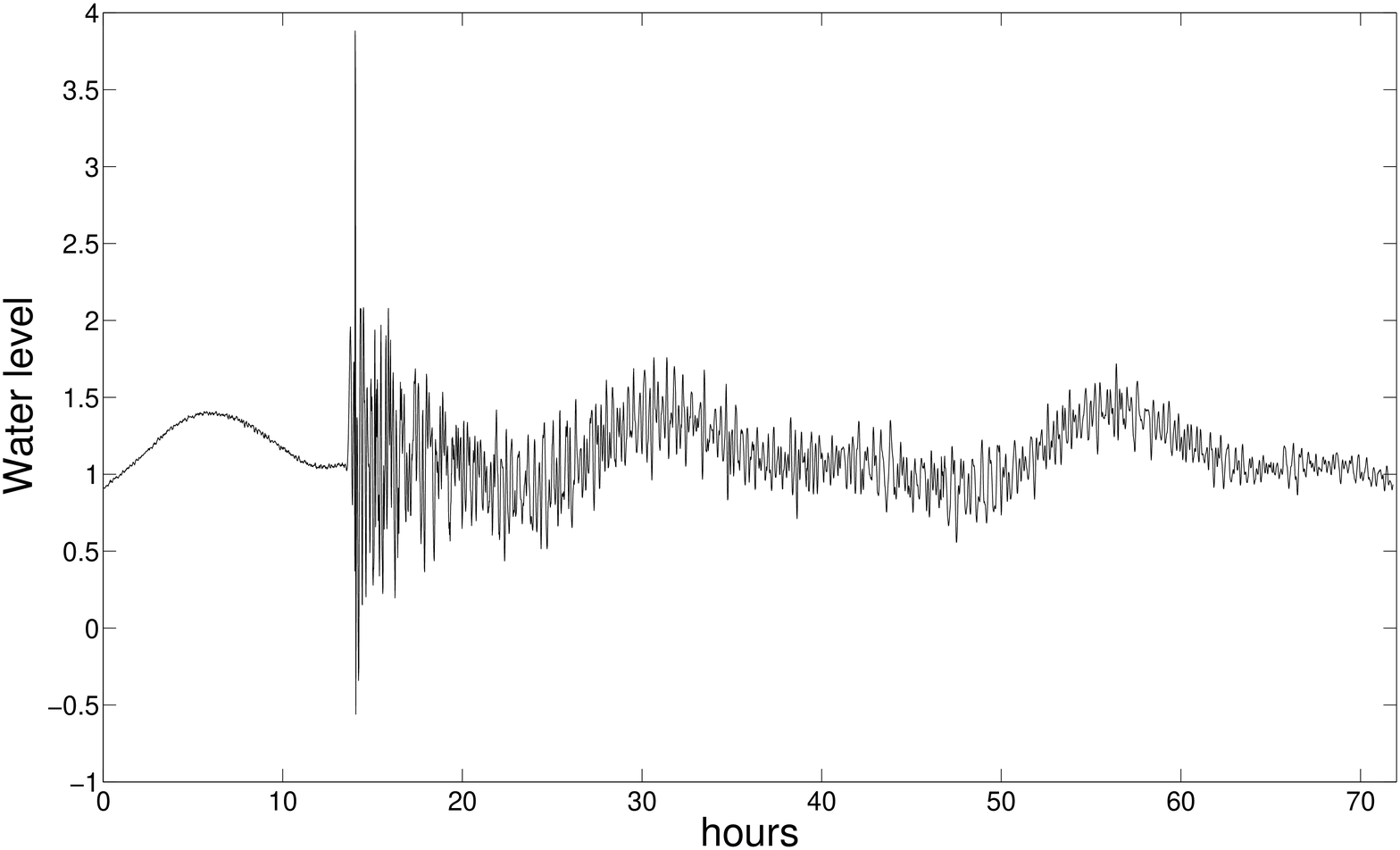}
                \caption{}
                \label{fig11-1}
        \end{subfigure}
        \begin{subfigure}[b]{0.5\textwidth}
        \includegraphics[width=\textwidth]{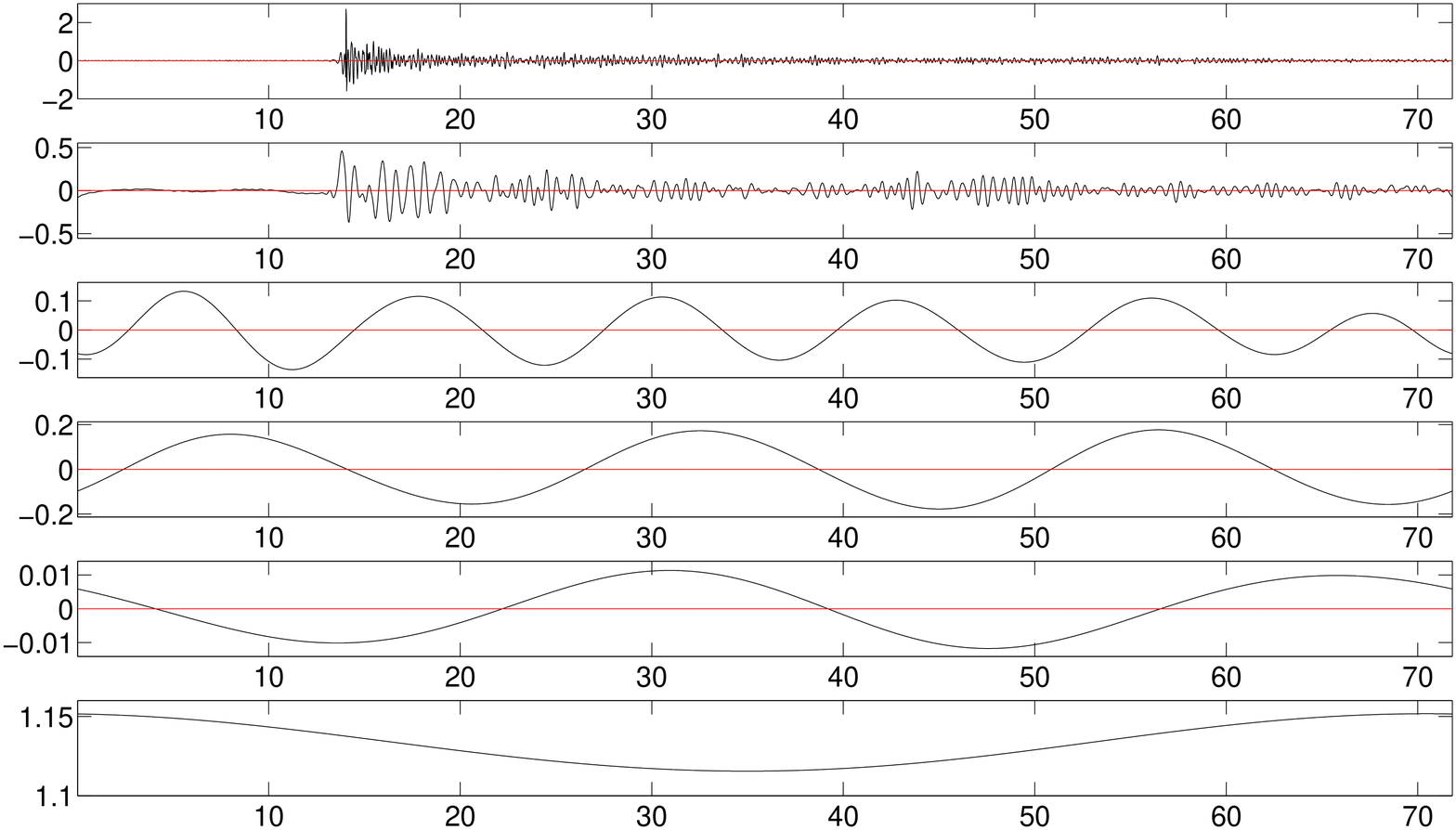}
        \caption{}
        \label{fig11-2}
        \end{subfigure}
        \caption{ (\subref{fig11-1}) The given wave height signal. (\subref{fig11-2}) The $6$ components in the IF decomposition. }\label{fig11}
\end{figure}

\noindent \textbf{Example 9}

We show the stability of the IF algorithm using this time real world signals. The two data sets shown
in Figure \ref{fig:Troposphere_Signal} are troposphere monthly mean temperature inferred from two research
groups from Jan 1979 to Dec 2004\footnote{Datasets available at \url{http://www.nsstc.uah.edu/data/msu/t2lt/uahncdc_lt_5.6.txt}
and \url{ftp://ftp.ssmi.com/msu/monthly_time_series/rss_monthly_msu_amsu_channel_tlt_anomalies_land_and_ocean_v03_3.txt}
}. We see that these two signals are quite close
each other, i.e. they have almost the same increasing and decreasing patterns except
the magnitudes are a slightly different. So we can regard one signal as a perturbation of
the other.

\begin{figure}[H]
                \centering
                \includegraphics[width=\textwidth]{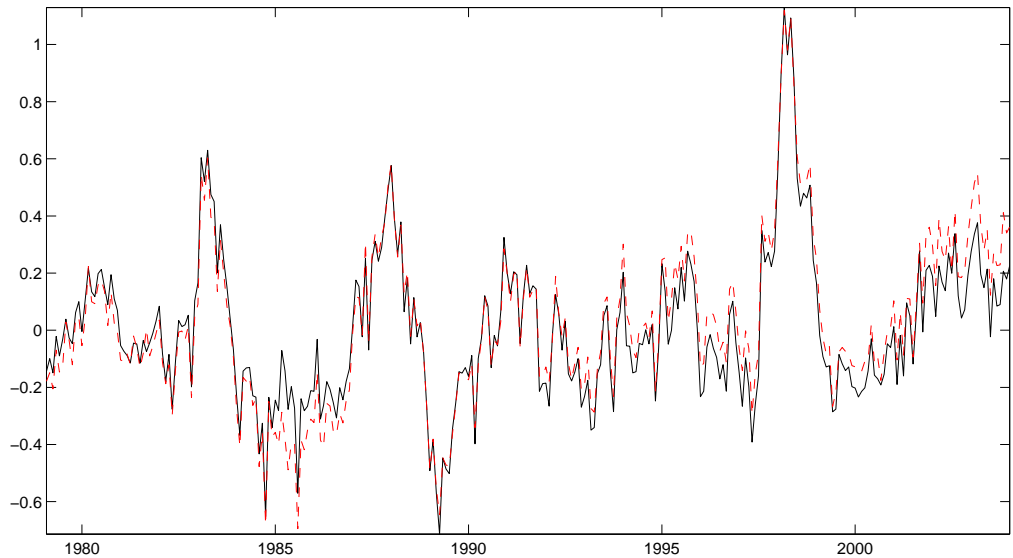}
        \caption{Troposphere monthly mean temperature inferred from two research groups
        from Jan 1979 to Dec 2004.}
        \label{fig:Troposphere_Signal}
        \end{figure}

        \begin{figure}[H]
        \includegraphics[width=\textwidth]{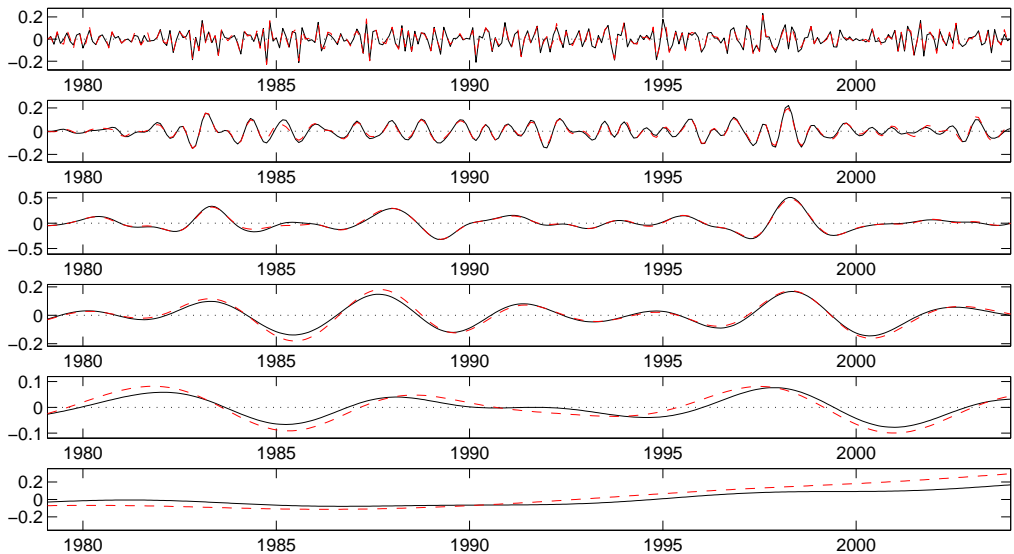}
        \caption{Decompositions of the two signals shown in Figure \ref{fig:Troposphere_Signal}. Corresponding components in the two decompositions are close each other.}
        \label{fig:Troposphere_IMFs}
\end{figure}

If we apply the IF algorithm to both signals we get
a decomposition for each of them. The results show that the IF technique
generates the same number of components for these two signals. We compare the
corresponding components in Figure \ref{fig:Troposphere_IMFs}. It is clear that each pair of corresponding
components are close to each other.

This particular example shows the stability of the IF method applied to real--life signals and, from the identifiability point of view \cite{Wu2014Ident}, suggests that also for real--life signals this technique is able to assign close--by decompositions to similar signals, even when the signals present fast accelerations and decelerations in the instantaneous frequency of its components, like the second IMF of the current example.

\section{Conclusion}

In this paper we first of all review the IF method, proving, in Theorem \ref{theo_1}, its inner loop convergence. In particular such theorem provides sufficient conditions on the filter shape which ensure the convergence of the IF method when applied to general non--stationary and non--periodical signals.

Secondly we propose the Adaptive Local Iterative Filtering (ALIF) algorithm with the purpose of designing a local, adaptive and stable iterative filtering method. The adaptivity of the algorithm is achieved by means of a filter length $l_n(x)$ who is changed pointwise following the behavior of the signal we want to decompose. The locality is guaranteed by the use of FP filters we designed based on a PDE model.

We prove the convergence of the ALIF method inner loop in Theorem \ref{theo_2}. Such theorem provides an a posteriori criterion for the convergence of the ALIF technique. It remains an open problem to find sufficient conditions on the filter and the non--uniform mask length $l_n(x)$ that ensure a priori the convergence of the inner loop of this method. We plan to work on this problem in the upcoming future.

All the numerical tests we ran suggest that under mild sufficient conditions on the filters the outer loops of both schemes converge. We plan to work on a rigorous proof of these convergences in the next future.

The stability of both IF and ALIF algorithm is shown by numerical examples in Section \ref{sec:Experiments}. We plan to study the stability properties of these techniques in the upcoming future.

Inspired by the numerical results we present in this paper we plan to analyze, from a theoretical standpoint, the identifiability problem of the decomposition of generic signals by means of the IF and ALIF methods.

In this work we present also new definitions of instantaneous phase and frequency which make use only of local properties of a given signal, allowing for a completely local time--frequency analysis.

Finally we point out that both IF and ALIF methods can be easily extended to handle higher dimensional signals as shown in \cite{Cicone2015MIF}. We plan to study the convergence of these extended techniques in the next future.

\section{Acknowledgments}

This work was supported by NSF Faculty Early Career Development (CAREER) Award DMS--0645266, DMS--1042998, DMS--1419027, and ONR Award N000141310408.

Antonio Cicone acknowledges support by National Group for Scientific Computation (GNCS -- INdAM) `Progetto giovani ricercatori 2014', and by Istituto Nazionale di Alta Matematica (INdAM) `INdAM Fellowships in Mathematics and/or Applications cofunded by Marie Curie Actions'. He wants also to thank Professor Patrick Flandrin for the enlightening discussions they had during his visit at the \'Ecole normale sup\'erieure de Lyon.

Furthermore the authors thank the anonymous referee for all the constructive and valuable comments that have contributed to the improvement of the present paper.

\small
\bibliography{ALIFpaper}

\begin{thebibliography}{10}

\bibitem{Wu2014Ident}
Cheng M.-Y. Chen Y.-C. and Wu~H.-T.
\newblock Non--parametric and adaptive modelling of dynamic periodicity and
  trend with heteroscedastic and dependent errors.
\newblock {\em Journal of the Royal Statistical Society Series B},
  76(3):651--682, 2014.

\bibitem{Cicone2015MIF}
A.~Cicone and H.~Zhou.
\newblock Multidimensional iterative filtering method for the decomposition of
  high--dimensional non--stationary signals.
\newblock {\em Preprint arXiv:1507.07173}, 2015.

\bibitem{cohen1995time}
L.~Cohen.
\newblock {\em Time-frequency analysis: theory and applications}.
\newblock Prentice-Hall, Inc., 1995.

\bibitem{daubechies2011synchrosqueezed}
I.~Daubechies, J.~Lu, and H.-T. Wu.
\newblock Synchrosqueezed wavelet transforms: An empirical mode
  decomposition--like tool.
\newblock {\em Applied and Computational Harmonic Analysis}, 30(2):243--261,
  2011.

\bibitem{dragomiretskiy2014variational}
K.~Dragomiretskiy and D.~Zosso.
\newblock Variational mode decomposition.
\newblock {\em IEEE transactions on signal processing}, 62(1-4):531--544, 2014.

\bibitem{el2010analysis}
S.~D. El~Hadji, R.~Alexandre, and A.-O. Boudraa.
\newblock Analysis of intrinsic mode functions: a pde approach.
\newblock {\em Signal Processing Letters, IEEE}, 17(4):398--401, 2010.

\bibitem{feldman2006time}
M.~Feldman.
\newblock Time-varying vibration decomposition and analysis based on the
  hilbert transform.
\newblock {\em Journal of Sound and Vibration}, 295(3):518--530, 2006.

\bibitem{gabor1946theory}
D.~Gabor.
\newblock Theory of communication. part 1: The analysis of information.
\newblock {\em Journal of the Institution of Electrical Engineers-Part III:
  Radio and Communication Engineering}, 93(26):429--441, 1946.

\bibitem{genton2007statistical}
M.~G. Genton and P.~Hall.
\newblock Statistical inference for evolving periodic functions.
\newblock {\em Journal of the Royal Statistical Society: Series B (Statistical
  Methodology)}, 69(4):643--657, 2007.

\bibitem{gilles2013empirical}
J.~Gilles.
\newblock Empirical wavelet transform.
\newblock {\em Signal Processing, IEEE Transactions on}, 61(16):3999--4010,
  2013.

\bibitem{grochenig2000foundations}
K.~Gr{\"o}chenig.
\newblock {\em Foundations of time-frequency analysis}.
\newblock Birkh{\"a}user Boston, 2000.

\bibitem{hahn1996hilbert}
S.~L. Hahn.
\newblock {\em Hilbert Transforms in Signal Processing}.
\newblock Artech House, 1996.

\bibitem{harten1987uniformly}
A.~Harten, B.~Engquist, S.~Osher, and S.~R. Chakravarthy.
\newblock Uniformly high order accurate essentially non-oscillatory schemes,
  iii.
\newblock {\em Journal of Computational Physics}, 71(2):231--303, 1987.

\bibitem{thomas2011adaptive}
T.Y. Hou and Z.~Shi.
\newblock Adaptive data analysis via sparse time-frequency representation.
\newblock {\em Adv. in Adap. Data Anal.}, 3(1):1--28, 2011.

\bibitem{hou2009variant}
T.Y. Hou, M.P. Yan, and Z.~Wu.
\newblock A variant of the emd method for multi--scale data.
\newblock {\em Adv. in Adap. Data Anal.}, 1(4):483--516, 2009.

\bibitem{huang2005hilbert}
N.~E. Huang and S.~S. Shen.
\newblock {\em Hilbert-Huang transform and its applications}, volume~5.
\newblock World Scientific, 2005.

\bibitem{huang2003confidence}
N.~E. Huang, M.-L.~C. Wu, S.~R. Long, S.~S.P. Shen, W.~Qu, P.~Gloersen, and
  K.~L. Fan.
\newblock A confidence limit for the empirical mode decomposition and hilbert
  spectral analysis.
\newblock {\em Proceedings of the Royal Society of London. Series A:
  Mathematical, Physical and Engineering Sciences}, 459(2037):2317--2345, 2003.

\bibitem{huang1999new}
N.E. Huang, Z.~Shen, and S.R. Long.
\newblock A new view of nonlinear water waves: The hilbert spectrum.
\newblock {\em Annual Review of Fluid Mechanics}, 31(1):417--457, 1999.

\bibitem{huang1998empirical}
N.E. Huang, Z.~Shen, S.R. Long, M.C. Wu, H.H. Shih, Q.~Zheng, N.C. Yen, C.C.
  Tung, and H.H. Liu.
\newblock The empirical mode decomposition and the hilbert spectrum for
  nonlinear and non-stationary time series analysis.
\newblock {\em Proceedings of the Royal Society of London. Series A:
  Mathematical, Physical and Engineering Sciences}, 454(1971):903, 1998.

\bibitem{huang2009instantaneous}
N.E. Huang, Z.~Wu, S.R. Long, K.C. Arnold, X.~Chen, and K.~Blank.
\newblock On instantaneous frequency.
\newblock {\em Adv. Adapt. Data Anal}, 1(2):177--229, 2009.

\bibitem{lin2009iterative}
L.~Lin, Y.~Wang, and H.~Zhou.
\newblock Iterative filtering as an alternative algorithm for empirical mode
  decomposition.
\newblock {\em Advances in Adaptive Data Analysis}, 1(4):543--560, 2009.

\bibitem{loughlin1997comments}
P.J. Loughlin and B.~Tacer.
\newblock Comments on the interpretation of instantaneous frequency.
\newblock {\em Signal Processing Letters, IEEE}, 4(5):123--125, 1997.

\bibitem{meignen2007new}
S.~Meignen and V.~Perrier.
\newblock A new formulation for empirical mode decomposition based on
  constrained optimization.
\newblock {\em Signal Processing Letters, IEEE}, 14(12):932--935, 2007.

\bibitem{pustelnik2012multicomponent}
N.~Pustelnik, P.~Borgnat, and P.~Flandrin.
\newblock A multicomponent proximal algorithm for empirical mode decomposition.
\newblock In {\em Signal Processing Conference (EUSIPCO), 2012 Proceedings of
  the 20th European}, pages 1880--1884. IEEE, 2012.

\bibitem{rilling2008one}
G.~Rilling and P.~Flandrin.
\newblock One or two frequencies? the empirical mode decomposition answers.
\newblock {\em Signal Processing, IEEE Transactions on}, 56(1):85--95, 2008.

\bibitem{rilling2009sampling}
G.~Rilling and P.~Flandrin.
\newblock Sampling effects on the empirical mode decomposition.
\newblock {\em Advances in Adaptive Data Analysis}, 1(01):43--59, 2009.

\bibitem{rilling2003empirical}
G.~Rilling, P.~Flandrin, P.~Goncalves, et~al.
\newblock On empirical mode decomposition and its algorithms.
\newblock In {\em IEEE-EURASIP workshop on nonlinear signal and image
  processing}, volume~3, pages 8--11. NSIP-03, Grado (I), 2003.

\bibitem{selesnick2011resonance}
I.~W. Selesnick.
\newblock Resonance-based signal decomposition: A new sparsity-enabled signal
  analysis method.
\newblock {\em Signal Processing}, 91(12):2793--2809, 2011.

\bibitem{sharpley2006analysis}
R.~C. Sharpley and V.~Vatchev.
\newblock Analysis of the intrinsic mode functions.
\newblock {\em Constructive Approximation}, 24(1):17--47, 2006.

\bibitem{shu1999high}
C.-W. Shu.
\newblock High order eno and weno schemes for computational fluid dynamics.
\newblock In {\em High-order methods for computational physics}, pages
  439--582. Springer, 1999.

\bibitem{wang2012iterative}
Y.~Wang, G.-W. Wei, and S.~Yang.
\newblock Iterative filtering decomposition based on local spectral evolution
  kernel.
\newblock {\em Journal of scientific computing}, 50(3):629--664, 2012.

\bibitem{wang2012mode}
Y.~Wang, G.-W. Wei, and S.~Yang.
\newblock Mode decomposition evolution equations.
\newblock {\em Journal of scientific computing}, 50(3):495--518, 2012.

\bibitem{wang2013convergence}
Y.~Wang and Z.~Zhou.
\newblock On the convergence of iterative filtering empirical mode
  decomposition.
\newblock In {\em Excursions in Harmonic Analysis, Volume 2}, pages 157--172.
  Springer, 2013.

\bibitem{wei1998instantaneous}
D.~Wei and A.C. Bovik.
\newblock On the instantaneous frequencies of multicomponent am-fm signals.
\newblock {\em Signal Processing Letters, IEEE}, 5(4):84--86, 1998.

\bibitem{wu2011one}
H.-T. Wu, P.~Flandrin, and I.~Daubechies.
\newblock One or two frequencies? the synchrosqueezing answers.
\newblock {\em Advances in Adaptive Data Analysis}, 3(01n02):29--39, 2011.

\bibitem{wu2004study}
Z.~Wu and N.E. Huang.
\newblock A study of the characteristics of white noise using the empirical
  mode decomposition method.
\newblock {\em Proceedings of the Royal Society of London. Series A:
  Mathematical, Physical and Engineering Sciences}, 460(2046):1597, 2004.

\bibitem{wu2005statistical}
Z.~Wu and N.E. Huang.
\newblock Statistical significance test of intrinsic mode functions.
\newblock {\em Hilbert--Huang Transform and Its Applications}, pages 107--127,
  2005.

\bibitem{wu2009ensemble}
Z.~Wu and N.E. Huang.
\newblock Ensemble empirical mode decomposition: A noise-assisted data analysis
  method.
\newblock {\em Adv. in Adap. Data Anal.}, 1(1):1--41, 2009.

\end{thebibliography}
\bibliographystyle{plain}
\end{document}